%% file: KR1.tex
\documentclass[11pt]{amsart}
\usepackage{amssymb,latexsym,amsmath,longtable,mathdots,comment,cite}
\usepackage[mathscr]{eucal}
\usepackage{bbm}
\usepackage{pmgraph}
\usepackage{color,lscape}
\usepackage[all]{xy}

\numberwithin{equation}{section}
\numberwithin{table}{section}
\numberwithin{figure}{section}

\input{macros.tex}
\input{thms.tex}

\def\BB66{MR0216035}
\def\CKpmhs{MR664326}
\def\CKSdeg{MR840721}
\def\DeligneII{MR0498551}
\def\FHW{MR2188135}
\def\GGK{MR2918237}
\def\Schmid{MR0382272}
\def\tNC{\sigma}

\theoremstyle{remark}
\newtheorem{RE}{Running Example}

\begin{document}

\title{Variations of Hodge structure and orbits in flag varieties}

\author[Kerr]{Matt Kerr}
\email{matkerr@math.wustl.edu}
\address{Department of Mathematics, Washington University in St. Louis, Campus Box 1146, One Brookings Drive, St. Louis, MO 63130-4899}
\thanks{Kerr is partially supported by NSF grants DMS-1068974, 1259024, and 1361147.}

\author[Robles]{Colleen Robles}
\email{robles@math.duke.edu}
\address{Mathematics Department, Duke University, Box 90320, Durham, NC  27708-0320} 
\thanks{Robles is partially supported by NSF grants DMS 02468621 and 1361120.  This work was undertaken while Robles was a member of the Institute for Advanced Study; she thanks the institute for a wonderful working environment and the Robert and Luisa Fernholz Foundation for financial support.}

\date{\today}

\begin{abstract}
Period domains, the classifying spaces for (pure, polarized) Hodge structures, and more generally Mumford-Tate domains, arise as open $G_\bR$--orbits in flag varieties $G/P$.
We investigate Hodge--theoretic aspects of the geometry and representation theory associated with these flag varieties.
In particular, we relate the Griffiths--Yukawa coupling to the variety of lines on $G/P$ (under a minimal homogeneous embedding), construct a large class of polarized $G_\bR$--orbits in $G/P$, and compute the associated Hodge--theoretic boundary components. 
An emphasis is placed throughout on adjoint flag varieties and the corresponding families of Hodge structures of levels two and four.
\end{abstract}


\maketitle

\setcounter{tocdepth}{1}
\let\oldtocsection=\tocsection
\let\oldtocsubsection=\tocsubsection
\let\oldtocsubsubsection=\tocsubsubsection
\renewcommand{\tocsection}[2]{\hspace{0em}\oldtocsection{#1}{#2}}
\renewcommand{\tocsubsection}[2]{\hspace{3em}\oldtocsubsection{#1}{#2}}
\renewcommand{\tocsubsubsection}[2]{\hspace{6em}\oldtocsubsubsection{#1}{#2}}

\section{Introduction}

This is the first in a sequence of papers (including \cite{KPR, KR2, SL2}) examining the relationship between 
\begin{itemize}
\item the Hodge theory of Mumford--Tate (MT) domains $D=G_\bR/R$ 
\end{itemize}
and 
\begin{itemize}
\item the projective geometry and $G_\bR$--orbit structure of their compact duals $\check D=G/P$.  
\end{itemize}
Its {\em motivating principle} is to use representation theory to classify potential differential and asymptotic features of variations of Hodge structure (VHS) with given ``symmetries'', with a view to applying the conclusions to decide what is possible or expected for geometric moduli.

The recent manifestations of this principle in the literature begin with the classification of Mumford--Tate groups \cite{MR2918237, Patrikis}, which clarifies the possible assortments of Hodge tensors in families, and hence what algebraic cycles are expected. On the differential front, a consideration of Schubert varieties in $G/P$ leads, via Kostant's theorem \cite{Kostant}, to an infinitesimal description of all maximal integrals of the infinitesimal period relation (IPR, a.k.a. Griffiths's transversality condition) \cite{MR3217458}. One thereby arrives at maximal expected dimensions for images of algebro-geometric period maps in any $\Gamma\backslash D$. In the asymptotic direction, a good understanding of abelian nilpotent cones in $\mathfrak{g}_{\mathbb{R}}$ allows for a classification of boundary components parametrizing all possible limiting mixed Hodge structures (LMHS) \cite{KU, KP2012}. This leads to precise predictions for the degeneration of varieties (and GIT-type boundary strata of relevant moduli spaces), as well as a means of attacking Torelli-type problems \cite{Usui, Jrthesis}.

In this paper, we point out another avenue through which representation theory influences period maps into $\Gamma \backslash D$, showing in $\S$\ref{S:sec5} how the projective geometry of the flag variety $\check{D}$ influences differential invariants of such VHS. Specifically, the variety of lines on $\check{D}$ (in its minimal homogeneous embedding) is closely related to the Griffiths-Yukawa coupling of a weight-2 variation $\mathcal{V}$. This relationship is connected to the second fundamental form of $\check{D}$ and the characteristic variety of $\mathcal{V}$, and should yield a useful criterion for distinguishing VHS (in the spirit of the disproof \cite{SXZ} of Dolgachev's conjecture).

We also introduce a tool for making the classification of boundary components effective.  Recently it has emerged \cite{MR3331177,KPR} that classifying the ``polarizable'' $G_{\mathbb{R}}$-orbits in the analytic boundary $\text{bd}(D)\subset \check{D}$ is both easier (than dealing with nilpotent cones of all dimensions) and corresponds more accurately to the boundary components (up to $G_{\mathbb{R}}$-conjugacy, this is one-to-one). We describe in $\S$\ref{S:cayleyconstr} a simple algorithm in terms of root-theoretic data for parametrizing many (and, in important cases, all) of these orbits, as well as the corresponding nilpotent cones and boundary components. While Hodge-theoretic in motivation, this construction produces an interesting family of smooth horizontal subvarieties of $\check{D}$ --- ``enhanced $\tSL_2$-orbits'', described in $\S$\ref{S:sl2+adj} --- which can be expanded in $H^*(\check{D},\mathbb{Z})$ as effective linear combinations of classes of (typically singular) horizontal Schubert varieties, thereby yielding smooth representatives.

\begin{center}
\line(1,0){150}
\end{center}

Principles aside, the motivating {\it impulse} for this work, and the reason for the peculiar pairing of topics, was our desire to give a uniform treatment of an important class of MT domains related to algebraic surfaces with $p_g > 1$.  This comes from recent work of the authors\footnote{forthcoming work of Green, Griffiths, Laza, and Robles on $H$-surfaces; and of da Silva, Kerr and Pearlstein on Katz surfaces} on {\em (a)} Horikawa surfaces \cite{Horikawa} and {\em (b)} certain 1-parameter families of elliptic surfaces defined by Katz \cite{MR3364748}; these correspond (resp.) to {\em (a)} the period domain $D_{(2,n,2)}$ for Hodge structures (HS) with $h^{2,0}=h^{0,2}=2$ and $h^{1,1}=n$, and {\em (b)} the MT domain for Hodge structures with Hodge numbers $(2,3,2)$ and MT group $G_2$.  These are examples of {\em irreducible, non-classical} (i.e. non-Hermitian, hence with nontrivial IPR) {\em domains for HS of weight two}. The remaining specimens of this type are the usual period domains $D_{(a,b,a)}$ ($a>2$) and their unitary ``cousins'' (for HS over an imaginary quadratic field), and three further ``exceptional'' examples parametrizing HS with Hodge numbers $(6,14,6)$ (resp. $(12,30,12)$,\footnote{over $\mathbb{Q}$, arising from an $\mathbb{F}$-Hodge structure of type $(6,15,6)$ ($\mathbb{F}$ imaginary quadratic)} $(12,32,12)$) and MT group $F_4$ (resp. $E_6$, $E_7$).  While geometric sources of the latter have yet to be found, techniques of Yun and others (which have already produced higher-weight examples for $G_2$, $E_7$, $E_8$ \cite{Yun}) should eventually provide them.

To see what distinguishes this class of domains for weight-two HS, it is helpful to recall what makes weight-three Calabi-Yau VHS (of type $(1,n,n,1)$) on a symplectic space $(V,Q)$ special. The period map $\Phi:\, \mathcal{S}\to \Gamma\backslash D$ lifts to $\tilde{\Phi} :\, \tilde{\mathcal{S}} \to D$ on the universal cover, and considering Hodge filtrands $\{F^3$, $F^2\}$ embeds $D \hookrightarrow \mathbb{P}V \times \mathbb{P}(\bigwedge^{n+1} V)$. Composing $\tilde{\Phi}$ with the projection $D\overset{\pi}{\to} \mathbb{P} V$ (forgetting $F^2$) yields an immersion $\tilde{\phi}:\, \tilde{\mathcal{S}} \to \mathbb{P}V$ integrating the natural contact structure on $\mathbb{P}V$. If $\Phi (\mathcal{S})$ is maximal (of dimension $n$), then $\tilde{\phi}$ is Legendrian and $\tilde{\Phi}$ can be completely reconstructed from the $1$-jet of $\tilde{\phi}$. For the Hermitian CY 3-VHS classified by Friedman and Laza \cite{FL}, the situation is even better: $\pi$ itself is an embedding!

Given the wide and fruitful attention received by these cases in the literature, it seems difficult to justify ignoring the (arguably more natural) case where the horizontal distribution on $\check{D}$ itself is a (homogeneous!) contact structure. Equivalently, $\check{D}$ is an adjoint variety, i.e. the orbit of a highest-weight line in $\mathbb{P}\mathfrak{g}$. Each of the adjoint varieties has a unique open $G_{\mathbb{R}}$-orbit $D$ with the structure of a MT domain for polarized HS of weight $2$ and $p_g >1$ (with MT group $G$). In fact, for $G$ of type $B$ or $D$, these are just the period domains $D_{(2,m,2)}$; for type $A$, they parametrize HS of type $(1,m,1)$ over an imaginary quadratic field (type $(2,2m,2)$ over $\mathbb{Q}$); and for types $E_6$, $E_7$, $F_4$, $G_2$ they recover the four examples described above.\footnote{The $E_8$ adjoint variety only parametrizes HS of weight 4.} The adjoint varieties turn out to admit a uniform treatment, with the method of $\S$\ref{S:cayleyconstr} describing all of their boundary components, and the codimension-one components closely related to the Hermitian C-Y VHS of \cite{FL}.

\begin{center}
\line(1,0){150}
\end{center}

We now turn to a more detailed discussion of the contents of this paper, beginning with some background.

\subsection*{Hodge-theoretic classifying spaces}

Mumford--Tate domains are classifying spaces for Hodge structures on a $\bQ$-vector space $V_{\bQ}$ with fixed Hodge tensors, including a polarizing form $Q\in (V^*)^{\otimes 2}$.  Recall how these arise: a (pure) Hodge structure polarized by $Q$, is given by a nonconstant homomorphism 
\[
  \varphi: S^1 \ \to \ \tAut(V_{\bR},Q)
\]
of real algebraic groups with $Q(v,\varphi(\sqrt{-1})\bar{v})>0$, for all $v\in V_\bC\setminus \{0\}$.   Let $G_{\bQ}$ denote the smallest $\bQ$--algebraic subgroup of $\tAut(V_{\bQ},Q)$ with $G_\bR\supseteq \varphi(S^1)$.  This is the \emph{Mumford--Tate group} (MT group) of $\varphi$.  The real form $G_\bR$ is known to be reductive \cite{MR654325}.  This group acts on $\varphi$ by conjugation, and the orbit
\[
   D \ \dfn \ G_\bR\cdot \varphi\ \cong \ G_\bR/R
\]
is a \emph{Mumford--Tate domain} (MT domain).  The stabilizer $R$ is compact, and its Lie algebra $\fr$ contains a compact Cartan subalgebra $\ft$ of $\fg_\bR$ \cite{\GGK}.   The MT domain is an analytic open subset of its compact dual, the projective homogeneous variety 
\[
  \check D \ \dfn \ G\cdot F_{\varphi}^\sb \ \cong G/P\,.
\]
Here, $G$ is the complexification of $G_\bR$, and the action of $G$ (and $G_\bR$) is by left translation on the associated Hodge flag $F^\sb_{\varphi} \in \check D$.

One way to construct  MT domains is to begin with a simple adjoint group $G_{\bQ}$, for which $G_\bR$ contains a compact maximal torus $T$.  If $\ttE\in\bi\ft=\bi\mathit{Lie}(T)$ is a grading element which is odd (respectively, even) on the noncompact (respectively, compact) roots, then $\varphi(z) = e^{2\log(z)\ttE}$ defines a weight zero Hodge structure (HS) on $\fg_{\bQ}=\mathit{Lie}(G_{\bQ})$ that is polarized by $-(\cdot,\cdot)$, where $(\cdot,\cdot)$ is the \emph{Killing form} on $\fg_\bQ$.  Moreover, $\varphi$ may lift to a HS on a representation $V_{\bQ}$.  In either case, a sufficiently general $G_\bR$-conjugate of this HS has  MT group $G_{\bQ}$.  In case $\varphi$ lifts, the Hodge decomposition $V_\bC = \op V^{p,q}$ is the $\ttE$--eigenspace decomposition; the grading element acts on $V^{p,q}$ by the scalar $(p-q)/2$.  Then $\varphi$ is identified with the Hodge flag $F^p = \op_{a\ge p} V^{a,b}$, the compact dual $\check D = G/P$ is identified with the $G$--flag variety containing $F^\sb$, and the MT domain $D = G_\bR/R$ with the open $G_\bR$--orbit of $F^\sb$ in $\check D$.  The Lie algebra of $P = \tStab_G(F^\sb)$ is $\fp = \fg^{\ge0} = \op_{p\ge0}\fg^p$, where $\fg^p = \fg^{p,-p}$.  The Lie algebra of $R$ is a real form of $\fg^0$; we shall write 
\[
  G^0 \ \dfn \ R_\bC \,.
\]
For the purposes of this introduction, we shall  assume that $D$ has a base point $o=\varphi$ defined over $\bQ$ (\cf~Definition \ref{def:MTCD}).  

\begin{RE}
A case in point, which we shall use as a running example throughout this Introduction, is the $\bQ$--form of the exceptional, rank two group $G_2$ constructed in $\S$\ref{S:Qform}.  The grading element $\ttS^2\in\ft_{\bQ(\sqrt{-1})}$ (dual to the second simple root) induces (i) a weight zero, $-(\cdot,\cdot)$ polarized Hodge decomposition of the complex Lie algebra 
\[
  \fg \ = \ \fg^{-2}\oplus \fg^{-1}\oplus\fg^0\oplus \fg^1\oplus\fg^2 \,,
  \quad \fg^{-p} = \fg^{-p,p} \,,
\]
with Hodge numbers $(1,4,4,4,1)$; and (ii) a weight two polarized Hodge decomposition 
\[
  V \ = \ V^{0,2}\oplus V^{1,1}\oplus V^{2,0} 
\]
on the standard, 7--dimensional representation with Hodge numbers $(2,3,2)$.  The compact dual is the 5--dimensional adjoint variety in $\bP\fg$.
\end{RE}

\subsection*{Variations of Hodge structure and Schubert varieties}

Recall that a variation of Hodge structure $\mathcal{V}$ over a complex manifold $S$ consists of a $\bQ$-local system $\mathbb{V}$, together with a filtration of $\mathbb{V}\otimes \cO_S$ by holomorphic subbundles $\mathcal{F}^\sb$ whose fibre-wise restrictions are Hodge flags, and which satisfy the infinitesimal period relation $\nabla(\mathcal{F}^\sb)\subset \mathcal{F}^{\sb -1}\otimes \Omega^1_S$ for the flat connection  $\nabla$ on $\mathcal{V}$ annihilating $\mathbb{V}$.  The asymptotics of $\mathcal{V}$ over a punctured disk $\Delta^*\subset S$ are captured by the (limiting) mixed Hodge structure $(V_\tlim,F^\sb_\tlim,W_{\sb})\dfn (\tilde{\mathbb{V}},\mathcal{F}^\sb,W_{\sb}(N))|_{s=0}$, where $N$ is the monodromy logarithm of $\mathbb{V}|_{\Delta^*}$ and $\tilde{\mathbb{V}}=\texp(-\log(s)/2\pi\bi)\mathbb{V}$. 

If $G_{\bQ}$ contains the  MT groups of the fibres of $\mathcal{V}$, the possible limiting MHS are classified by the boundary components $\tilde{B}(N)$ of \cite{KP2012}.  Moreover, $\mathcal{V}$ induces a \emph{period map} $\Phi:\, S\to \Gamma\backslash D$, with $\Gamma$ the monodromy group.  We may lift $\left.\Phi\right|_{\Delta^*}$ to the upper--half plane $\widetilde\Phi : \sH \to D$, and take the \emph{\naive~ limit} $\lim_{\tIm(\tau)\to\infty} \widetilde\Phi(\tau) \in \check D$, which lies in the analytic boundary $\tbd(D)\subset\check D$.  This boundary breaks into finitely many $G_\bR$--orbits, and those accessible by such Hodge--theoretic limits were termed \emph{polarizable} in \cite{MR3331177}.\footnote{This notion of a ``polarized'' orbit is distinct from J.~Wolf's in \cite[Definition 9.1] {MR0251246} where the orbit $G_\bR \cdot x\subset \check D$ is \emph{polarized} if it realizes the minimal homogeneous CR--structure on the homogeneous manifold $G_\bR/\tStab_{G_\bR}(x)$, \cf\cite[Remark 5.5]{MR2668874}.}  

The IPR forces the differential of the period map $\Phi$ to lie in the \emph{horizontal distribution} $T^1 \subseteq T\check{D}$.  At the point $\varphi$ holomorphic tangent space $T_\varphi\check D$ is identified with $\fg/\fp$; the horizontal subspace is $T^1_\varphi \simeq \fg^{-1,1}/\fp$.  A subvariety (or submanifold) $Y\in\check{D}$ is \emph{horizontal} if its tangent space $T_y Y$ is contained in $T^1_y$ at every smooth point $y\in Y$. The horizontal Schubert varieties $X\subset \check{D}$ encode a great deal of information about the IPR, cf. \cite{MR3217458}. In particular, the rank of the differential $d\Phi$ of a period map is bounded above by $\mathrm{max}\{ \dim(X) \, | \, X\subset \check{D}\;\text{horizontal Schubert}\}$.

\begin{RE}
The horizontal distribution $T^1$ is a contact $4$-plane distribution. Modulo the action of $G_2$, there is a unique Schubert variety $X_d \subset \check{D}$ of dimension $d$, for $0\leq d\leq 5$. The Schubert variety is horizontal if and only if $d\leq 2$. 
\end{RE}

A natural question about horizontal Schubert varieties is: when can they be (up to translation) the compact dual of a MT subdomain of $\check{D}$?  Any such Schubert variety is necessarily homogeneous and therefore smooth. The majority of Schubert varieties are singular. Moreover, a smooth Schubert variety $X\subset \check{D}$ need not be homogeneous. In \cite{KR2}, it is shown that a smooth \emph{horizontal} Schubert variety is necessarily a homogeneous embedded Hermitian symmetric space; furthermore, with an assumption on the ``$\mathbb{Q}$-types'' of $G_{\mathbb{Q}}$ and $o\in D$, these are all MT subdomains (Prop. \ref{Prop B}).

\begin{RE}
No translate of the 2-dimensional horizontal Schubert variety $X_2$ can be MT. This can be seen at once (and was first observed by the authors) by looking at its tangent space, which cannot be that of a reductive group orbit in $\check{D}$. In fact, it turns out that $X_2$ is singular, and thus not homogeneous. (The 1-dimensional horizontal Schubert variety $X_1$ is a $\mathbb{P}^1$, and therefore homogeneous.)
\end{RE}

\subsection*{Lines on $\check D$ and the Griffiths--Yukawa coupling}

When $\check D$ is the minimal homogeneous embedding $G/P \inj \bP V$, the minimal degree rational curves on $\check D$ are lines $\bP^1 \subset \bP V$ (degree one).  Let $\tilde \cC_o \subset \bP T_o\check D$ be the variety of tangent directions to lines passing though $o \in \check D$.  (We will often think of $\tilde\cC_o$ as the set of lines $o \in \bP^1 \subset \check D$.)  Of particular interest here is the set 
\[
  \cC_o \ \dfn \ \tilde \cC_o\,\cap\,\bP T^1
\] 
of directions horizontal with respect to the IPR:  If $P$ is a maximal parabolic, then the variety
\begin{equation} \label{E:cone}
  X \ \dfn \ \bigcup_{\bP^1 \in \cC_o} \bP^1
\end{equation}
swept out by these lines --- which is a cone over $\cC_o$ with vertex $o$ --- is a Schubert variety.  This is typically singular, and \emph{a priori} need not be horizontal. When $P$ is a maximal parabolic associated to a non-short root, we have $\cC_o = \tilde \cC_o$ and $X$ is in fact horizontal, \cf\cite[Proposition 2.11]{MR1966752} and Proposition \ref{P:Xo}.  When $\check D$ is an adjoint variety ($V=\fg$), $\cC_o$ realizes the homogeneous Legendrian varieties studied by J.M.~Landsberg and L.~Manivel \cite{MR1890196, MR2372722}.  

\begin{RE}
In our $G_2$ example, $X = X_2$ is the famous twisted cubic cone known to \'E.~Cartan, \cf $\S$\ref{S:sec5}.  
\end{RE}

In this paper we will discuss two Hodge--theoretic characterizations of $\cC_o$.  The first is in terms of the Griffiths--Yukawa coupling, a differential invariant of VHS's.  (The second is quite distinct, and will be discussed below.)  When the kernel of the Griffiths--Yukawa coupling is nonempty, it necessarily contains $\cC_o$ (Theorem \ref{T:GY}).  Moreover,  for certain compact duals (including the $A_n$, $E_6$, $E_7$, $E_8$, $F_4$, and $G_2$ adjoint varieties, \cf Lemma \ref{L:m=2} and \cite[\S5.3]{KR1long}) equality holds ($\cC_o$ is the kernel of the Griffiths--Yukawa coupling) for some Hodge representation of weight two (Theorem \ref{T:adj_m=2}).  

\subsection*{Polarized orbits}

The polarized orbits, and the corresponding limiting mixed Hodge structures, describe the asymptotics of a variation of Hodge structure on a MT domain.  More generally, one may take the \naive~limit of nilpotent orbit $\exp(zN)F^\sb$ on $\check D$.  In this paper our emphasis is on representation--theoretic descriptions of polarized orbits and mixed Hodge structures associated with nilpotent orbits, with the motivation that such descriptions may guide (i) the construction of geometric/motivic examples, and (ii) the application of Hodge theory to the study of nonclassical moduli.\footnote{See \cite{GGVan, GGR, GTAGS, GFRG3, GatF60, GMiami2015} for more on this line of thought.}

\begin{RE}
The boundary $\tbd(D)$ contains a unique real-codimension one $G_\bR$--orbit, which is polarizable.  The weight-graded pieces of the corresponding limiting mixed Hodge structures have $\dim\, Gr^W_3 V_\tlim = \dim\, Gr^W_1 V_\tlim = 2$ and $\dim\, Gr^W_2 V_\tlim=3$ (with Hodge numbers $(1,1,1)$).  Additional detail on this example may be found in \cite[$\S$6.1.3]{MR3331177}.
\end{RE}

Section \ref{S:cayleyconstr} gives a systematic construction of polarizable boundary strata in $\tbd(D)\subset\check D$ via distinguished ``$\bQ$-Matsuki''\footnote{The reason for this terminology is that the Cartan subalgebra $\fh\subset\fg$ corresponding to $o$ is assumed $\bQ$-rational (hence stable under complex conjugation) and stable under the Cartan involution; hence so are its Cayley transforms.} points on them and sets $\cB=\{\b_1,\ldots,\b_s\}\subset \fh^*$ of strongly orthogonal roots\footnote{The set $\cB$ is constrained by an additional condition imposed by the IPR.}, culminating in Theorem \ref{T:matsuki}.  

\begin{remark*}
In the case that $D$ is Hermitian symmetric, all the boundary orbits $\cO \subset \tbd(D)$ are polarizable.  From \cite[Theorem 3.2.1]{\FHW} it may be seen that they are all parameterized by the construction of $\S$\ref{S:cayleyconstr}, and the parameterization is essentially that given by the Harish--Chandra compactification of $D$.
\end{remark*}

\noindent While notationally dense, the construction is very natural, and we will briefly summarize it here.  With each root $\b_j \in \cB$ is associated a Cayley transform $\bc_{\b_j} \in G$, and we let $\bc_\cB$ denote the composite/product.  From the $\bQ$--Matsuki point $o$ and $\bc_\cB$ we obtain a $\bQ$--Matsuki point $o_\cB \in \check D$, with corresponding Hodge flag $F^\sb_\cB$, a ``weight filtration'' $W_\sb^\cB$, and nilpotent cone
\[
  \sigma_\cB \ \dfn \ 
  \{ t^1 N_1^{\cB} + \cdots + t^s N^{\cB}_s \ | \ t^i \ge 0\}
\]
with $N_j^\cB \in \fg_\bR$ and $N_j^\cB F^p_\cB \subset F^{p-1}_\cB$, for all $p$.  These objects have the properties that the nilpotent cone $\s_\cB$ underlies a multiple variable nilpotent orbit 
\[
  (z^1 , \ldots , z^s ) \ \mapsto \ 
  \texp( z^1N_1^\cB + \cdots z^s N_s^\cB) \cdot F^\sb_\cB \,,
\]
$W_\sb^\cB = W_\sb(\s_\cB)$, and $(F^\sb_\cB,W^\cB_\sb)$ is a ($\bQ$--split) $\bQ$--MHS on $\fg$.  In particular, the MHS $(F^\sb_\cB,W_{\sb}^{\cB})$ belongs to the boundary component $\tilde{B}(\sigma_\cB)$, and the orbit 
\[
  \cO_\cB \ \dfn \ G_\bR\cdot o_\cB
\]
is polarizable.\footnote{The notation of this introduction differs slightly from that of $\S$\ref{S:cayleyconstr}, where the terms $\bc_\cB$, $o_\cB$, $F^\sb_\cB$, $W_\sb^\cB$, $\s_\cB$, $\cO_\cB$, $\fg^{p,q}_\cB$  are denoted $\bc_s$, $o_s$, ${}_sF^\sb$, ${}_sW_\sb$, $\s_s$, $\cO_s$, ${}_s\fg^{p,q}$, respectively.}  

\subsubsection*{Boundary components}

The construction also a yields a semisimple element $\ttY_\cB \in \fg_\bR$ which grades the weight filtration.  The MT group of $(F^\sb_\cB,W_{\sb}^{\cB})$ is contained in 
\[
  G_\cB \ \dfn \ \tStab_G\{\ttY_\cB \} \,\cap\, \tStab_G\{N^\cB_1\}
  \,\cap\cdots\cap\,\tStab_G\{N^\cB_s\} \,.
\]
Defining\footnote{Here one should really think of $G_\cB$ as acting on the associated graded of the MHS $(F_\cB^\sb,W_{\sb}^{\cB})$.  Since this is a direct sum of Hodge structures, $D_\cB$ is a MT domain in the usual sense.} 
\[
  D_\cB \ \dfn \ G_{\cB,\bR}\cdot F^\sb_\cB 
  \ \subset \ G_\cB\cdot F_\cB^\sb \ \dfn \ \check D_\cB \,,
\]
the set 
\[
  B(\sigma_\cB)\ \dfn \ e^{\bC\,\sigma_\cB}\setminus \tilde{B}(\sigma_\cB)
\]
of nilpotent orbits fibres naturally over $D_\cB$. Passing to the quotient by the intersection $\Gamma$ of an arithmetic subgroup of $G_{\bQ}$ with $\tStab(\sigma_\cB)$, the surjection $\Gamma \backslash B(\sigma_\cB) \twoheadrightarrow\Gamma\backslash D_\cB$ has the structure of an iterated fibration by generalized intermediate Jacobians.

Part of the value of this construction is that it allows us to compute the ranks of the components $\fg_\cB^{p,q}$ in the Deligne splitting $\fg_\bC = \op \fg^{p,q}_\cB$ and the structure of the $\tilde{B}(\sigma_\cB)$.  For example, when $P$ is a maximal parabolic there is exactly one orbit $\cO_1\subset\tbd(D)$ of real codimension one (Proposition \ref{P:maxP}); it is of the form $\cO_\cB$ with $\cB = \{\b_1\}$, and the corresponding boundary components $B(\s_\cB) = B(N)$ are completely worked out in \cite[Appendix A]{KR1long}.  

\subsubsection*{Adjoint varieties and $\cC_o \simeq \check D_\cB$}
For the fundamental adjoint varieties, the Deligne bigrading takes a very intriguing form (Figure \ref{f:adj_bigr}): an element of the Weyl group of $(\fg,\fh)$ exchanges the Hodge and weight gradings, and the homogeneous Legendrian varieties $\cC_o$ above reappear in the guise of $\check D_\cB$.  Meanwhile, by Proposition \ref{P:adjB(N)}, the $\Gamma \backslash B(N)$ realize the intermediate Jacobian families associated to the Friedman--Laza \cite{FL} weight-three maximal Hermitian VHS's of Calabi-Yau type!  

\begin{RE}
The quotient $\Gamma \backslash B(N)$ is just an elliptic modular surface.  (The Friedman--Laza variation is the symmetric cube of a VHS of type $(1,1)$.)
\end{RE}

\subsection*{Enhanced horizontal $\tSL_2$--orbits and homology classes}

Another motivation for the construction of $\S$\ref{S:cayleyconstr} goes back to a basic question of Borel and Haefliger \cite{MR0149503} on smoothability of cohomology classes in $G/P$.  By \cite[Theorem  4.1]{MR3217458}, the invariant characteristic cohomology of $\check D$ is generated by classes of horizontal Schubert varieties (HSV).  Moreover, the homology class of any horizontal cycle in $\check D$ may be expressed as a linear combination of horizontal Schubert classes \cite[Theorem 4.7]{CC}.  Applying a ``horizontal twist'' to Borel and Haefliger's question, we ask: when does the class (or a multiple thereof) of a singular HSV, such as \eqref{E:cone}, admit a smooth, horizontal, algebraic representative?  One natural source of smooth subvarieties (in fact, MT subdomains) which are often horizontal, are the ``enhanced multivariable $\tSL_2$-orbits'' $X(\sigma_\cB)$ which arise as follows.  The strongly orthogonal roots $\cB$ above determine a collection $\{ \fsl_2^{\b_j} \ | \ \b_j \in \cB\}$ of commuting $\fsl_2$'s.  Letting $\tSL_2^\cB(\bC) \subset G_\bC$ be the connected Lie subgroup with Lie algebra $\fsl_2^\cB(\bC) = \fsl_2^{\b_1}(\bC) \times \cdots \times \fsl_2^{\b_s}(\bC)$, the enhanced $\tSL_2$--orbit is  
\[
  X(\sigma_\cB) \ \dfn \ 
  SL_2^{\cB}(\bC)\cdot G_{\cB,\bC}  \cdot o_\cB
  \ \subset \ \check D\,,
\]
\cf $\S$\ref{S:B+D}.  In the case that $\check D$ is a fundamental adjoint variety, the enhanced $\tSL_2$-orbit attached to $\cO_1$ is a horizontal and a cylinder on $\cC_o$; that is, 
\[
  X(\s_\cB) \ = \ X(N)\ \cong \ \cC_o\times\mathbb{P}^1 \,.
\]  
Recalling that $X$ is a cone over $\cC_o$, \cf\eqref{E:cone}, we conclude this article by testing the irresistible hypothesis that $[X(N)] = [X]$.  

\begin{RE}
We will show that $[X(N)]=2[X]$ in $\S$\ref{S:[X(N)]eg}.
\end{RE}

\noindent The case that $G = \tSO(2r+1,\bC)$ is worked out in \cite[Appendix B]{KR1long}.  

\subsection*{Future work}

Of the many directions one could pursue from here, among the more salient is the systematic computation of the cohomology classes of smooth horizontal subvarieties (and the horizontal $X(\sigma_\cB)$ in particular).  We will undertake this in a sequel (by a different method than that of $\S$\ref{S:[X(N)]eg}).  It is also important to generalize the construction of $\S$\ref{S:cayleyconstr} in order to parameterize all polarizable orbits; this work will appear in \cite{SL2}, and is applied in \cite{KPR} to tie together combinatorial structures related to nilpotent cones and boundary orbits.

Finally, one reason for classifying Hodge--theoretic phenomena is to predict algebro-geometric ones.  Turning one last time to the running example, suppose we have a 2-parameter family of algebraic surfaces $\pi:\, \mathcal{X}\to S$ for which (a subquotient of) $R^2 \pi_* \bQ$ underlies a maximal VHS with Hodge numbers $(2,3,2)$ and MT group $G_2$. Then one can ask what sort of degeneration produces a point in the ``elliptic modular surface'' boundary component $\Gamma\backslash B(N)$, and what sort of geometry produces a subvariation whose period map image lies in the ``cubic cone'' Griffiths--Yukawa kernel.  Now since one expects such 2-parameter VHS to arise from certain elliptic surfaces (with internal fibration over an elliptic curve, cf. \cite{MR3364748,Jrthesis}), neither question looks too difficult.  With some optimism, one might imagine asking the analogous questions for all the fundamental adjoint varieties.

\subsection*{Acknowledgements}

The authors thank P.~Griffiths and J.M.~Landsberg for helpful discussions.

\section{Preliminaries}

In this section we set notation and review the necessary background material from Hodge theory and representation theory.

\subsection{Flag varieties} \label{S:rhv}

Let $G$ be a connected, complex semisimple Lie group, and let $P \subset G$ be a parabolic subgroup.  The homogeneous manifold $G/P$ admits the structure of a rational homogeneous variety as follows.  Fix a choice of \emph{Cartan} and \emph{Borel subgroups} $H \subset B \subset P$.  Let $\fh \subset \fb \subset\fp\subset\fg$ be the associated Lie algebras.  The choice of Cartan determines a set of \emph{roots} $\Delta = \Delta(\fg,\fh) \subset \fh^*$.  Given a root $\a\in \Delta$, let $\fg^\a \subset \fg$ denote the \emph{root space}.  Given a subspace $\fs \subset \fg$, let 
$$
  \Delta(\fs) \ \dfn \ \{ \a\in\Delta \ | \ \fg^\a \subset \fs \} \,.
$$ 
The choice of Borel determines \emph{positive roots} $\Delta^+ = \Delta(\fb) = \{ \a\in\Delta \ | \ \fg^\a \subset \fb \}$.  Let $\sS = \{\a_1,\ldots,\a_r\}$ denote the \emph{simple roots}, and set 
\begin{equation} \label{E:I}
  I \ = \ I(\fp) \ \dfn \ \{ i \ | \ \fg^{-\a_i} \not\subset \fp \}\,.
\end{equation}
Note that 
$$
  I(\fb) \ = \ \{1,\ldots,r\} \,,
$$
and $I = \{\tti\}$ consists of single element if and only if $\fp$ is a \emph{maximal parabolic}.

Let $\{ \w_1,\ldots,\w_r\}$ denote the \emph{fundamental weights} of $\fg$ and let $V$ be the irreducible $\fg$--representation of highest weight 
\begin{equation} \label{E:mu}
  \m \ = \ \m_I \ \dfn \ \sum_{i \in I} \w_i \,.
\end{equation}
Assume that the representation $\fg \to \tEnd(V)$ `integrates' to a representation $G \to \tAut(V)$ of $G$.  (This is always the case if $G$ is simply connected.)  Let $o \in \bP V$ be the \emph{highest weight line} in $V$.  The $G$--orbit $G \cdot o \subset \bP V$ is the \emph{minimal homogeneous embedding} of $G/P$ as a rational homogeneous variety.

\begin{remark}[Non-minimal embeddings] \label{R:nonminemb}
More generally, suppose that $V$ is the irreducible $G$--representation of highest weight $\tilde \m = \sum_{i \in I} a^i\,\w_i$ with $0 < a^i \in \bZ$.  Again, the $G$--orbit of the highest weight line $o \in \bP V$ is a homogeneous embedding of $G/P$.  We write $G/P \inj \bP V$.  The embedding is minimal if and only if $a^i=1$ for all $i \in I$.  For example, the Veronese re-embedding $v_d(\bP^n) \subset \bP\,\tSym^d\bC^{n+1}$ of $\bP^n$ is if minimal if and only if $d = 1$.  (Here $V = \tSym^d\bC^{n+1}$ has highest weight $d\w_1$.)
\end{remark}

The rational homogeneous variety $G/P$ is sometimes indicated by circling the nodes of the Dynkin diagram (Appendix \ref{S:dynkin}) corresponding to the index set $I(\fp)$.

\subsection{Flag domains and Mumford--Tate domains} \label{S:mtd}

Let $G_\bR$ be a (connected) real form of $G$.  There are only finitely many $G_\bR$--orbits on $G/P$.  An open $G_\bR$--orbit
\[
  D \ = \ G_\bR/R
\]
is a \emph{flag domain}.  The stabilizer $R \subset G_\bR$ is the centralizer of a torus $T' \subset G_\bR$ \cite[Corollary 2.2.3]{\FHW}.    The flag variety
\[
  \check D \ \dfn \ G/P
\]
is the \emph{compact dual} of the flag domain.

We will be interested in the case that $D$ admits the structure of a MT domain.  MT domains are generalizations of period domains, and as such arise as the classification spaces of polarized Hodge structures (possibly with additional structure); see \cite{\GGK} for a thorough treatment.   When $D$ admits the structure of a MT domain the stabilizer $R$ is compact, and as a consequence there exists a compact maximal torus $T \subset G_\bR$ such that $T' \subset T \subset R$ \cite{\GGK}.  We will assume this to be the case throughout.  

In the MT domain case, $G$ also has an underlying $\bQ$-algebraic group $G_{\bQ}$, whose groups of real and complex points recover $G_\bR$ and $G$, respectively.  We will assume this only where relevant.  In a few places (mostly limited to $\S$\ref{S:cayleyconstr}), we shall make the stronger assumption that $G_{\bQ}$ is a MT--Chevalley group (Definition \ref{def:MTC}). 

\subsection{Grading elements} \label{S:grelem}

Let $\{ \ttS^1 , \ldots , \ttS^r\}$ be the basis of $\fh$ dual to the simple roots.  A \emph{grading element} is any member of $\tspan_\bZ\{ \ttS^1,\ldots,\ttS^r\}$; these are precisely the elements $\ttT\in \fh$ of the Cartan subalgebra with the property that $\a(\ttT) \in \bZ$ for all roots $\a\in \Delta$.  Since $\ttT$ is semisimple (as an element of $\fh$), the Lie algebra $\fg$ admits a direct sum decomposition
\begin{subequations} \label{SE:grading}
\begin{equation}
  \fg \ = \ \bigoplus_{\ell\in\bZ} \fg^\ell
\end{equation}
into $\ttT$--eigenspaces
\begin{equation} \label{E:grT}
  \fg^\ell \ \dfn \ \{ \xi \in \fg \ | \ [\ttT,\xi] = \ell \xi \} \,.
\end{equation}
In terms of root spaces, we have
\begin{equation} \label{E:gr1} \renewcommand{\arraystretch}{1.3}
\begin{array}{rcl}
  \displaystyle \fg^\ell & = & 
  \displaystyle \bigoplus_{\a(\ttT)=\ell} \fg^\a \,,\quad \ell \not=0 \,,\\
  \displaystyle \fg^0 & = & 
  \displaystyle \fh \ \op \ \bigoplus_{\a(\ttT)=0} \fg^\a \,.
\end{array}
\end{equation}
\end{subequations}
The $\ttT$--eigenspace decomposition is a \emph{graded Lie algebra decomposition} in the sense that 
\begin{equation}\label{E:gr}
  \left[ \fg^\ell  , \fg^m \right] \ \subset \ \fg^{\ell+m} \,,
\end{equation}
a straightforward consequence of the Jacobi identity.  It follows that $\fg^0$ is a Lie subalgebra of $\fg$ (in fact, reductive), and each $\fg^\ell$ is a $\fg^0$--module.  The Lie algebra
\begin{equation} \label{E:p}
  \fp \,=\, \fp_\ttT \ = \ \fg^0 \,\op \,\fg^+ 
\end{equation}
is the \emph{parabolic subalgebra determined by the grading element $\ttT$}.  See \cite[$\S$2.2]{MR3217458} for details.

\subsubsection{Minimal grading elements} \label{S:TvE}

Two distinct grading elements may determine the same parabolic $\fp$.  As a trivial example of this, given a grading element $\ttT$ and $0 < d \in \bZ$, both $\ttT$ and $d \ttT$ determine the same parabolic.  Among those grading elements $\ttT$ determining the same parabolic \eqref{E:p}, only one will have the property that $\fg^{\pm1}$ generates $\fg^\pm$ as an algebra.  That canonical grading element is defined as follows.  Given a parabolic $\fp$ subalgebra (and choice of Cartan and Borel $\fh \subset \fb \subset \fp$), the \emph{grading element associated to $\fp$} is
\begin{equation} \label{E:ttE}
  \ttE \ \dfn \ \sum_{i \in I(\fp)} \ttS^i \,.
\end{equation}
The reductive subalgebra $\fg^0 = \fg^0_\mathrm{ss} \op \fz$ has center $\fz = \tspan_\bC\{ \ttS^i \ | \ i \in I \}$ and semisimple subalgebra $\fg^0_\mathrm{ss} = [\fg^0,\fg^0]$.  A set of simple roots for $\fg^0_\mathrm{ss}$ is given by $\sS(\fg_0) = \{ \a_j \ | \ j \not \in I\}$.


\subsubsection{$\ttE$--eigenspace decompositions} \label{S:Edecomp}

Any $\fg$--representation $U$ admits a $\ttE$--eigenspace decomposition 
$$
  U \ = \ \bigoplus_{q\in\bQ} U^q \,,
$$ 
see \cite[$\S$2.2]{MR3217458}.  In the case that $U = \fg$, we recover \eqref{E:gr1}.  Again, the Jacobi identity implies
\begin{equation}\label{E:gU}
  \fg^\ell ( U^q ) \ \subset \ U^{q+\ell} \,.
\end{equation} 
In particular,
\begin{center}
\emph{each $\ttE$--eigenspace $U^q$ is a $\fg^0$--module.}
\end{center}

If $V$ is the irreducible $\fg$--module of highest weight $\m = \,\sum \m^i \w_i$, then the $\ttE$--eigenspace decomposition is of the form $V = V^m \op V^{m-1} \op V^{m-2} \op \cdots$.  Moreover, if $\m^i = 0$ when $i \not\in I$, then $V^m$ is the (one--dimensional) highest weight line, cf.~\cite[Lemma 6.3]{MR3217458}.  If $\mu^i > 0$ if and only if $i \in I$, then the $G$--orbit of $[V^m] \in \bP V$ is a homogeneous embedding of $G/P$; the embedding is minimal if $\m^i = 1$ for all $i\in I$.

\subsection{Infinitesimal period relation} \label{S:IPR}

This section is a very brief review of the infinitesimal period relation on $G/P$ and the solutions of this differential system, the horizontal submanifolds of $G/P$.  The reader interested in greater detail is encouraged to consult \cite{MR3217458} and the references therein.

The (holomorphic) tangent bundle of 
\[
  \check D \ \dfn \ G/P
\] 
is the $G$--homogeneous vector bundle 
\[
  T\check D \ = \ G \times_P (\fg/\fp) \,.
\]
By \eqref{E:gr} and \eqref{E:p}, the quotient $\fg^{\ge-1}/\fp$ is a $\fp$--module.  Therefore,
\begin{equation} \label{E:ipr}
  T^1 \ \dfn \ G \times_P (\fg^{\ge-1}/\fp)
\end{equation}
defines a homogeneous, holomorphic subbundle of $T\check D$.  This is the \emph{horizontal subbundle}, a.k.a.~the \emph{infinitesimal period relation} (IPR). The IPR is \emph{trivial} if $T^1 = T \check D$.  This is the case when $G/P$ is Hermitian symmetric. 

A \emph{horizontal submanifold} (or \emph{subvariety}) is an integral of $T^1$; that is, a connected complex submanifold $Z \subset G/P$ with the property that $T_zZ \subset T^1_z$ for all $z \in Z$, or any irreducible subvariety $Y \subset \check D$ with the property that $T_yY \subset T^1_y$ for all smooth points $y \in Y$.

\subsection{Adjoint varieties} \label{S:AV}

Throughout the paper we will illustrate the ideas and results with adjoint varieties.  From the Hodge--theoretic perspective, these are the simplest varieties for which the infinitesimal period relation is nontrivial: the adjoint varieties are precisely those $G/P$ for which the IPR \eqref{E:ipr} is a contact distribution.

Let $G \subset \tAut(\fg)$ be the adjoint group of a complex simple Lie algebra $\fg$, and let $\tilde\a \in \Delta^+$ be the \emph{highest root}.  See Table \ref{t:highest_rt}.
%
%
\begin{footnotesize}
\begin{table}[!ht]
\caption{The highest root $\tilde\a$ of $G$.}
\renewcommand{\arraystretch}{1.2}
\begin{tabular}{|c|l|}
\hline $G$ & $\tilde\a$ \\ \hline\hline
 $A_r$ & $\a_1\,+\cdots+\,\a_r \ = \ \w_1 \,+\, \w_r$ \\
 $B_r$ & $\a_1 \,+\, 2 (\a_2\,+\cdots+\,\a_r) \ = \ \w_2$ \\ 
 $C_r$ & $2 (\a_1\,+\cdots+\,\a_{r-1}) \,+\, \a_r \ = \ 2 \w_1$ \\
 $D_r$ & $\a_1 \,+\, 2 (\a_2\,+\cdots+\,\a_{r-2}) \,+\,\a_{r-1} \,+\, \a_r 
         \ = \ \w_2$ \\
 $E_6$ & $\a_1 \,+\, 2\a_2 \,+\, 2\a_3 \,+\, 3\a_4 \,+\, 2\a_5 \,+\, \a_6 
          \ = \ \w_2$ \\
 $E_7$ & $2\a_1 \,+\, 2\a_2 \,+\, 3\a_3 \,+\, 4\a_4 \,+\, 3\a_5 \,+\, 
          2\a_6 \,+\, \a_7 \ = \ \w_1$ \\
 $E_8$ & $2\a_1 \,+\, 3\a_2 \,+\, 4\a_3 \,+\, 6\a_4 \,+\, 5\a_5 \,+\, 4\a_6 \,+\, 
          3\a_7 \,+\, 2\a_8 \ = \ \w_8$ \\
 $F_4$ & $2\a_1 \,+\, 3\a_2 \,+\, 4\a_3 \,+\, 2\a_4 \ = \ \w_1$ \\
 $G_2$ & $3\a_1 \,+\, 2\a_2 \ = \ \w_2$ \\ \hline
\end{tabular}
\label{t:highest_rt}
\end{table}
\end{footnotesize}
Recall $\fg^{\tilde\a}$ is one-dimensional, and therefore a line in $\fg$; let $o \in \bP\fg$ be the corresponding point.  Then the \emph{adjoint variety} $\check D$ is the $G$--orbit of $o$.  Writing $\check D=G/P$, if $\ttT$ is the grading element associated to $P$ then Table \ref{t:highest_rt} yields 
\[
  \fg \ = \ \fg^{-2}\oplus \fg^{-1}\oplus \fg^0 \oplus \fg^1 \oplus \fg^2 \,,
\]
where $\fg^2=\fg^{\tilde{\a}}$.

\begin{remark}
The adjoint varieties are precisely the compact, simply connected, homogeneous complex contact manifolds \cite{MR0137126}.
\end{remark}

\begin{example}
\begin{a_list}
\item 
If $\fg = \fsl_{r+1}\bC$, then the adjoint variety is $\tFlag(1,r,\bC^{r+1})$ is the (partial) flag variety of lines in hyperplanes.
\item
If $\fg = \fso_m\bC$, then the adjoint variety is the orthogonal grassmannian $\tOG(2,\bC^m)$ of 2--planes that are $\nu$--isotropic for a nondegenerate symmetric bilinear form $\n$.
\item If $\fg = \fsp_{2r}\bC$, then the adjoint variety is the second Veronese re-embedding $v_2(\bP^{2r-1}) \subset \bP \,\tSym^2\bC^{2r}$.
\end{a_list}
\end{example}

Assume that $G \not= A_r, C_r$ so that the adjoint representation $\fg$ is fundamental; in this case we call $\check{D}$ a \emph{fundamental adjoint variety}.  Equivalently, $\tilde\a = \w_\tti$ (cf. Table \ref{t:highest_rt}), $\fg$ is the irreducible $\fg$--representation of highest weight $\w_\tti$ and $P$ is maximal; we have $I(\fp) = \{\tti\}$ and 
\begin{equation} \label{E:ta_si}
  \tilde \a \,-\, \a_j \hbox{ is a root if and only if $j=\tti$.}
\end{equation}

\begin{table}[!ht]
\caption[Adjoint varieties]{The dimensions $n = \tdim\,\check D$ and $\sfN = \tdim\,\bP\fg$, and degree $d = \tdeg\,\check D$ for the exceptional adjoint varieties.}
\begin{tabular}{|c|c|c|c|}
\hline
$G$ & $n$ & $d$ & $\sfN$ \\
\hline
$E_6$ & 21 & 151,164 & 77 \\
$E_7$ & 33 & 141,430,680 & 132 \\
$E_8$ & 57 & 126,937,516,885,200 & 247 \\
$F_4$ & 15 & 4,992 & 51 \\
$G_2$ & 5 & 18 & 13 \\ \hline
\end{tabular}
\label{t:numerics}
\end{table}

\subsection{Nilpotent orbits} \label{S:no}

Let $\fg_\bR \subseteq \tEnd(V_\bR,Q)$ be the Lie algebra of $G_\bR$.   A ($n$--variable) \emph{nilpotent orbit} on $D$ consists of a tuple $(F^\sb; N_1,\ldots,N_n )$ such that $F^\sb \in \check D$, the $N_i \in \fg_\bR$ commute and $N_iF^p \subset F^{p-1}$, and the holomorphic map $\psi : \bC^m \to \check D$ defined by
\begin{equation}\label{E:htno}
  \psi(z^1,\ldots,z^n) \ = \ \texp( z^i N_i ) F^\sb 
\end{equation}
has the property that $\psi(z) \in D$ for $\tIm(z^i) \gg 0$.  We shall use the term $\sigma$\emph{-nilpotent orbit} to refer to the submanifold $e^{\bC\,\sigma}F^\sb\subset \check D$.  The associated (open) \emph{nilpotent cone} is
\begin{equation} \label{E:s}
  \tNC\ = \ \{ t^i N_i \ | \ t^i > 0 \} \,.
\end{equation}

Given a nilpotent $N \in \fg_\bR$ such that $N^{k+1}=0$, there exists a unique increasing filtration $W_0(N) \subset W_1(N) \subset \cdots \subset W_{2k}(N)$ of $V_\bR$ with the properties that 
\[
  N\,W_\ell(N) \ \subset \ W_{\ell-2}(N)
\]
and the induced 
\[
  N^\ell : \tGr_{k+\ell} W_\sb(N) \ \to \ \tGr_{k-\ell} W_\sb(N)
\]
is an isomorphism for all $\ell \le k$.  Above, $\tGr_m W_\sb(N) = W_m(N)/W_{m-1}(N)$.  Moreover, 
\[
  Q_\ell(u,v) \ = \  Q(u , N^\ell v)
\]
defines a nondegenerate $(-1)^{k+\ell}$--symmetric bilinear form on $\tGr_{k+\ell} W_\sb(N)$.  

Define
\[
  \tGr_{k+\ell} W_\sb(N)_\tprim \ = \ 
  \tker\,\{
    N^{\ell+1} : \tGr_{k+\ell} W_\sb(N) \to \tGr_{k-\ell-2} W_\sb(N) \} \,,
\]
for all $\ell \ge 0$.  A \emph{limiting mixed Hodge structure} (or \emph{polarized mixed Hodge structure}) on $D$ is given by a pair $(F^\sb,N)$\footnote{Of course, the actual (real) mixed Hodge structure is the pair $(F^\sb,W_{\sb}(N))$.  If $V$ and $\fg$ admit compatible $\bQ$-rational structures, and $N\in\fg_{\bQ}$, then $(V_{\bQ},F^\sb,W_{\sb}(N))$ is a $\bQ$-mixed Hodge structure.} such that $F^\sb \in \check D$, $N \in \fg_\bR$ and $N(F^p) \subset F^{p-1}$, the filtration $F^\sb$ induces a weight $m$ Hodge structure on $\tGr_m W_\sb(N)$ for all $m$, and the Hodge structure on $\tGr_{k+\ell}(W_\sb(N))_\tprim$ is polarized by $Q_\ell$ for all $\ell\ge0$.  The notions of nilpotent orbit and limiting mixed Hodge structure are closely related.  Indeed, they are equivalent when $n=1$.

\begin{theorem}[Cattani, Kaplan, Schmid] \label{T:cks}
Let $D \subset \check D$ be a Mumford--Tate domain (and compact dual) for a Hodge representation of $G_\bR$.
\begin{a_list_emph}
\item 
A pair $(F^\sb;N)$ forms a one--variable nilpotent orbit if and only if it forms a limiting mixed Hodge structure, \emph{\cite[Corollary 3.13]{\CKSdeg} and \cite[Theorem 6.16]{\Schmid}}.  
\item 
Given an $n$--variable nilpotent orbit $(F^\sb;N_1,\ldots,N_n)$, the weight filtration $W_\sb(N)$ does not depend on the choice of $N\in\tNC$ \emph{\cite[Theorem 3.3]{\CKpmhs}.  Let $W_\sb(\tNC)$ denote this common weight filtration.} 
\end{a_list_emph}
\end{theorem}

The \emph{Deligne bigrading} \cite{\CKSdeg, \DeligneII} 
\begin{subequations} \label{SE:deligne}
\begin{equation}
  V_\bC \ = \ \bigoplus I^{p,q}
\end{equation}
associated with the limiting mixed Hodge structure is given by
\begin{equation}
  I^{p,q} \ = \ F^p \,\cap\, W_{p+q} \, \cap \, 
  \Big( \overline{F^q} \,\cap\,W_{p+q} \,+\, 
         \sum_{j\ge1} \overline{F^{q-j}} \,\cap\, W_{p+q-j-1} \Big) \,.
\end{equation}
\end{subequations}
It is the unique bigrading of $V_\bC$ with the properties that 
\begin{equation} \nonumber 
  F^p \ = \ \bigoplus_{r \ge p} I^{r,\sb} \tand
  W_\ell(\tNC) \ = \ \bigoplus_{p+q \le \ell} I^{p,q} \,,
\end{equation}
and
\[
  \overline{I^{p,q}} \ = \ I^{q,p} \quad\hbox{mod} \quad
  \bigoplus_{r<q,s<p} I^{r,s} \,.
\]
The (real) mixed Hodge structure $(F^\sb,W_{\sb}(\sigma))$ is $\bR$-split, i.e. isomorphic to its associated graded, if $\overline{I^{p,q}} = I^{q,p}$.

\section{Griffiths--Yukawa couplings and lines on $G/P$} \label{S:sec5}

Historically, the first exceptional group to be realized geometrically was $G_2$, as the transformation group of a complex 5-manifold preserving a nontrivial distribution \cite{MR1509120, Bryant2000}.  There are two distinct such realizations: one may take the manifold to be either of the homogeneous spaces  $G_2/P_1$ or $G_2/P_2$; the second is an adjoint variety ($\S$\ref{S:AV}).  Each of these $G/P_i$ may be viewed as the compact dual of a MT domain with bracket-generating horizontal distribution \cite[$\S$IV.F]{MR2918237}. This section is motivated by an observation about the Hodge theory associated to the adjoint variety $\check D=G_2/P_2$ and its contact (hyperplane) distribution, which we now briefly explain.  

Suppose $\mathcal{V}=\oplus_{j=0}^n \mathcal{V}^{n-j,j}$ is a variation of Hodge structure (in the classical sense) over a complex manifold $S$.  The IPR is expressed in terms of the flat connection as $\nabla \cF^p \subset \Omega^1_S \ot \cF^{p-1}$.  It follows that $\nabla$ induces an $\cO_S$--linear map $\Theta_S \to \tHom( \cF^p, \cF^{p-1}/\cF^p)$, where $\Theta_S$ is the sheaf of holomorphic vector fields on $S$.  Setting $p = n$, and iterating this map, we obtain an $\cO_S$--linear map $\tSym^n\Theta_S \to \tHom( \cF^n, \cF^0/\cF^1)$.  Recalling that $\cF^0/\cF^1 \simeq \cV^{0,n}$ as smooth vector bundles, we obtain a well-defined linear mapping
\begin{equation} \label{E:preGYC} 
  \tSym^n T_s S \ \to \ \tHom(V^{n,0}_s,V^{0,n}_s)
  \ \cong \ (V^{0,n}_s)^{\otimes 2}
\end{equation}
for each $s\in S$, which we shall call the \emph{Griffiths--Yukawa coupling}.  When $\mathcal{V}$ comes from a family of varieties, it may be of particular interest to study the geometry of subfamilies for which \eqref{E:preGYC} vanishes.  Recent work of Katz \cite{MR3364748} suggests that motives with Hodge numbers $(2,3,2)$ and MT group $G_2$, hence admitting a period map into a quotient of $D=G_2(\bR)/(P_2\cap\overline{P_2})(\bR) \subset \check D$, arise from certain families of elliptically fibered surfaces.\footnote{For one of Katz's families, a direct geometric proof of this is given in the Ph.D. thesis of Genival da Silva Jr. \cite{Jrthesis}.}  We shall not pursue the (algebro-)geometric angle here, but instead look at what representation theory can tell us.


Fixing a point $o\in D$ with stabilizer $P \subset G$ and a compact Cartan $\ft \subset \fg_\bR \cap \fp$ determines the grading element $\ttT$.  The action of $\fg$ on $V$ defines a map
\begin{equation}\label{E:pre2GYC} 
  \tSym^n \fg^{-1} \ \to \ \tHom(V^{n,0},V^{0,n}) 
\end{equation}
which pulls back to \eqref{E:preGYC} under the period map.  If $\fg=\fg_2$ and $\ttT$ is the adjoint grading element, which is defined by $\a_1(\ttT)=0$ and $\a_2(\ttT)=1$ on the simple roots, and $V$ is the 7-dimensional irreducible representation (with $n=2$ and $V^{2,0}=\bC\langle e_1,e_2\rangle$), one may ask in which horizontal tangent directions \eqref{E:pre2GYC} vanishes.  Writing
\[
  \xi \ \dfn \ 
  \xi_0 x^{-\a_2} + \xi_1 x^{-\a_1-\a_2} 
  + \xi_2 x^{-2\a_1-\a_2} + \xi_3 x^{-3\a_1-\a_2} \ \in \ \fg^{-1}\,,
\]
and $\{ e^*_1,e^*_2\}$ for the dual basis of $V^{0,2}$, a straightforward computation (with appropriate normalizations) gives
\begin{equation}\label{E:xi^2} 
  {}_{e^*}[\xi^2]_e \ = \ 
  \left( \begin{array}{cc} -2\xi_1\xi_3+2\xi_2^2 
    & \xi_1\xi_2-\xi_0\xi_3 \\ 
    \xi_1\xi_2-\xi_0\xi_3 
    & -2\xi_0\xi_2+2\xi_1^2  \end{array}\right)\,.
\end{equation}
The common vanishing locus of the matrix entries in \eqref{E:xi^2} yields a twisted cubic curve $v_3(\mathbb{P}^1)\subset\mathbb{P}\fg^{-1}$, \cf\cite{MR3115136}.  Upon varying $o\in D$ (or $\check D$), this recovers the field of cubic cones in the contact planes $\fg^{-1}_o\subset T_o\check D$ preserved by $G_2$, \cf\cite{Bryant2000}.

Furthermore, it is known that the lines $\ell\subset \fg^{-1} \subset T_o\check D$ in each cubic curve are precisely the lines through $o$ contained in $\check D$, under its minimal homogeneous embedding in $\mathbb{P}\fg$.  To see this, we apply $(\tad\xi)^2$ to the highest weight vector $v=x^{3\a_1+2\a_2}\in \fg^2$ to check vanishing of the second fundamental form:
\[
  (\tad\xi)^2 v \ = \ (-2\xi_0\xi_2+2\xi_1^2)x^{\a_1}
  +(\xi_0\xi_3-\xi_1\xi_2)\ttH^{3\a_1+2\a_2}
  +(-2\xi_1\xi_3+2\xi_2^2)x^{-\a_1} \ \in \ \fg^0 \,.
\]
Hence we get an identification of the variety of horizontal lines $\cC_o$ (cf. \eqref{E:Co}) with the ``variety of Yukawa vanishing directions'', which provides a homogeneous-space description of the latter and a Hodge--theoretic interpretation of $\cC_o$. Moreover, the argument produces equations for both as projective varieties in $\mathbb{P}\fg^{-1}$.  The purpose of this section is to discern the degree to which both phenomena generalize.

We shall require a few preliminaries on the projective geometry of flag varieties ($\S\S$\ref{S:lines}-\ref{S:uniruling}); the main results follow in $\S\S$\ref{S:GYCvan}-\ref{S:GYeqns}.

\subsection{Lines on flag varieties} \label{S:lines}

Let $\check D$ be the image of the minimal homogeneous embedding $G/P \inj \bP V$, cf.~$\S$\ref{S:rhv}.  Fix a highest weight vector $0\not=v \in V$, so that $[v] = o \in \bP V$ is the highest weight line.  By \eqref{E:p}, the tangent space $T_o\check D$ is naturally identified with $\fg/\fp$ as a $\fp$--module, and with $\fg^-$ as a $\fg^0$--module; for the most part, we will work with the latter identification.  The set of embedded, linear $\bP^1 \subset \bP V$ containing $o$ and tangent to $\check D$ at that point is in bijection with $\bP \fg^- = \bP\,T_o\check D$.  To be precise, given a tangent line $[\xi] \in \bP \fg^-$, we have 
$$
  \bP^1 \ = \ \bP^1(o,[\xi]) \ \dfn\ \bP \, \tspan_\bC\{v , \xi(v)\} \subset \bP V \,.
$$  
Making use of this identification, let 
$$
  \tilde \cC_o \ \dfn \ \{ [\xi] \in \bP \,\fg^- \ | \ \bP^1(o,[\xi]) \subset \check D \}
  \ = \ \{ \bP^1 \subset \bP V \ | \ o \in \bP^1 \subset \check D \} 
$$ 
be the set of lines on $\check D$ passing through $o$.  (For a general embedding $G/P \inj \bP V$, not necessarily minimal, $\tilde \cC_o$ is defined to be the variety of minimal rational tangents, cf.~\cite{MR1748609}.)  The subvariety of lines tangent to $\fg^{-1} \subset \fg^- \simeq T_o\check D$ is
\begin{equation}\label{E:Co}
  \cC_o \ \dfn \ \{ [\xi] \in \tilde \cC_o \ | \ \xi \in \fg^{-1} \} 
  \ = \ \{ \bP^1 \subset \check D \ | \ \bP^1 \ni o 
  \hbox{ is horizontal}\}\,.
\end{equation}

Let 
\begin{equation}\label{E:Xo}
  X \ \dfn \ \bigcup_{\bP^1 \in \cC_o} \bP^1
\end{equation}
be the variety swept out by the lines that pass through $o$ and are horizontal.

\begin{proposition}[{\cite{KR2}}] \label{P:Xo}
Assume that $P$ is a maximal parabolic subgroup of $G$, and let $\check D$ be the minimal homogeneous embedding of $G/P$ \emph{($\S$\ref{S:rhv})}.  Then 
\begin{a_list_emph}
\item $X$ is a cone over $\cC_o$ with vertex $o$ and a Schubert variety.  
\item $X$ is horizontal if and only if the simple root $\a_\tti$ associated with the maximal parabolic $\fp$ is not short.
\end{a_list_emph}
\end{proposition}

\subsection{The case that $P$ is maximal} \label{S:Pmax}

We now recall two properties of $\cC_o$ in the case that $P$ is maximal (equivalently, $I = \{\tti\}$ and $\ttE = \ttS^\tti$).  The results that follow are due to \cite{MR1966752}, where $\cC_o$ and $\tilde\cC_o$ are discussed for arbitrary (not necessarily maximal) $P$.
\begin{a_list_bold}
\item  
   Since $P$ is maximal, $\fg^{-1}$ is an irreducible $\fg^0$--module with highest weight line $\fg^{-\a_\tti}$, cf.  \cite[Theorem 8.13.3]{MR928600}.  The variety of lines $\cC_o \subset \bP \fg^{-1}$ is the $G^0$--orbit of this highest weight line, cf.~\cite[Theorem 4.3]{MR1966752}.  In particular, $\cC_o$ is a rational homogeneous variety; indeed, 
$$
  \cC_o \ \simeq \ G^0/(G^0\cap Q)\,,
$$ 
where $Q \supset B$ is the parabolic subgroup defined by 
\begin{equation} \label{E:Iq}
  I(\fq) \ = \ \{ j \ | \ \fg^{-\a_j} \not\subset \fq\} 
  \ \dfn \ \{ j \ | \ \langle \a_\tti , \a_j \rangle \not=0\} \,.
\end{equation}
That is, the simple roots indexed by $I(\fq)$ are those adjacent to $\a_\tti$ in the Dynkin diagram of $\fg$; cf.~\cite[Proposition 2.5]{MR1966752}.  With only a few exceptions, $\cC_o$ is a $G_0$--cominuscule variety; equivalently, $\cC_o \simeq G^0/(G^0\cap Q)$ admits the structure of a compact Hermitian symmetric space, cf.~\cite[Proposition 2.11]{MR1966752}.  
\item
If the simple root $\a_\tti$ associated with the maximal parabolic $P$ is not short, then $\cC_o = \tilde \cC_o$, cf.~\cite[Theorem 4.8.1]{MR1966752}.  If the simple root is short, then $\tilde \cC_o$ is the union of two $P$--orbits, and open orbit and its boundary $\cC_o$.  
\end{a_list_bold}

\begin{remark}[Adjoint varieties] \label{R:Pmax}
In $\S\S$\ref{S:adII}-\ref{S:[X(N)]eg}, we will be interested in the variety \eqref{E:Xo} in the case that $G/P$ is a fundamental adjoint variety.    In those cases, $P$ is a maximal parabolic, so that \cite[Corollary 4.12]{KR2} applies: the $\cC_o$ are listed in Table \ref{t:adj}.
\begin{table}[!ht]
\caption{The set $\cC_o$ for the fundamental adjoint varieties}
\renewcommand{\arraystretch}{1.3}
\begin{tabular}{|c||c|c|c|c|c|c|}
    \hline
    $G/P$ & $\tOG(2,\bC^n)$ & $E_6/P_2$ & $E_7/P_1$ & $E_8/P_8$ & $F_4/P_1$ 
          & $G_2/P_2$ \\  
  $\cC_o$ & $\bP^1 \times \cQ^{n-6}$ & $\tGr(3,\bC^6)$ & $\cS_6$ 
          & $E_7/P_7$ & $\tLG(3,\bC^6)$ & $v_3(\bP^1)$ \\ \hline
\end{tabular}
\label{t:adj}
\end{table}
In the table $\cS_6$ is a Spinor variety, one of the two connected components of the orthogonal grassmannian $\tOG(6,\bC^{12})$.

Moreover, for each of the fundamental adjoint varieties, the simple root $\a_\tti$ associated with the maximal parabolic $P = P_\tti$ is not short.  Whence Proposition \ref{P:Xo} applies and the variety $X$ swept out by lines passing through a fixed point is horizontal.
\end{remark}

\subsection{The case of a general parabolic} \label{S:Pgen}

If we drop the assumption that $P$ is maximal, then $\cC_o$ may be described as follows.  Let $I$ be the index set \eqref{E:I} corresponding to $\fp$.  Given $i \in I$, let 
$$
  \fg^{-1}_i \ = \ \{ \z \in \fg^{-1} \ | \ [\ttS^i , \xi] = -\xi \} \,.
$$
Each $\fg^{-1}_i$ is an irreducible $G^0$--module with highest weight line $\fg^{-\a_i}$, and 
$$
  \fg^0 \ = \ \bigoplus_{i \in I} \fg^{-1}_i 
$$
is a $G^0$--module decomposition, cf.  \cite[Theorem 8.13.3]{MR928600}.  Let $\cC_{o,i} \subset \bP \fg^{-1}_i$ be the $G^0$--orbit of the highest weight line $\fg^{-\a_i}$.  Then the variety of lines $\cC_o \subset \bP\fg^{-1}$ is the disjoint union
\begin{equation} \label{E:CgenP}
  \cC_o \ = \ \bigsqcup_{ i \in I } \cC_{o,i} \,.
\end{equation}
As in the case that $P$ is maximal ($\S$\ref{S:Pmax}), the assertions here are established in \cite[$\S$4]{MR1966752}.

\subsection{Uniruling of $\check D$} \label{S:uniruling}

For ease of exposition we continue with the assumption that $P$ is maximal.  However, analogous statements follow for unirulings on general $G/P$.

Given $x = g o \in \check D$, with $g \in G$, let $\cC_x = g \cC_o$ denote the corresponding set of lines through $x$. (It is an exercise to show that $\cC_x$ is well-defined; that is, $\cC_x$ does not depend on our choice of $g$ yielding $x = go$.)  Then 
$$
  \cC \ \dfn \ \{ \bP^1 \ | \ \bP^1 \in \cC_x \,,\ x\in \check D\} 
  \ = \ \bigcup_{g \in G} g \,\cC_o \,.
$$
forms a uniruling of $\check D$.  (By \cite[Corollary 4.12]{KR2}, this uniruling is parameterized by $G/Q$ -- that is, $\cC \simeq G/Q$.)  

\begin{remark}\label{R:tC}
More generally, the set of \emph{all} lines on $G/P$ is $\tilde \cC = \cup_{g\in G} \, g \,\tilde\cC_o$.
\begin{a_list_emph}
\item  It follows from definition \eqref{E:Co} and the homogeneity of the IPR, that 
\begin{equation} \label{E:Cogen}
  \cC \ = \ \{ \bP^1 \in \tilde\cC \ | \ \bP^1 \hbox{ is horizontal}\}
\end{equation}
is precisely the set of lines on $G/P$ that are integrals of the IPR.
\item  As noted in $\S$\ref{S:Pmax}(b), if the simple root associated to the maximal parabolic $P$ is not short, then $\tilde \cC = \cC$ consists of a single $G$--orbit.  If the simple root is short, then $\tilde\cC$ consists of two $G$--orbits, an open orbit and its boundary $\cC$, cf.~\cite[Theorem 4.3]{MR1966752}.
\end{a_list_emph}
\end{remark}

\begin{remark}[$\bP^1$--unirulings of $G/P$ for maximal $P$] \label{R:P1}
In the case that $P \subset G$ is a maximal parabolic ($\S$\ref{S:rhv}), there is a unique $G$--homogeneous variety $G/Q$ parameterizing a uniruling of $G/P$ by lines $\bP^1$; it may be identified by inspection of the Dynkin diagram $\sD$ of $\fg$ as follows.  The maximality of $P$ is equivalent to $I(\fp) = \{\tti\}$ for some $\tti$.  In order to obtain a uniruling by $\bP^1$s, we must choose $Q$ (equivalently, the index set $I(\fq)$) so that $G'/P' = \bP^1$.  To that end, let $J = \{ j \not= \tti \ | \ (\a_\tti,\a_j) \not=0 \}$ index the nodes in the Dynkin diagram that are adjacent to the $\tti$--th node.  Then $G'/P' \simeq \bP^1$ (equivalently, $G/Q$ parameterizes a uniruling of $G/P$ by $\bP^1$s) if and only if $J \subset I(\fq)$.  When $I(\fq) = J$ we say that \emph{$G/Q$ is the smallest rational $G$--homogeneous variety parameterizing a uniruling of $G/P$ by lines $\bP^1$}.

Examples \ref{eg:ptSO} and \ref{eg:excep_adj} below identify the varieties $G/Q$ parameterizing lines on the fundamental adjoint varieties ($\S$\ref{S:AV}).
\end{remark}

\begin{example}[Unirulings of the orthogonal adjoint varieties] \label{eg:ptSO} 
Fix a nondegenerate symmetric bilinear form $\n$ on $\bC^n$, $n\ge 7$.  Let $G = \tAut(\bC^n,\n)^0 = \tSO(\bC^n)$ denote the identity component of the orthogonal group.  The adjoint variety ($\S$\ref{S:AV})
$$
  G/P_2 \ = \ \tOG(2,\bC^n) \ = \ 
  \left\{ E \in \tGr(2,\bC^n) \ \left| \ \left.\n\right|_{E} = 0 \right. \right\}
$$ 
is the set of $\n$--isotropic 2-planes.  The partial flag variety
$$
  G/P_{1,3} \ = \ \tFlag_\n(1,3,\bC^n)
  \ = \ \{ F^1 \in \tOG(1,\bC^n) \times \tOG(3,\bC^n) \ | \ F^1 \subset F^3 \}  
$$ 
parameterizes a uniruling of the adjoint variety $G/P_2$ by $\bP^1$s.  Given one such flag $F^1 \subset F^3$ the corresponding line is $\{ E \in \tOG(2,\bC^n) \ | \ F^1 \subset E \subset F^3 \}$.

In this case $\sD' = \sD \backslash\{2\}$ so that $\fg' = \fsl_2\bC \,\op\, \fso_{n-4}\bC$.  Additionally, $I(\fq) = I(\fp') = \{1,3\}$.
\end{example}

\begin{example}[Unirulings of the exceptional adjoint varieties] \label{eg:excep_adj}
Let $G/P \inj \bP\fg$ be the adjoint variety ($\S$\ref{S:AV}) of an exceptional simple Lie group $G$.  The rational homogeneous variety $G/Q$ parameterizing a uniruling of $G/P$ by lines $\bP^1$ is given in Table \ref{t:adjvar_lines}.
\begin{table}[!ht] \renewcommand{\arraystretch}{1.3}
\caption{Lines on exceptional adjoint varieties}
\begin{tabular}{|r||c|c|c|c|c|}
  \hline
  Adjoint variety $G/P$ & 
     $E_6/P_2$ & $E_7/P_1$ & $E_8/P_8$ & $F_4/P_1$ & $G_2/P_2$ \\ 
  Variety $G/Q$ of lines & 
     $E_6/P_4$ & $E_7/P_3$ & $E_8/P_7$ & $F_4/P_2$ & $G_2/P_1$ \\ \hline
\end{tabular}
\label{t:adjvar_lines}
\end{table}
In the table the subscript $\tti$ in $P_\tti$ indicates the set $I = \{\tti\}$ of \eqref{E:I}.  For example, $G/P = E_6/P_2$ indicates that $\fg$ is the exceptional simple Lie algebra of rank six and and $I(\fp) = \{2\}$.  
\end{example}

\subsection{Griffiths--Yukawa vanishes along horizontal lines} \label{S:GYCvan}

Fix $\check D = G/P$.  (We do not assume that $P$ is maximal here.)  Let $\ttE$ be the grading element \eqref{E:ttE} associated with the parabolic $\fp$, and let $V$ be the irreducible $G$--module giving the minimal homogeneous embedding $G/P \inj \bP V$ ($\S$\ref{S:IPR}).  

Fix a $G$--module $U$, and suppose that the $\ttE$--eigenspace decomposition is of the form 
\begin{equation} \label{E:U}
  U \ = \  U^{m/2} \,\op\, U^{m/2-1}\,\op\cdots\op\, 
  U^{1-m/2}\,\op\, U^{-m/2} \,,
\end{equation}
with $0 \le m \in \bZ$, $U^{m/2} \not=0$ and $(U^{a})^* \simeq U^{-a}$ as $\fg_0$--modules.

\begin{remark} \label{R:GYC}
In general, the $\ttE$--eigenvalues of $U$ are rational.  However, if $G$ is $\bQ$--algebraic and $\ttE$ is the grading element $\ttT_\varphi$ associated to a Hodge representation $(G,U_\bQ,\varphi)$, cf.~\cite[$\S$2.3]{MR3217458}, then the $\ttE$--eigenspace decomposition of $U = U_\bQ \ot \bC$ is of the form \eqref{E:U}.
\end{remark}

Any element of $\tSym^m(\fg^{-1})$ naturally defines a $G^0$--module map $U^{m/2} \to U^{-m/2}$.  Explicitly, given $\xi \in \fg^{-1}$ and $u \in U^{m/2}$, the relation \eqref{E:gU} implies $\xi^m(u) \in U^{-m/2}$.  Thus, we have a $G^0$--module map 
$$
  \Psi : \tSym^m(\fg^{-1}) \, \to\, \tHom(U^{m/2} , U^{-m/2}) \, ,
$$
which is the Griffiths--Yukawa coupling from above.  Define
\begin{subequations} \label{SE:cY}
\begin{equation} \label{E:cY}
  \cY_U \ \dfn \ 
  \{ [\x] \in \bP(\fg^{-1}) \ | \ \Psi(\x^m) = 0 \} \, .
\end{equation}
In particular, given $0\not=\xi \in \fg^{-1}$, 
\begin{equation}
  [\xi] \in \cY_U \quad\hbox{if and only if}\quad \left.\xi^m\right|_{U^{m/2}} = 0\,.
\end{equation}
\end{subequations}
Observe that $\cY_U$ is a closed, $G^0$--invariant subvariety of $\bP\fg^{-1}$.

Our main goal in this section is to understand the relationship between $\cY_U$ and $\cC_o$.

\begin{theorem} \label{T:GY}
If the the kernel of the Griffiths--Yukawa coupling is nonempty, then it contains the (tangent directions to) lines through $o = P/P$ in the minimal homogeneous embedding $G/P \inj \bP V$ that are horizontal.  That is, if $\cY_U \not= \emptyset$, then
\begin{equation} \label{E:CvY}
  \cC_o \ \subset \ \cY_U \,.
\end{equation}
\end{theorem}

\begin{remark}
The relationship \eqref{E:CvY} was proved by Sheng and Zuo \cite{MR2657440} in the special case that $D$ is Hermitian symmetric and $U = V$.  In this setting $m$ is the rank of $D$, and $\bP \fg^{-1}$ decomposes into $m$ $G^0$--orbits $\sC_1 , \ldots , \sC_m$ with the properties $\sC_{a-1} \subset \overline{\sC_{a}}$ for all $2\le a \le m$, and $\sC_1 = \cC_o$ and $\overline{\sC_{m-1}} = \cY_U$.
\end{remark}

\noindent The theorem is proved in $\S$\ref{S:prfGY}.  The equations cutting out $\cC_o$ and $\cY_U$ are given in $\S$\ref{S:GYeqns}.  In general, containment is strict in \eqref{E:CvY}, \cf\cite[Example 5.19]{KR1long}.   A large class of examples for which equality holds is the following.  Take
$$
  V \ = \ \fg \,,
$$  
so that $\check D \subset \bP \fg$ is the adjoint variety of $G$ ($\S$\ref{S:AV}).  The associated grading element is $\ttE = \sum_{i \in I} \ttS^i$ where the set $I$ is defined by $\tilde\a = \sum_{i \in I} \w_i$, cf. Table \ref{t:highest_rt}.  Further assume that $m=2$; equivalently, the $\ttE$--eigenspace decomposition \eqref{E:U} of $U$ is 
\begin{equation} \label{E:U1}
  U \ = \ U^1 \,\op\, U^0 \,\op\, U^{-1} 
  \,.
\end{equation}
 
\begin{theorem}\label{T:adj_m=2}
Let $V = \fg$, so that $G/P \inj \bP\fg$ is the adjoint variety of $G$.  Assume this homogeneous embedding of $G/P$ is minimal.  Suppose that $m=2$, so that \eqref{E:U1} holds.  If $U^1$ is a faithful representation of the Levi factor $\fg^0$, then $\cC_o = \cY_U$.
\end{theorem}

\begin{remark} \label{R:adj_sp}
The adjoint variety fails to be a minimal homogeneous embedding only for the simple Lie algebra $\fg = \fsp_{2r}\bC$ --- it is the second Veronese embedding $v_2(\bP^{2r-1}) \subset \bP\,\tSym^2\bC^{2r}$.  As such it contains \emph{no} lines.
\end{remark}


Theorem \ref{T:adj_m=2} is proved in $\S$\ref{S:prf_adj_m=2}.  Given the theorem, it is interesting to identify those irreducible representations $U$ of $G$ for which \eqref{E:U1} holds when $\ttE$ is the grading element associated to an adjoint variety.  These representations are listed in 

\begin{lemma} \label{L:m=2}
Let $G/P \inj \bP\fg$ be an adjoint variety with associated grading element $\ttE$.  Let $U$ be an irreducible representation of highest weight $\lambda$.  Then $\lambda(\ttE) = 1$ if and only if $\lambda$ is among those listed below.
\begin{a_list_emph}
\item $\fg = \fsl_{r+1}\bC$ and $\ttE= \ttS^1 + \ttS^r$: $\lambda = \w_i$, for any $1 \le i \le r$; in this case $U = \tw^i \bC^{r+1}$.
\item $\fg = \fso_{2r+1}\bC$ and $\ttE = \ttS^2$: either $\lambda = \w_1$, in which case $U = \bC^{2r+1}$ is the standard representation; or $\lambda = \w_r$, in which case $U$ is the spin representation.
\item $\fg = \fsp_{2r}\bC$ and $\ttE = \ttS^1$: $\lambda = \w_i$, for any $1 \le i \le r$.
\item $\fg = \fso_{2r}\bC$ and $\ttE = \ttS^2$:  either $\lambda = \w_1$, in which case $U = \bC^{2r}$ is the standard representation; or $\lambda = \w_{r-1}\,,\ \w_r$, in which case $U$ is one of the spin representations.
\item $\fg = \fe_6$ and $\ttE = \ttS^2$: $\lambda = \w_1 \,,\ \w_6$. \\
$\fg = \fe_7$ and $\ttE = \ttS^1$: $\lambda = \w_7$. \\
(In the case that $\fg = \fe_8$ and $\ttE = \ttS^8$, we have $\lambda(\ttE) > 1$ for all $\lambda$.)
\item $\fg = \ff_4$ and $\ttE = \ttS^1$: $\lambda = \w_4$.
\item $\fg = \fg_2$ and $\ttE = \ttS^2$: $\lambda = \w_1$ and $U = \bC^7$ is the standard representation.
\end{a_list_emph}
In each of these cases, we also have $\lambda^*(\ttE) = 1$.
\end{lemma}


\begin{proof}
As a highest weight $\lambda$ is of the form $\lambda = \sum_i \lambda^i\w_i$ with all $\lambda^i \geq 0$.  The tables of \cite{Bourbaki} express the $\w_i$ as linear combinations of the simple roots $\a_j$.  The lemma follows.
\end{proof}

\noindent Among the cases in Lemma \ref{L:m=2}, one can check\footnote{See \cite[\S 5.3]{KR1long} for these and related computations.} that the exceptional cases (e), (f), (g) satisfy the faithfulness condition of Theorem \ref{T:adj_m=2}, whereas (a)-(d) do not. In at least one case, the situation is improved by considering reducible $U$: if $\fg = \mathfrak{sl}_{r+1}$ and $\ttE = \ttS^1 + \ttS^r$, then $U=V^{\omega_i}+V^{\omega_{r+1-i}}$ (with $2\leq i \leq r-1$ unless $r=2$) satisfies the conditions of the theorem.

\subsection{Equations cutting out $\cC_o$ and $\cY_U$}\label{S:GYeqns}

We begin with $\cY_U$.  Let 
$$
  \Psi^* : \tHom(U^{-m/2} , U^{m/2}) \to \tSym^m(\fg^{-1})^*
$$
denote the dual map.  To be explicit: let $\n \ot u \in U^{m/2} \ot (U^{-m/2})^*= \tHom(U^{-m/2},U^{m/2})$, then the polynomial $\Phi^*(\n\ot u) \in \tSym^m(\fg^{-1})^*$ is given by $\Phi^*(\n \ot u)(\xi) = \n(\xi^m(u))$, for any $\xi \in \fg^{-1}$.  The image
\begin{equation} \label{E:eqnY}
\hbox{\emph{$\tim\,\Psi^* \subset \tSym^m(\fg^{-1})^*$ is the set of polynomials defining $\cY_U \subset \bP \fg^{-1}$}.}
\end{equation}

Turning to $\cC_o$, the following proposition characterizes $\cC_o$ as the set of lines in $\fg^{-1}$ whose iterated action annihilates the highest weight line of $V$.  

\begin{proposition} \label{L:GY}
Let $G/P \inj \bP V$ be the minimal homogeneous embedding.  Let $0 \not= v \in  V$ be a highest weight vector.  Let $\xi \in \fg^{-1} \subset\fg^- \simeq T_o\check D$.  Then $[\x] \in \cC_o$ if and only if $\x^2(v) = 0$.
\end{proposition}

\noindent Before proving the proposition, we note the following corollary which describes the equations cutting out the lines $\cC_o$.  Recall the $\ttE$--eigenspace decomposition 
$$
  V \ = \ V^h \,\op\, V^{h-1} \,\op\, V^{h-2}\,\op\cdots
$$
of $\S$\ref{S:Edecomp}, and that $V^h$ is one--dimensional.  Fix a highest weight vector $0\not= v \in V^h$.  Given $\n \in (V^{h-2})^*$ define $p_\n \in \tSym^2(\fg^{-1})^*$ by $p_\n(\xi) = \n(\xi^2 v )$.  This defines a $\fg_0$--module map
$$
  \Phi^* : (V^{h-2})^*\ot V^h \ \to \ \tSym^2(\fg^{-1})^* \,.
$$
By Lemma \ref{L:GY}, this map has the property that the image
\begin{equation} \label{E:eqnC}
\hbox{\emph{$\tim\,\Phi^* \subset \tSym^2(\fg^{-1})^*$ is the set of polynomials defining $\cC_o \subset \bP \fg^{-1}$}.}
\end{equation}
In particular, if $G/P$ is one of the fundamental adjoint varieties, then $\cC_o$ is the variety of lines through a fixed point $o$ (cf. \ref{R:Pmax}); and \eqref{E:eqnC} gives equations for $\cC_o$\footnote{and for $\cY_U$, when Theorem \ref{T:adj_m=2} applies} as a projective variety.

\begin{proof}[Proof of Proposition \ref{L:GY}]
Define $T = \tspan_\bC \{ v ,\, (\fg^-)v \} \subset V$.  Then $T$ is the embedded tangent space to the cone $C(\check D) \subset V$ over $\check D$ at $v$.  (Note that $T$ depends only on the highest weight line $o = V^h$, not on our choice of $v \in o$.)  

The unique $\bP^1 \subset \bP V$ containing $o$ and satisfying $\x \in T_o\bP^1$ is $\bP\,\tspan_\bC\{ v , \x(v)\}$.  Since $\check D$ is cut out by degree two equations \cite[$\S$2.10]{MR1782635}, the line $\bP^1$ is contained in $\check D$ if and only if the line osculates with $\check D$ to second order; equivalently, the second fundamental form $F^2$ vanishes at $\x$.  That is, 
$$
  \cC_o \ = \ \bP \{ \xi \in \fg^{-1} \ | \ F^2(\xi) = 0 \}  \,.
$$  
Since $F^2(\x) = \x^2(v)/T \in V/T$, the vanishing $F^2(\x) = 0$ is equivalent to $\x^2(v) \in T$.  

If $\x^2(v) = 0$, then it is immediate that $F^2(\x) = 0$.  Whence, $[\x] \in \cC_o$.  

Conversely, suppose that $[\x] \in \cC_o$.  Then $[\x]$ is necessarily contained in one of the $\cC_{o,i}$ of \eqref{E:Cogen}.  Since $\cC_{o,i}$ is the $G^0$--orbit of $\fg^{-\a_i} \in \bP\fg^{-1}$, we may assume without loss of generality that $\x \in \fg^{-\a_i}$.  Let $\tilde\m$ denote the highest weight of $V$.  Then $\x^2(v)$ is a weight vector of $V$ for the weight $\tilde\mu - 2\a_i$.  Since $[\xi] \in\cC_o$, and this variety is the zero locus of the second fundamental form, $\x^2(v)$ is necessarily of the form $a v + \z(v)$ for some $a \in \bC$ and $\z \in \fg^{-}$.  Write $\z = \sum \z^\b$ with $\z^\b \in \fg^{-\b}$.  Each $\z^\b(v)$ is a weight vector of $V$ for the weight $\tilde\m - \b$.  Therefore, the condition that $\x^2(v) = a v + \z(v)$ be a weight vector for $\tilde\mu - 2\a_i$ forces $a =0$ and all but at most one $\z^\b(v)$ to vanish.  If $\z^\b(v) \not=0$, then $\tilde\mu - 2\a_i = \tilde\m - \b$ forcing $\b = 2\a_\tti$.  This is not possible, since $2\a_i$ is not a root.  Therefore, $\x^2(v) = 0$.
\end{proof}

\subsection{Proof of Theorem \ref{T:GY}}  \label{S:prfGY}

Let $I = \{ i_1 , \ldots, i_\ell\}$ be the index set \eqref{E:I} associated with the parabolic subalgebra $\fp$.  The proof will proceed in two steps.  We begin with the case that $\fp$ is a maximal parabolic (equivalently $|I| = 1$); this forms the basis for the general argument.

\begin{lemma}\label{L:GYmaxP}
Suppose that $\check D = G/P$ and that $P$ is a maximal parabolic.  If $\cY_U \not= \emptyset$, then $\cY_U$ is connected and 
$$
  \cC_o \ \subset \ \cY_U\,.
$$
\end{lemma}

\begin{proof}
The parabolic $P$ is maximal if and only if $\fg^{-1}$ is an irreducible $\fg^0$--module.  In this case, $\cC_o \subset \bP\fg^{-1}$ is the unique closed $G^0$--orbit ($\S$\ref{S:Pmax}).  It is clear from the definition \eqref{SE:cY} that $\cY_U$ is closed and preserved under the action of $G^0$.  Therefore, each connected component of $\cY_U$ contains a closed $G^0$--orbit.
\end{proof}

We now turn to the case that $|I| > 1$.  First, we review the relationship between the lines on $G/P \inj \bP V$ and $G/P_i \inj \bP V_{\w_i}$, $i \in I$.  The lines $\cC_o$ on the former are described in $\S$\ref{S:Pgen}.  For the latter, given $i \in I$, let $P_i \subset G$ be the maximal parabolic associated with the index set $\{i\}$.  Let $\check D_i \subset \bP V_{\w_i}$ be the image of the minimal homogeneous embedding $G/P_i \inj \bP V_{\w_i}$, and let $\cC_{o_i}(\check D_i)$ be the lines on $\check D_i$ passing through $o_i = P_i/P_i$.  Then the $\cC_{o,i}$ of \eqref{E:CgenP} is 
\begin{equation} \label{E:prfGY}
  \cC_{o,i} \ = \ \left\{ [\xi] \in \bP\fg^{-1} \ | \ [\xi]\in\cC_{o_i}(\check D_i) \,,\ 
  \xi \in \fp_j \ \forall j\in I\setminus \{i\} \right\} \,.
\end{equation}

Recall that $\ttE = \sum_{i \in I} \ttS^i$, and $U^{m/2}$ is the eigenspace of the largest $\ttE$--eigenvalue $m/2$ on $U$.  Let $m_i/2$ be the largest $\ttS^i$--eigenvalue of $U$, and let $\sU_i \subset U$ be the corresponding eigenspace.  Then 
$$
  m \ = \ \sum_{i \in I} m_i \tand
  U^{m/2} \ = \ \bigcap_{i \in I} \sU_i \,.
$$  
Let $[\xi] \in \cC_{o_i}(\check D_i)$.  Then \eqref{SE:cY} and Lemma \ref{L:GYmaxP} imply that $\left.\xi^{m_i}\right|_{\sU_i} = 0$.  Since $m_i \le m$ and $U^{m/2} \subset \sU_i$, it follows that $\left.\xi^{m}\right|_{U^{m/2}} = 0$.  Theorem \ref{T:GY} now follows from \eqref{E:CgenP}, \eqref{SE:cY} and \eqref{E:prfGY}.

\subsection{Proof of Theorem \ref{T:adj_m=2}}  \label{S:prf_adj_m=2}

Given Theorem \ref{T:GY} it suffices to show that 
\begin{equation}\label{E:YinC}
  \cY_U \ \subset \ \cC_o \,.
\end{equation}
To that end, let $[\xi] \in \cY_U$.  Equivalently,
\begin{equation}\label{E:uinY}
  \xi^2(u) \ = \ 0 \quad \hbox{for all} \quad u \in U^1 \,.
\end{equation}
Fix a highest root vector $0 \not= v \in \fg^{\tilde\a}$.  Given Proposition \ref{L:GY}, to establish \eqref{E:YinC} it suffices to show that 
\begin{equation}\label{E:ad2v}
  \tad_\x^2(v) \ = \ 0 \,.
\end{equation}

As an operator on $U$
\[
  \tad_\xi^2(v) \ = \ \xi \xi v \,-\, 2 \xi v \xi \,+\, v \xi \xi \,.
\]
As an element of $\fg^{-1}$, the endomorphism $\xi$ lowers eigenvalues by one; that is $\xi$ maps $U^a$ into $U^{a-1}$.  Likewise, as an element of $\fg^2 = \fg^{\tilde\a}$, the root vector $v$ maps $U^a$ into $U^{a+2}$.  It then follows from \eqref{E:U1} that 
\begin{equation} \label{E:ad2vU}
  \left. \tad_\xi^2(v) \right|_{U^1} \ = \ \left.v \xi \xi\right|_{U^1} \ 
  \stackrel{\eqref{E:uinY}}{=} \ 0  \,.  
\end{equation}
On the other hand $\tad_\x^2(v) \in \fg^0$.  By assumption $U^1$ is a faithful representation of $\fg^0$.  So \eqref{E:ad2vU} holds if and only if \eqref{E:ad2v} holds.

\section{Construction of limiting mixed Hodge structures} \label{S:cayleyconstr}

\subsection{Overview}

In this section we construct a large class of polarized $G_\bR$--orbits; each orbit will contain the image of a \naive~ limit mapping $\Phi_\infty: B(\sN) \to \check D$, where $\sN$ is an element of a nilpotent cone $\s$ (cf. $\S$\ref{S:B+D}).

Fix a rational homogeneous variety $\check D = G/P$ ($\S$\ref{S:rhv}), and let $\ttE$ be the associated grading element \eqref{E:ttE}.  Associated with $\ttE$ is a rational form $\fg_\bQ$ of $\fg$ ($\S$\ref{S:Qform}).  The rational form is equipped with a Hodge structure that realizes an open $G_\bR$--orbit $D \subset \check D$ as a MT domain ($\S$\ref{S:HSgZ}).  We will construct a sequence of $G_\bR$--orbits $\{\cO_j\}_{j=0}^s$ in $\check D$ such that $\cO_0$ is the open orbit $D$ and 
\begin{equation} \label{E:orbit}
  \cO_{j+1} \ \subset \ \del \cO_j \,.
\end{equation}
We will see that these orbits are rational, polarized and (weakly) cuspidal in the sense of \cite{MR3331177} (Remarks \ref{R:rp} and \ref{R:cusp}).  The orbits arise as follows.  We begin with the identity coset $o_0 = P/P \in \check D$.  By definition $\cO_0 = D$ is the $G_\bR$--orbit of $o_0$, cf.~$\S$\ref{S:HSgZ}.  We will then define $o_j = g_j(o_0)$, where $g_j \in G$ and $\bc_j = \tAd_{g_j}$ is a composition of Cayley transforms ($\S$\ref{S:cayley}).  These Cayley transforms are determined by a suitable set \eqref{SE:dfnA} of strongly orthogonal noncompact roots.

\begin{remark}[Closed orbits] \label{R:c_orb}
In order for $\cO_s$ to be the unique closed orbit, it is necessary (but not sufficient) that $s$ be the real rank of $\fg_\bR$.
\end{remark}

\begin{remark}[Maximal parabolics] \label{R:maxP}
In the case that $P$ is maximal ($\S$\ref{S:maxP}) we will see that, with a few exceptions, there exists a sequence $\{ \cO_j\}_{j=0}^s$ with $s$ equal the real rank of $\fg_\bR$ (Lemma \ref{L:sorPmax}). 
\end{remark}

The construction is given in $\S$\ref{S:details}.  It is first necessary to review the representation theory underlying the construction.

\subsection{The elements $\ttH^\a \in \fh$} \label{S:Ha}

Given a root $\a$, define $\ttH^\a \in \fh$ by $\ttH^\a \in [\fg^\a,\fg^{-\a}]$ and $\a(\ttH^\a) = 2$.  Then
\begin{equation} \label{E:int}
  2\,\frac{(\b,\a)}{(\a,\a)} \ = \ \b(\ttH^\a) \ \in \ \bZ 
  \quad\hbox{for all } \ \a,\b \,\in\,\Delta \,,
\end{equation}
cf.~\cite{MR1153249}.\footnote{The reader consulting other references should beware that the $H_\a$ of \cite{MR496761, MR1920389} is \emph{not} our $\ttH^\a$.  To be precise, $\ttH^\a = \frac{2}{(\a,\a)}\,H_\a$ is the $Z_\a$ of \cite{MR496761} and the $H'_\a$ of \cite{MR1920389}.  However, our $\ttH^\a$ \emph{is} the $H_\a$ of \cite{MR1153249}, and the $h_\a$ of \cite{MR499562}.  (Ha ha!)}  In particular, $\ttH^\a$ is a grading element.  
Indeed, if $0 \le m_\a , n_\a \in \bZ$ are defined by the conditions that 
$$
  \b - m_\a \,\a \,,\ldots,\, \b + n_\a\,\a
$$
is the $\a$--string through $\b$, then $m_\a - n_\a = \b(\ttH^\a)$, \cf\cite[Proposition 2.29]{MR1920389}.  Moreover, if $r_\a$ denotes the the reflection in the root $\a$, then 
\begin{equation} \label{E:refl}
  r_\a(\b) \ = \ \b \,-\, \b(\ttH^\a)\,\a \,.
\end{equation}

\subsection{A rational form} \label{S:Qform}

Given a grading element $\ttT$, we may define an integral structure on $\fg$ as follows.  
\begin{center}
\emph{Fix a Cartan subalgebra $\fh \subset \fg$.}
\end{center}
The $\{\ttH^{\a_i} \ | \ 1\le i \le r \}$ span $\fh$.  Complete this to a \emph{Chevalley basis} 
$$
  \{ x^\a \ | \ \a \in \Delta(\fh) \} \,\cup\,
  \{ \ttH^{\a_1},\ldots,\ttH^{\a_r}\}
$$
of $\fg$, cf.~\cite[$\S$25.2]{MR499562}.  That is, $x^\a \in \fg^\a$ and the following properties hold: the Killing form satisfies $(\ttH^\a,\ttH^\a) >0$ and $(x^\a , x^{-\a}) > 0$ for all $\a \in\Delta$; and the Lie bracket satisfies
\begin{bcirclist}
\item $[x^\a , x^{-\a}] = \ttH^\a \in \tspan_\bZ\{ \ttH^{\a_1},\ldots,\ttH^{\a_r}\}$, for all $\a\in \Delta(\fh)$,
\item $[\ttH^\b , x^\a] \in \bZ x^\a$, for all $\a,\b\in\Delta(\fh)$,
\item if $\a,\b,\a+\b\in\Delta$ and $[x_\a,x_\b] = c_{\a,\b} x^{\a+\b}$, then $c_{-\a,-\b} = -c_{\a,\b} \in \bZ$.
\end{bcirclist}
We emphasize that the structure coefficients with respect to the Chevalley basis are all integers.  Fix
$$
  \bi \ \dfn \ \sqrtminusone \,,
$$
and set
\begin{eqnarray*}
  h^j & \dfn & \bi\,\ttH^{\a_j} \,,\\
  u^\a & \dfn & 
  \renewcommand{\arraystretch}{1.3}\left\{ \begin{array}{ll}
    x^\a-x^{-\a} & \hbox{ if } \a(\ttT) \hbox{ is even,}\\
    \bi (x^\a-x^{-\a}) & \hbox{ if } \a(\ttT) \hbox{ is odd;}
  \end{array} \right. \\
  v^\a & \dfn & 
  \renewcommand{\arraystretch}{1.3}\left\{ \begin{array}{ll}
    \bi (x^\a+x^{-\a}) & \hbox{ if } \a(\ttT) \hbox{ is even,}\\
    x^\a+x^{-\a} & \hbox{ if } \a(\ttT) \hbox{ is odd.}
  \end{array} \right.
\end{eqnarray*}
Define 
\begin{equation} \label{E:gZ}
  \fg_\bZ \ = \ \fk_\bZ \,\op \,\fk^\perp_\bZ
\end{equation}
by
\begin{equation}\label{E:dfn_kkperp}
\renewcommand{\arraystretch}{1.5}
\begin{array}{rcl}
  \fk_\bZ & \dfn & \tspan_\bZ\{ h^j \ | \ 1 \le j \le r \} \ \cup \ 
  \tspan_\bZ\{ u^\a , v^\a \ | \ \a(\ttT) \hbox{ even}\} \,,\\
  \fk_\bZ^\perp & \dfn & \tspan_\bZ\{ u^\a , v^\a \ | \ \a(\ttT) \hbox{ odd}\} \,.
\end{array}
\end{equation}
It is a straightforward exercise to confirm that $[\fg_\bZ , \fg_\bZ] \subset \fg_\bZ$ and $\fg = \fg_\bZ \ot \bC$; that is, $\fg_\bZ$ is an integral form of $\fg$.  It follows that the Killing form $B : \fg_\bZ \times \fg_\bZ \to \bZ$ is defined over $\bZ$.  Likewise,
$$
  [ \fk_\bZ \,,\, \fk_\bZ] \ \subset \ \fk_\bZ \,,\quad
  [ \fk_\bZ \,,\, \fk_\bZ^\perp ] \ \subset \ \fk_\bZ^\perp \tand
  [ \fk_\bZ^\perp \,,\, \fk_\bZ^\perp ] \ \subset \ \fk_\bZ \,,
$$
and one may confirm that the Killing form is negative definitive on $\fk_\bZ$ and positive definite on $\fk^\perp_\bZ$, so that \eqref{E:gZ} is a \emph{Cartan decomposition}.  Let $\theta : \fg_\bZ \to \fg_\bZ$ denote the corresponding \emph{Cartan involution}; that is,
\begin{equation} \label{E:theta}
  \left. \theta\right|_{\fk_\bZ} \ = \ \one \tand
  \left. \theta\right|_{\fk_\bZ^\perp} \ = \ -\one \,.
\end{equation}

\begin{definition} \label{def:MTC}
Let $G_{\bQ}$ be the $\bQ$--form of $G_\bR$ with Lie algebra $\fg_{\bQ}\dfn \fg_{\bZ}\otimes \bQ$.  When a $\bQ$-algebraic group arises in this way (from the above construction), we will call it a \emph{Mumford--Tate--Chevalley group} (MTC group).
\end{definition}

\begin{remark} \label{rem:MTC}
In general, if $G_{\bQ}$ is a ($\bQ$-algebraic)  MT group of a $\bQ$-Hodge structure underlying a semisimple real Lie group $G_\bR$, it need not be MTC.  Indeed, the above construction implies that MTC groups are split over $\bQ (i)$.\footnote{Of course, the construction could be modified to replace $\bQ (i)$ by any imaginary quadratic field.}  Unlike $\bQ$--Chevalley groups, they need not be split over $\bQ$, since $G_\bR$ need not be split over $\bR$; for instance, apply the construction to $\tSU(2,1)$.  However, we do have that $\fg^{\a}\oplus \fg^{-\a}$ is defined over $\bQ$ for each $\a\in \Delta$.
\end{remark}

A Cartan subalgebra of $\fg_\bR$ is given (over $\bZ$) by 
\begin{equation} \label{E:tcpt}
  \fh_\bZ \ \dfn \ \fh \,\cap\, \fg_\bZ \ = \ 
  \tspan_\bZ\{ \bi\,\ttH^1\,,\ldots,
  \bi\,\ttH^r \} \ \subset \ \fk_\bZ \,.
\end{equation}
From $\a_i(\ttS^j) = \d^i_j$ and $\a_i(\ttH^\b) \in \bZ$, we see that the $\ttS^i$ are $\bQ$--linear combinations of the $\ttH^j$.  Therefore, the $\bi\,\ttS^i$ are defined over $\bQ$.  In particular, the rational form is
\begin{equation} \label{E:toQ}
  \fh_\bQ \ = \ \tspan_\bQ\{ \bi\,\ttS^1\,,\ldots,
  \bi\,\ttS^r \}  \ \subset \ \fk_\bQ\,.
\end{equation}
Note that the real form
$$
  \ft \ \dfn \  \fh_\bR \quad \hbox{is compact.}
$$ 

Let $G_\bR \subset G$ be the real form of $G$ with Lie algebra $\fg_\bR$.  For later use, we introduce the subalgebra
\begin{equation} \label{E:sl2a}
  \fsl^\a_2(\bZ) \ \dfn \ \tspan_\bZ\{ u^\a \,,\, \bi\,\ttH^\a \,,\, v^\a\}
  \ \subset \ \fg_\bZ \,.
\end{equation}
Let 
\begin{equation} \label{E:SL2a}
  \tSL_2^\a(\bR) \ \subset \ G_\bR \tand
  \tSL_2^\a(\bC) \ \subset \ G
\end{equation}
be the connected Lie subgroups with Lie algebras $\fsl_2^\a(\bR)$ and $\fsl_2^\a(\bC)$, respectively.

\begin{example}[$\fg=\fsl_2\bC$] \label{eg:sl2a}
Consider the case that $\fg = \fsl_2\bC$, and let $\ttT = \ttS^1$.  We may take the Chevalley basis to be
$$
  x^\a \ = \ \left(\begin{array}{cc} 0 & 1 \\ 0 & 0 \end{array}\right)\,,\quad
  \ttH^\a \ = \ \left(\begin{array}{cc} 1 & 0 \\ 0  & -1 \end{array}\right)\,,\quad
  x^{-\a} \ = \ \left(\begin{array}{cc} 0 & 0 \\ 1 & 0 \end{array}\right)\,.
$$
Then the rational form \eqref{E:gZ} is spanned by 
$$
  u^\a \ = \ \bi\left(
             \begin{array}{cc} 0 & 1 \\ -1 & 0  \end{array}\right)\,,\quad
  h^1 \ = \ \bi \left(
             \begin{array}{cc} 1 & 0 \\ 0  & -1 \end{array}\right)\,,\quad
  v^a \ = \ \left(
             \begin{array}{cc} 0 & 1 \\ 1 & 0 \end{array}\right)\,.
$$
Observe that the real form defined by \eqref{E:gZ} is
$$
  \fg_\bR \ = \ \fsu(1,1) \ = \ 
  \left\{ X \in \fsl_2\bC \ | \ \overline{X}{}^t Q + Q X = 0 \right\} \,,
$$ 
where $Q = \left( \begin{array}{cc} 1 & 0 \\ 0 & -1 \end{array}\right)$.  This is the Lie algebra of the group 
$$
  G_\bR \ = \ \tSU(1,1) \ = \ \{ A \in \tSL_2\bC \ | \ A^* Q A = Q \} \,.
$$
preserving the Hermitian form $\bar z^t Q w$.
\end{example}

\subsection{Hodge structure on $\fg_\bZ$} \label{S:HSgZ}

If $\fg = \op \fg^p$ is the $\ttT$--eigenspace decomposition \eqref{SE:grading}, then \eqref{E:dfn_kkperp} yields
\begin{equation} \label{E:kkperp}
  \fk \ \dfn \ \fk_\bZ \ot \bC \ = \ \fg^\teven \tand
  \fk^\perp \ \dfn \ \fk_\bZ^\perp \ot \bC \ = \ \fg^\todd \,.
\end{equation}
Since $-(u,\theta\overline u) > 0$ for all $0\not= u \in \fg$, it follows that $\fg^{p,-p} = \fg^p$ defines a weight zero, $-(\cdot,\cdot)$--polarized Hodge structure on $\fg$.  (We regard $\ttT$ as an infinitesimal Hodge structure; it is the (rescaled) derivative of a Hodge structure $\varphi_0 : S^1 \to \tAd(G_\bR)$, \cf\cite[$\S$2.3]{MR3217458}.)   The corresponding \emph{Hodge flag} $F^\sb$ is given by 
$$
  F^p \ = \ \bigoplus_{q\ge p} \fg^{q,-q} \,.
$$
Note that $P = \tStab_G(F^\sb)$.  This allows us to identify with the flag $F^\sb$ with the point 
\begin{equation} \nonumber 
  o_0 \ \dfn \ P/P \ \in \ G/P \,.
\end{equation}
Define 
$$
  \check D \ \dfn \ G/P \,,
$$
and let
\begin{equation} \nonumber 
  D  \ \dfn \ G_\bR \cdot o_0 \,.
\end{equation}
Then $D$ is open in $\check D$, and therefore a flag domain ($\S$\ref{S:mtd}).

\begin{definition} \label{def:MTCD}
A Mumford--Tate domain arising in this way is called a \emph{Mumford--Tate--Chevalley domain} (MTC domain).
\end{definition}

\begin{remark} \label{rem:MTCD}
From the Hodge--theoretic perspective, $D$ is the MT domain parameterizing weight zero Hodge structures on $\fg_\bQ$ that are polarized by $-(\cdot,\cdot)$ and have MT group contained in $G_\bQ$.  By construction, this $G_{\bQ}$ is a MTC group.  Moreover, the grading element $\ttT$ associated with $o_0$ is defined over $\bQ (i)$ and is purely imaginary.  The MT group of $o_0$ is therefore a $1$-torus with real points $\varphi_0 (S^1)$. 

In general, a MT domain (determined by Hodge--theoretic data) which is isomorphic to $D$ as a $G_\bR$-homogeneous space, need not be MTC:  this is a property which reflects the arithmetic of the $\bQ$-Hodge tensors.  The MTC case may be thought of as ``maximizing'' the density of MT subdomains in $D$.
\end{remark}

\begin{example}[$\fg = \fsl_2\bC$] \label{eg:sl2b}
This is a continuation of Example \ref{eg:sl2a}.  We have $G = \tSL_2\bC$ and fix the parabolic subgroup 
$$
  P \ = \ 
  \left\{ \left. \left( \begin{array}{cc} a & b \\ 0 & 1/a \end{array} \right) 
  \ \right| \ a,b \in \bC \,,\ a \not=0 \right\} \,.
$$
Let $\check D$ be the complex projective line $\bC\bP^1$.  Then the identity coset $P/P$ corresponds to the point 
$$
  o_0 \ = \ (1:0) \ \in \ \bC\bP^1 \,.
$$
The $G_\bR = \tSU(1,1)$--orbit 
$$
  D \ = \ \{ (1:z) \in \bC\bP^1 \ | \ |z| < 1 \}
$$
of $o_0$ is naturally identified with the unit disc in $\bC$.
\end{example}

\subsection{Standard triples} \label{S:sl2trp}

Let $\fg_\bbk$ be a Lie algebra defined over a field $\bbk$.  A \emph{standard triple} in $\fg_\bbk$ is a set of three elements $\{ N^+ , Y , N\} \subset \fg_\bbk$ such that 
$$
  [Y , N^+] \ = \ 2\,N^+ \,,\quad
  [N^+,N] \ = \ Y \tand
  [Y,N] \ = \ -2\,N \,.
$$
Note that $\{N^+,Y,N\}$ span a three--dimensional semisimple subalgebra (TDS), necessarily isomorphic to $\fsl_2$.  We call $Y$ the \emph{neutral element}, $N$ the \emph{nilnegative element} and $N^+$ the \emph{nilpositive element}, respectively, of the standard triple.

\begin{example} \label{eg:stdtri_1}
The matrices 
\begin{equation} \label{E:stdtri_sl2}
  N^+ \ = \ \left(\begin{array}{cc} 0 & 1 \\ 0 & 0 \end{array}\right) \,,\quad 
  Y \ = \ \left(\begin{array}{cc} 1 & 0 \\ 0 & -1 \end{array}\right) \tand
  N \ = \ \left(\begin{array}{cc} 0 & 0 \\ 1 & 0 \end{array}\right)
\end{equation}
form a standard triple in $\fsl_2\bR$; while the matrices
\begin{equation} \label{E:stdtri_su11}
  N^+ \ = \ \half \left(\begin{array}{cc} \bi & 1 \\
            1 & -\bi \end{array}\right) \,,\quad
  Y \ = \ \left(\begin{array}{cc} 0 & \bi \\ 
           -\bi & 0 \end{array}\right) \tand
  N = \half \left(\begin{array}{cc} -\bi & 1 \\ 
            1 & \bi \end{array}\right)
\end{equation}
form a standard triple in $\fsu(1,1)$.
\end{example}

\subsection{Cayley transforms} \label{S:cayley}

This is a brief review of Cayley transforms; see \cite[Chapter VI.7]{MR1920389} for details.  

Let $\fg_\bR = \fg_\bZ \ot \bR$ be the real form of $\fg$ constructed in $\S$\ref{S:Qform}.  By \eqref{E:toQ}, all roots are imaginary.  Fix a noncompact positive root $\a$.  That is, $\fg^\a \subset \fk^\perp$; equivalently, $\a(\ttE)$ is odd.  Then 
$$
  \overline{x^\a} \ = \ x^{-\a} \,.
$$
The Cayley transform associated to $\a$ is 
\begin{equation} \label{E:ca}
  \bc_\a \ \dfn \ \tAd\left(\texp\,\tfrac{\pi}{4} (x^{-\a} - x^\a) \right) 
  \ \in \ \tAd(G) \,.
\end{equation}

Since $\tAd(g)$ is a Lie algebra automorphism for any $g \in G$, it follows that $\bc_\a$ is a Lie algebra automorphism.  Therefore, 
$$
  \fh' \ = \ \bc(\fh)
$$
is a $\theta$--stable Cartan subalgebra.  The root spaces of $\fh'$ are the Cayley transforms ${}'\fg^\a =\bc(\fg^\a)$ of the root spaces of $\fh$.  Likewise, ${}'\ttH^\b = \bc_\a(\ttH^\b)$.  In particular,
$$
  {}'\ttH^\a \ = \ \bc(\ttH^\a) \ = \ 
  x^\a \,+\, x^{-\a} \ = \ v^\a \ \in \ \fk^\perp_\bZ\,.
$$
The real form $\fh'_\bR = \fh' \,\cap\,\fg_\bR$ of $\fh'$ is
\begin{equation} \label{E:t}
  {}'\fh_\bR \ \dfn \ \bc(\fh) \,\cap\, \fg_\bR 
  \ = \ \tker\left\{ \a : \ft \to \bi\,\bR \right\} 
  \ \op \ \tspan_\bR\{ {}'\ttH^\a \} \,.
\end{equation}
Additionally, we have
\begin{equation} \nonumber 
  \bc_\a(x^\a - x^{-\a}) \,=\, x^\a - x^{-\a} \tand
  \bc_\a(x^\a + x^{-\a}) \,=\, - \ttH^\a \,.
\end{equation}
Therefore, the root space ${}'\fg^{-\a} = \bc(\fg^{-\a})$ is spanned by 
\begin{equation} \label{E:N}
  y^{-\a} \ \dfn \ \bi \,\bc_\a( x^{-\a} ) \ = \ 
  \tfrac{\bi}{2}( x^{-\a} - x^\a - \ttH^\a) 
  \ \in \ \fg_\bQ \,.
\end{equation}  
Likewise, ${}'\fg^\a$ is spanned by 
\begin{equation}\label{E:N+}
  y^\a \ = \ -\bi \,\bc_\a(x^\a) 
  \ = \ \tfrac{\bi}{2}\left( x^{-\a} - x^\a + \ttH^\a \right)
  \ \in \ \fg_\bQ \,.
\end{equation}
From $[x^\a , x^{-\a}] = \ttH^\a$ and $\bc_\a(\ttH^\a) = [\bc_\a(x^\a) , \bc_\a(x^{-\a})]$, we see that 
\begin{equation}\label{E:sl2Q}
  \hbox{$\{ y^\a \,,\, '\ttH^\a \,,\, y^{-\a}\}$ is a standard triple
  defined over $\bQ$,}
\end{equation}
\cf $\S$\ref{S:sl2trp}.  Recall \eqref{E:sl2a} and note that 
$$
  \fsl^\a_2(\bQ) \ = \ \tspan_\bQ\{ y^\a \,,\, '\ttH^\a \,,\, y^{-\a} \} \,.
$$

\begin{example}[$\fg = \fsl_2\bC$] \label{eg:sl2c}
This is a continuation of Examples \ref{eg:sl2a} and \ref{eg:sl2b}.   We have
$$
  \bc_\a \ = \ \frac{1}{\sqrt{2}}\left( \begin{array}{cc}
    1 & \bi \\ \bi & 1 
  \end{array} \right) \,,
$$  
and
$$
  y^\a \ = \ \frac{\bi}{2}\left(
             \begin{array}{cc} 1 & -1 \\ 1 & -1  \end{array}\right)\,,\quad
  {}'\ttH^\a \ = \ \left(
             \begin{array}{cc} 0 & 1 \\ 1 & 0 \end{array}\right)\,,\quad
  y^{-\a} \ = \ \frac{\bi}{2}\left(
             \begin{array}{cc} -1 & -1 \\ 1 & 1 \end{array}\right)\,.
$$
\end{example}


\subsection{The weight zero $\a$--primitive subspace} \label{S:a-prim}

Let 
$$
  \Gamma_\a^0 \ \dfn \ \{ \z \in \fg \ | \ [\fsl_2^\a,\z] = 0 \}
$$
be the trivial isotypic component in the $\fsl_2^\a$--module decomposition of $\fg$; this is the \emph{weight zero, $\a$--primitive} subspace of $\fg$.  Note that 
\begin{center}
  \emph{the Cayley transform $\bc_\a$ restricts to the identity on $\Gamma_\a^0$.}
\end{center}

The Jacobi identity implies $\Gamma^0_\a$ is a subalgebra of $\fg$.  More precisely,
$$
  \Gamma_\a^0 \ = \ \tker\{\left.\a\right|_{\fh}\}
  \ \op \ \bigoplus_{\b\in\Delta(\Gamma_\a^0)} \fg^\b \,,
$$
where
$$
  \Delta(\Gamma_\a^0) \ = \ \{ \b \in \Delta \ | \ \b\pm\a \not\in \Delta \} 
$$
is the set of roots that are strongly orthogonal to $\a$.  In particular, 
\begin{center}
\emph{$\Gamma^0_\a$ is a reductive subalgebra defined over $\bQ$.}
\end{center}
Let 
$$
  \Gamma_\a \ \dfn \ [\Gamma_\a^0,\Gamma_\a^0]
$$
be the semisimple component, and let $\sG \subset G$ be the connected Lie subgroup with Lie algebra $\Gamma_\a$.  Note that $\sP = \sG\,\cap\,P$ is a parabolic subgroup, and let
$$
  \check \sD \ \dfn \ \sG/\sP
$$ 
be the associated rational homogeneous variety.  Let $o = \sP/\sP \in \check D$ be the identity coset.  The real form $\Gamma_\a(\bR) \,\cap\,\fh \subset \tker\{\left.\a\right|_{\ft}\}$ of the Cartan subalgebra is contained in $\ft$, and therefore is compact.  Whence, \begin{center}
\emph{the $\sG(\bR)$--orbit $\sD$ of $o$ is open in $\check\sD$.}
\end{center}

\subsection{The construction} \label{S:details}

Recall the Cartan subalgebra $\fh\subset \fg$ with \emph{compact} real form $\ft$  of $\S$\ref{S:Qform}.  Fix a set 
\begin{subequations} \label{SE:dfnA}
\begin{equation} \label{E:A}
  \cB \ = \ \{ \b_1 , \ldots , \b_s\} \ \subset \ \fh^*
\end{equation}
of strongly orthogonal noncompact roots with the property that
\begin{equation} \label{E:a=1}
  \b_j(\ttE) \ = \ 1 \,.
\end{equation}
\end{subequations}
Given $1 \le j \le s$,  let $\bc_{\b_j}$ be the Cayley transformation associated to $\b_j$ ($\S$\ref{S:cayley}).  Set 
\[
  \bc_j \ \dfn \ \bc_{\b_j} \circ \cdots \circ \bc_{\b_1} \,.
\]
(Strong orthogonality implies that the roots $\{\b_{j+1},\ldots,\b_s\}$ are still imaginary and noncompact with respect to $\fh_j$.  Whence, $\bc_j$ is well-defined.)   By \eqref{E:ca}, the Cayley transform $\bc_\a$ may be written as $\bc_\a = \tAd_{g_\a}$ where $g_\a = \texp \frac{\pi}{4}(x^{-\a} -x^\a) \in G$.  Let $g_j = g_{\b_j}\cdots g_{\b_1}$, so that $\bc_j = \tAd_{g_j}$.  It will be convenient to also set $g_0 = \one$ and $\bc_0 = \tAd_\one = \one$.

Recall the weight zero polarized Hodge structure of $\S$\ref{S:HSgZ}.  This Hodge structure is naturally identified with the identity coset $o_0 = P/P \in \check D$.  Let $\cO_0 = D$ be the the $G_\bR$--orbit of $o_0$.   Given $1 \le j \le s$, set 
\[
  o_j \ \dfn \ g_j(o_0) \ \in \ \check D\,,
\] 
and let $\cO_j \subset \check D$ be the $G_\bR$--orbit of $o_j$.  (We will also write $\cO_{\cB} := \cO_s$, as in the Introduction.) The containment \eqref{E:orbit} follows from this construction, cf.~\cite{MR3331177}.

\begin{remark}[Hermitian symmetric case] \label{R:HSDcase}
Note that $\bc_j$ is the Cayley transform $c_b$ of \cite[$\S$1.6]{\BB66}.  Therefore, in the case that $D$ is Hermitian symmetric, the $\cO_j$ account for \emph{all} the $G_\bR$--orbits in the boundary of $D$.  That is, 
\[
  \overline D \ = \ \bigcup_{j} \cO_j \,.
\]
\end{remark}

\begin{example}[$\fg = \fsl_2\bC$] \label{eg:sl2d}
This is a continuation of Examples \ref{eg:sl2a}, \ref{eg:sl2b} and \ref{eg:sl2c}.  We have 
$$
  o_1 \ = \ (1 : \bi ) \ \in \ \bC\bP^1 \,,
$$  
and 
$$
  \cO_1 \ = \ \{ (1:z) \ | \ |z|=1 \}
$$
is naturally identified with the unit circle in $\bC$.  If we define 
$$
  \bc_\a(t) \ \dfn \ \tAd\left(\texp\,\tfrac{\pi}{4} t (x^{-\a} - x^\a) \right) 
  \ = \ \left( \begin{array}{cc}
    \cos \tfrac{\pi}{4} t & \bi\,\sin \tfrac{\pi}{4} t \\
    \bi\,\sin \tfrac{\pi}{4} t & \cos \tfrac{\pi}{4} t 
  \end{array} \right) \,,
$$  
then the Cayley transform is $\bc_\a = \bc_\a(1)$.  Moreover, the curve $\bc_\a(t)\cdot o_0$ is contained in the unit disc $\cO_0$ for all $0 \le t < 1$.
\end{example}

Associated with the Matsuki point $o_j$ are two filtrations \eqref{E:FW} of $\fg$, which we now describe.  Set $\fh_0 = \fh$ and define
\begin{equation} \nonumber
  \fh_j \ \dfn \ \bc_j(\fh) \,.
\end{equation}
Then $\tdim_\bR (\fh_j \cap \fk^\perp_\bR) = j$.  Thus 
\begin{equation} \label{E:s'}
  s \ \le \ \trank_\bR\,\fg_\bR \,.
\end{equation}

Recall the elements $\ttH^\a \subset \fh_j$ of $\S$\ref{S:Ha},\footnote{Properly speaking, $\ttH^\b = \ttH^\b_j$ depends on $\fh_j$.  Precisely, if $\ttH^\b_0$ is defined with respect to $\fh \subset \fb \subset \fp$, then $\ttH^\b_j = \bc_j(\ttH^\b_0)$.  We have elected to drop the subscript $j$ to keep the notation clean and simply write $\ttH^\b \in \fh_j$.} and  define
\begin{equation} \label{E:ttYj}
  \ttY_j \ \dfn \ \ttH^{\b_1} + \cdots + \ttH^{\b_j} \,.
\end{equation}
Let $\fg = \op\,{}_j\fg^p$ be the $\ttE$--eigenspace\footnote{Again, $\ttE = \bc_j(\ttE_0)$, where $\ttE_0$ is the grading element associated to $\fh \subset \fb \subset \fp$.} decomposition \eqref{SE:grading}, and let $\fg = \op\,{}_j\fg_\ell$ be the $\ttY_j$--eigenspace decomposition.  That is,
\begin{eqnarray} 
  \label{E:gp}
  {}_j\fg^p & \dfn &  \{ \xi \in \fg \ | \ [\ttE,\xi] = p \,\xi \} \,,\\
  \label{E:gl}
  {}_j\fg_\ell & \dfn & \{\xi \in \fg \ | \ [\ttY_j , \xi] = \ell \,\xi \} \,.
\end{eqnarray}
By \eqref{E:int}, the $\ttY_j$--eigenvalues are integers.  As commuting semisimple endomorphisms, $\ttE$ and $\ttY_j$ are simultaneously diagonalizable; therefore,
\begin{subequations} \label{SE:gpq}
\begin{equation} \label{E:gpq}
  {}_j\fg^{p,q} \ \dfn \ {}_j\fg^p \,\cap\, {}_j\fg_{p+q} \ = \
  \left\{ \xi \in \fg \ | \ [\ttE , \xi] = p \xi \,,\ 
  [\ttY_j, \xi] = (p+q) \xi \right\}
\end{equation}
defines a direct sum decomposition 
\begin{equation} 
  \fg \ = \ \bigoplus {}_j\fg^{p,q} \,.
\end{equation}
\end{subequations}
By the Jacobi identity, this is a bigraded Lie algebra decomposition; that is, 
\begin{equation} \nonumber 
  [ {}_j\fg^{p,q} , {}_j\fg^{r,s}] \ \subset \ {}_j\fg^{p+r,q+s} \,.
\end{equation}
Define filtrations
\begin{equation} \label{E:FW}
  {}_jF^p \ \dfn \ \bigoplus_{q\ge p} {}_j\fg^{q}
  \tand {}_jW_\ell \ \dfn \ \bigoplus_{p+q \le \ell} {}_j\fg^{p,q}
  \ = \ \bigoplus_{k \le \ell} {}_j\fg_k\,.
\end{equation}
Note that $\tStab_G({}_jF^\sb) = g_j P g_j^{-1} = \tStab_G(o_j)$, so that there is a natural identification of the flag ${}_jF^\sb$ with the Matsuki point $o_j \in \check D$.

\begin{theorem} \label{T:matsuki}
For each $0 \le j \le s$, there exists a $j$--dimensional, rational nilpotent cone $\s_j \subset \fg_\bQ$ such that $\{\texp(\bC\,\s_j)\cdot o_j\} \in B(\s_j)$ is a nilpotent orbit.\footnote{See $\S$\ref{S:B+D} for the definition of $B(\sigma)$.}  The weight filtration $W_\sb(\s_j)$ is the filtration ${}_jW_\sb$ of \eqref{E:FW}.  The cone $\s_{j-1}$ is a face of $\s_j$, for all $1 \le j \le s$.
\end{theorem}

\begin{remark} \label{R:LMHS}
As a corollary to Theorem \ref{T:matsuki}, $(\fg_\bQ , {}_j F^\sb, {}_j W_\sb)$ is a (limit) $\bQ$--mixed Hodge structure.
\end{remark}

\begin{remark}[Rational and polarized] \label{R:rp}
In the language of \cite[$\S$5]{MR3331177}, Theorem \ref{T:matsuki} implies that the orbit $\cO_j$ is polarized and contained in the nilpotent closure of the open $G_\bR$--orbit $D = \cO_0$.  Moreover, from $\s_j \subset \fg_\bQ$ it follows that ${}_jW_\sb$ is defined over $\bQ$, so that $\cO_j$ is rational, \cf\cite[Definition 5.5]{MR3331177}.  
\end{remark}

Recall the $\ttY_j$--eigenspace ${}_j\fg_0$ defined in \eqref{E:gl}.  The Jacobi identity implies that ${}_j\fg_0$ is a Lie subalgebra of $\fg$.  In fact, ${}_j\fg_0$ is reductive.  Let ${}_j\fg_0^\tss = [{}_j\fg_0\,,\,{}_j\fg_0]$; then ${}_j\fg^\tss_0$ is semisimple, and ${}_j\fg_0 = {}_j\fg_0^\tss \op \fz$, where $\fz$ is the center of ${}_j\fg_0$.  Following \cite[Definition 5.9]{MR3331177}, we say that the $G_\bR$--orbit $\cO_j$ is \emph{cuspidal} if ${}_j\fg_0^\tss(\bR)$ contains a compact Cartan subalgebra; that is, a Cartan subalgebra on which the Killing form restricts to be negative definite.

\begin{remark}[Weakly cuspidal] \label{R:cusp}
The orbit $\cO_j$ is cuspidal in the following weak sense.  In the course of the proof we will show that there exist semisimple subgroups $G_{s} \subset \cdots\subset G_1 \subset G_0 = G$ with the properties that:
\begin{bcirclist}
\item The orbit $\check D_j = G_j \cdot o_j$ is a rational homogeneous variety containing the point $o_{j+1}$.
\item The orbit $D_j = G_{j,\bR}\cdot o_j$ is open in $\check D_j$ and equals $\check D_j\,\cap\,\cO_j$; and $o_{j+1} \in \del D_j$.
\item  The orbit $\cC_{j+1} = G_{j,\bR}\cdot o_{j+1} \subset \check D_j\,\cap\,\cO_{j+1}$ is contained in the boundary of $D_j$ and is cuspidal in $\check D_j$.  (This will follow from Lemma \ref{L:O1c} below.)  Note that
$$
  D_{j+1} \ \subset \ \cC_{j+1} \ \subset \ \del D_j\,.
$$
\end{bcirclist}
\end{remark}

\begin{remark}[Sub--Hodge structures]
The subdomains $\{ D_j\}_{j=0}^s$ may be described in greater detail.  In the course of the proof we will show that:
\begin{bcirclist}
\item 
The Lie algebra of $G_j$ is 
\begin{equation}\label{E:Gammaj}
  \Gamma_j \ = \ \bigcap_{1 \le k \le j}\,\Gamma_{\b_i} \,,
\end{equation}
where the $\Gamma^{\b_i}$ are as in $\S$\ref{S:a-prim}.  
\item
The set $\{ \texp(\bC\,N_j) \cdot o_j \}$ is a nilpotent orbit in $\check D_{j-1}$, for all $1 \le j \le s$, \cf\eqref{E:norb_j}.
\end{bcirclist}
These two observations, and results of \cite{KP2012}, yield a natural identification of the open $G_{j,\bR}$--orbit $D_j \subset \check D_j$ as the MT domain of a Hodge structure $\varphi_{j,\sb}$, which we now describe.  Let $W_\sb(N_j)$ denote the weight filtration on $\Gamma_{j-1}$, with the convention that $\Gamma_0 = \fg$.  Let $\tGr_\sb(W_\sb(N_j))$ denote the associated graded and choose $0 \not= N_j^+ \in \fg^{\b_j}$.  Then, given $\ell \ge0$, the filtration ${}_jF^\sb$ induces a weight $\ell$ Hodge structure $\varphi_{j,\ell}$ on $\tker\{ (N_j^+)^{\ell+1} : \tGr_\ell(W_\sb(N_j)) \to \tGr_{-\ell-2}(W_\sb(N_j))\}$ that is polarized by $-(\cdot,N_j^\ell\cdot)$.\footnote{In the notation of \cite{KP2012}, the domain $D_j$ is denoted $D(N_j)$, and the Hodge structure $\varphi_{j,{\tiny{\bullet}}}$ by $\varphi_\mathrm{split}$.}
\end{remark}

Theorem \ref{T:matsuki} and the lemmas that follow are proved in $\S$\ref{S:proofs}.  The first three lemmas codify some distinguished properties of the bigrading \eqref{SE:gpq} and the Cartan subalgebra $\fh_j$.

\begin{lemma} \label{L:conj+Ci}
The bigrading \eqref{SE:gpq} satisfies the symmetries
\begin{subequations} \label{SE:conj+Ci}
\begin{equation} \label{E:bar-pq}
  {}_j\overline{\fg^{p,q}} \ = \ {}_j\fg^{q,p} 
\end{equation}
and
\begin{equation} \label{E:theta-pq}
  \theta( {}_j\fg^{p,q}) \ = \ {}_j\fg^{-q,-p} \,.
\end{equation}
\end{subequations}
\end{lemma}

\begin{remark}\label{R:deligne}
By \eqref{E:FW}, \eqref{SE:gpq} is the Deligne bigrading \eqref{SE:deligne} on the $\bQ$-MHS $(\fg_{\bQ},{}_j F^\sb,{}_j W_{\sb} )$.  It follows from \eqref{E:bar-pq} that this MHS is $\bR$-split, but in fact more is true.  Since $\ttY_j\in \fg_{\bQ}$, $(\fg_{\bQ},{}_j F^\sb,{}_j W_{\sb} )$ is $\bQ$-split; and by \eqref{E:Gammaj}, the other MHS in its $G_{j,\bR}$-orbit remain $\bQ$-split.
\end{remark}

\begin{lemma}\label{L:EbarE}
For $\ttE$ and $\ttH^\b$ defined with respect to the Cartan subalgebra $\fh_j$, we have
$$
  \ttE \, + \, \overline{\ttE} \ = \ \ttY_j \,.
$$
\end{lemma}

\noindent As a consequence, ${}_j\fg^{p,q} = {}_j\fg^p \,\cap\,{}_j\overline{\fg^q}$. 

\begin{lemma} \label{L:bar-b}
The roots $\a \in \Delta(\fg , \fh_j)$ of the Cartan subalgebra $\fh_j$ satisfy
\begin{equation}\label{E:bar-b}
  \overline\a \ = \ -\a \,+\,\sum_{i=1}^j \a(\ttH^{\b_i})\,\b_i \,.
\end{equation}
\end{lemma}

\noindent Consider Lemma \ref{L:bar-b} in the case that $j=1$.  Then \eqref{E:refl} and \eqref{E:bar-b} yield $\overline\a = -r_{\b_1}(\a)$.  More generally, if $j > 1$, then the strong orthogonality of the roots $\cB$ implies $\bar\a = -r_{\b_j} \cdots r_{\b_1}(\a)$.

The following lemma will be invoked in the proof of Theorem \ref{T:matsuki}, which is by induction.  

\begin{lemma} \label{L:O1p}
There exists $N_1 \in \fg^{-\b_1}_\bQ$ such that $\{ \texp(\bC\,N_1)\cdot o_1 \} \in B(N_1)$ is a nilpotent orbit.  \emph{(In particular, $\cO_1$ is rational and polarized.)}
\end{lemma}

\begin{lemma} \label{L:O1c}  
The $G_\bR$--orbit $\cO_1$ is cuspidal.
\end{lemma}

\begin{lemma} \label{L:codimO1}
If the first root $\b_1\in\cB$ is simple, then the $G_\bR$--orbit $\cO_1$ has codimension one in $\check D$.
\end{lemma}

\noindent \cite[Appendix A]{KR1long} presents a thorough study of these codimension one boundary orbits in the case that $P$ is a maximal parabolic.

\begin{remark}[$\tcodim\,\cO_1$] \label{R:codimO1}
The codimension of $\cO_1$ is the number of roots in 
\[
  \Delta(+,+) \ = \ \{ \a\in \Delta \ | \ \a(\ttE) \,,\ \bar\a(\ttE) \ge 1 \}\,,
\] 
\cf\cite[Proposition 4.1]{MR3331177}.  By Lemma \ref{L:bar-b}
\[
  \Delta(+,+) \ = \  \{ \a\in \Delta \ | \ 1 \le \a(\ttE) < \a(\ttH^{\b_1}) \} \,,
\]  
and $\bar\b_1 = \b_1 \in \Delta(+,+)$.  
\end{remark}

\begin{example}[$\tcodim\,\cO_1 > 1$] \label{eg:codimO1}
In the case that $\b_1$ is not simple one may have $\tcodim_\bR\cO_1 > 1$. If $\check{D}=G_2/P_2$ (i.e. $\ttE_i = \ttS^2$) and $\beta_1 = 2\alpha_1 + \alpha_2$, then $\ttH^{\beta_1}=\ttS^1$ and $\Delta(+,+)=\{ 2\alpha_1 + \alpha_2, 3\alpha_1+\alpha_2, 3\alpha_1+ 2\alpha_2\}$, so that $\tcodim\,\cO_1=3$. See Example 6.54 of \cite{KR1long} for further instances.
\end{example}

\begin{remark}[Relationship to work of Kerr--Pearlstein]
Theorem \ref{T:matsuki} is closely related to the discussion and results of \cite{MR3331177}.  Indeed, Lemma \ref{L:EbarE} implies that the bigrading \eqref{SE:gpq} is that of \cite[Lemma 3.2]{MR3331177}.  One may also view $o_j={}_j F^\sb$ as the image of a $\bQ$-split point in a boundary component under the \naive~ limit map, as we explain briefly in $\S$\ref{S:B+D}.

\end{remark}

\subsection{The case that $P$ is a maximal parabolic} \label{S:maxP}

Assume that $G$ is simple and $P$ is a maximal parabolic. Then the grading element \eqref{E:ttE} is of the form 
$$
  \ttE \ = \ \ttS^\tti \,.
$$

\begin{proposition} \label{P:maxP}
Let $P$ be a maximal parabolic in a simple Lie group $G$.  Fix a real form $G_\bR$ and an open $G_\bR$--orbit $D \subset G/P$ admitting the structure of a Mumford--Tate domain.  Then $\tbd(D) \subset \check D$ contains a unique codimension one $G_\bR$--orbit.
\end{proposition}

\noindent The proposition is proved in $\S$\ref{S:proofs}.  A generalization of Proposition \ref{P:maxP} holds for arbitrary (not necessarily maximal) parabolics: the number of codimension one orbits in $\tbd(D)$ is $|I(\fp)|$, \cf\cite{SL2}.

\begin{remark} \label{rem:openGR}
(i) That an open $G_\bR$--orbit admit the structure of a  MT domain is a strong constraint.  Given $o_0\in D$ of this type, with corresponding grading element $\ttE\in\fh\subset\fg^0$, we have 
\begin{equation}\label{E:DcDn} 
  \ttE(\Delta_c)\subset 2\mathbb{Z} \quad\tand\quad
  \ttE(\Delta_n)\subset 2\mathbb{Z}+1\,,
\end{equation}
where $\Delta_c$ and $\Delta_n$ denote the compact and noncompact roots, respectively \cite{\GGK}.  Applying the complex Weyl group $W(\fg,\fh)$ to $\ttE$ yields points $w(o_0)$ in each open orbit; more precisely, it produces a bijection between double cosets in $W(\fg_{\bR},\fh_{\bR})\setminus W(\fg,\fh)/W(\fp,\fh)$ and open orbits \cite[Theorem  4.5(ii)]{MR3331177}.  The orbit containing $w(o_0)$ is a MT domain if and only if $w(\ttE)$ satisfies \eqref{E:DcDn}. This holds in general, for arbitrary parabolics $P$. When $P$ is maximal, $D$ is often unique (or nearly so), and in this case Proposition \ref{P:maxP} is useful for reducing the candidates using orbit incidence data.

(ii) Furthermore, a maximal parabolic $P_{\tti}$ only stabilizes a filtration $F^{\bullet}$ on $\fg_{\bC}$ induced by $\ttE=\ttS^{\tti}$. If $\check{D}=G/P_{\tti}$ is a fundamental adjoint variety, then we have $\dim_{\bC} {Gr}_F^{-2} \fg_{\bC} = 1$ and $\dim_{\bC} {Gr}_F^{-2k}\fg_{\bC}=0$ ($\forall k>1$), so that \eqref{E:DcDn} forces $\dim_{\bC}(K_{\bC}\cdot o_0)=1$. Running $\mathtt{kgp}$ in ATLAS \cite{ATLAS}, one finds (from this constraint alone) that in each of the fundamental adjoint cases $D$ is indeed unique.
\end{remark}


As noted in \eqref{E:s'}, the number of elements $s$ in the set \eqref{SE:dfnA} is bounded above by the real rank of $\fg_\bR$.

\begin{lemma} \label{L:sorPmax}
\emph{(a)} 
There exists a choice of $\cB$ that both maximizes $s$ and contains the simple root $\a_\tti$.
\smallskip\\
\emph{(b)} 
Excepting the pairs 
\begin{equation} \label{E:excp}
  (\fg,\tti) \ = \  (\fe_7,5) \,,\, (\fe_8;2,5,6)\,,\, 
  (\ff_4,2) \,,\, (\fg_2,1) \,,
\end{equation}
there exists a set \eqref{SE:dfnA} with $s$ equal to the real rank of $\fg_\bR$. 
\end{lemma}

\begin{proof}
For most pairs $(\fg,\tti)$ the sequence $\{\b_j\}_{j=1}^s$ is exhibited in \cite[Appendix C.3-4]{MR1920389}.  The remainder may be identified from inspection of the root system.
\end{proof}

\begin{corollary}\label{C:sorthset}
Excepting the pairs \eqref{E:excp}, the Cartan subalgebra $\fh_s$ has maximally noncompact real form $\fh_s \cap \fg_\bR$.
\end{corollary}

\begin{remark}[Exceptional cases] \label{R:excp}
For each of the exceptional cases of \eqref{E:excp} we may choose the set $\{\b_j\}_{j=1}^s$ as follows.  Below we express the root $\b = \sum_i b^i\a_i$ as $(b^1\cdots b^r)$.
\begin{bcirclist}
\item 
$\cB(\fe_7,5) = \{ (0^4 1 0^2) \,,\, (0^3 1^3 0) \,,\, (0 1^2 21 0^2) \,,\, (0 1^5 0) \,,\, (0101^4) \,,\, (0^2 1^5) \}$. \\
Here $s = 6$, while $\trank_\bR\,\fg_\bR = 7$.   
\item 
$\cB(\fe_8,2) = \{ (010^6) \,,\, (01^2 210^3) \,,\, (1^3 21^2 0^2) \,,\, (1^2 2^3 1 0^2) \,,\, (1^2 2^2 1^3 0) \,,\, (1^3 2^2 1^2 0) \,,\,(01^2 2^3\,10)\}$.\\
Here $s=7$, while $\trank_\bR\,\fg_\bR = 8$.         
\item 
$\cB(\fe_8,5) = \{ (0^4 1 0^3)\,,\,(0^3 1^3 0^2)\,,\,(0 1^2 2 1 0^3)\,,\,(0 1^5 0^2)\,,\,(010 1^4 0)\,,\,(0^2 1^5 0) \}$.\\
Here $s=6$, while $\trank_\bR\,\fg_\bR = 8$.         
\item 
$\cB(\fe_8,6) = \{ (0^5 1 0^2)\,,\,(0^4 1^3 0)\,,\,(0^3 1^5)\,,\,(0 1^2 2^2 1 0^2)\,,\,(0 1^2 2 1^3 0)\,,\,(0 1^7) \}$.\\
Here $s=6$, while $\trank_\bR\,\fg_\bR = 8$.         
\item 
$\cB(\ff_4,2) = \{ (0100) , (1110) , (0120) \}$.\\
Here $s = 3$, while $\trank_\bR\,\fg_\bR = 4$.
\item 
$\cB(\fg_2,1) = \{ (10) \}$.\\
Here $s = 1$, while $\trank_\bR\,\fg_\bR = 2$.
\end{bcirclist}
Outside these cases, it is natural to ask whether all the orbits in $\tbd(D)$ are of the form $\mathcal{O}_{\mathcal{B}}$ ($1\leq s\leq \text{rk}_{bR}(G)$), hence also polarizable. This is verified for all the fundamental adjoint varieties in $\S$\ref{S:adIII}.
\end{remark}

\subsection{Proofs} \label{S:proofs}

\begin{proof}[Proof of Lemma \ref{L:O1p}]
Let $N_1$ be the nilpotent element $y^{-\b_1} \in {}'\fg^{-\b_1}_\bQ$ of \eqref{E:N}.  By \cite[Theorem 6.16]{MR0382272}, it suffices to show that $\texp (\bi\,t\,N_1) \cdot o_1 \in D$ for all $t > 0$.

We may reduce to the case that $\fg = \fsl_2\bC$ and $\check D = \bC\bP^1$ (of Examples \ref{eg:sl2a}, \ref{eg:sl2b}, \ref{eg:sl2c} and \ref{eg:sl2d}) as follows.  Set $\b = \b_1$ and recall the Lie subalgebra $\fsl_2^\b \subset \fg$ of \eqref{E:sl2a} and the Lie subgroup $\tSL_2^\b(\bC) \subset G$ of \eqref{E:SL2a}.  The Cayley transform $\bc_\b$ may be viewed as an element of $\tAd(\tSL_2^\b(\bC))$.  The $\tSL_2^\b(\bC)$--orbit of $o_0$ is a $\bC\bP^1 \subset \check D$ containing $o_1$.  We may identify $o_0$ with the point $(1:0) \in \bC\bP^1$ of Example \ref{eg:sl2b}, and $o_1$ with the point $(1:\bi) \in \bC\bP^1$ of Example \ref{eg:sl2d}.  Moreover, $\fsl^\b_2(\bR)$ is isomorphic with the Lie algebra $\fsu(1,1)$ of Example \ref{eg:sl2a}, and we may identify $D \cap \bC\bP^1$, which is the $\tSL_2^\b(\bR)$--orbit of $o_0$, with the unit disc as in Example \ref{eg:sl2b}.

Under these identifications, 
$$
  \bi\,y^{-\b} \ = \ 
  \half \left( \begin{array}{cc} 1 & 1 \\ -1 & -1 \end{array} \right) \,,
$$
and 
$$
  \texp( \bi\,t\,N_1) \cdot o_1 \ = \
  \texp\left( \begin{array}{cc} 1/2 & 1/2 \\ -1/2 & -1/2 \end{array} \right)
  \left(\begin{array}{c} 1 \\ \bi \end{array}\right) \\
  \ = \ \renewcommand{\arraystretch}{1.3}
  \left(\begin{array}{c} 1 + \half t + \half\bi\, t \\ -\half t + \bi (1-\half t) \end{array}\right) \,.
$$
Whence $\texp(\bi\,t\,N_1)\cdot o_1$ is contained in the unit disc $D \cap \bC\bP^1 \subset D$.
\end{proof}

\begin{proof}[Proof of Theorem \ref{T:matsuki}]

Throughout the proof it will be helpful to keep in mind that strong orthogonality of the roots $\cB$ implies
\begin{equation} \label{E:c=1}
\hbox{the Cayley transform $\bc_{\b_j}$ restricts to the 
identity on $\fsl_2^{\b_i}$ for all $j\not=i$.}
\end{equation}
Whence, it follows from \eqref{E:sl2a} and \eqref{E:sl2Q} that
\begin{equation} \label{E:c(sl2)Q}
\hbox{\emph{the $\bc_j(\fsl_2^{\b_i})$ are defined over $\bQ$ for all $i,j$.}}
\end{equation}

The proof of the theorem is by induction.  

\subsubsection*{The inductive hypotheses}

Assume the theorem holds for $j$, and that 
\begin{equation}\label{E:NC}
  \s_j \ = \ \tspan_{\bQ_{>0}}\{ N_1 , \ldots , N_j \} \quad\hbox{with } \ 
  N_i \,\in\, \fg^{-\b_i}_\bQ \quad\hbox{for all } \ 1 \le i \le j \,.
\end{equation}
Suppose furthermore that a semisimple group $G_j \subset G$, with Lie algebra $\Gamma_j \subset \fg$, and parabolic subgroup $P_j = G_j(\bC)\,\cap\,P$ are given with the following properties: 
\begin{a_list}
\item
The Lie algebra $\Gamma_j$ is defined over $\bQ$.  The roots of $\Gamma_j$ are precisely those that are strongly orthogonal to the $\{\b_1,\ldots,\b_j\}$; equivalently, \eqref{E:Gammaj} holds.  \begin{remark*}Observe that this hypothesis on the roots of $\Gamma_j$ implies that the subalgebra $\bc_{j}(\fsl_2^{\b_i}) = \fsl_2^{\b_i}$ is contained in $\Gamma_j$ for all $i>j$.\end{remark*}
\item 
The parabolic $P_j = \tStab_{G_j(\bC)}(o_j)$, so that $G_j(\bC)/P_j$ is naturally identified with the complex orbit $\check D_j = G_j(\bC) \cdot o_j$.  The real orbit $D_j = G_j(\bR)\cdot o_j$ is open in $\check D_j$.
\item
For all $i \ge j+1$ the Cayley transform $\bc_{\b_i}$ lies in $\tAd(G_j)$.  \begin{remark*}It follows from these last two hypotheses that $o_{j+1} = g_{\b_{j+1}}(o_j) \in \check D_j$.  By \eqref{E:orbit}, the orbit $\cC_{j+1} = G_j(\bR)\cdot o_{j+1}$ is contained in the topological boundary $\del D_j$.\end{remark*}
\end{a_list}

\subsubsection*{The base case}

Set $\s_0 = \{0\}$.  It follows from $\S$\ref{S:HSgZ} that the theorem holds for $j=0$.  The hypothesis \eqref{E:NC} is trivial.  Likewise, setting $G_0 = G$, $\Gamma_0 = \fg$, and $P_0 = P$, the hypotheses (a-c) are immediate.  (Here, $\check D_0 = \check D$, $D_0 = D$ and $\cC_1 = \cO_1$.)

\subsubsection*{The induction: part 1}

First note that \eqref{E:c=1} implies that the nilpotent cone
\begin{equation} \nonumber 
  \s_j \ = \ \bc_{\b_{j+1}}(\s_j) 
\end{equation}
is unchanged under the next Cayley transform.
  
By \eqref{E:c(sl2)Q} the root space ${}_{j+1}\fg^{-\b_{j+1}}$ is defined over $\bQ$.  By Lemma \ref{L:O1p}, we may choose $N_{j+1} \in {}_{j+1}\fg^{-\b_{j+1}}_\bQ$ so that $\{ \texp(\bi \,y N_{j+1}) \cdot o_{j+1} \} \in B_j(N_{j+1})$ is a nilpotent orbit in $\check D_j$.  Equivalently,
\begin{equation} \label{E:norb_j}
  \texp(\bi \,y N_{j+1}) \cdot o_{j+1} \ \in \ D_j
  \quad \hbox{for all } \ y > 0 \,.
\end{equation}

By the hypothesis (a), $N_{j+1} \subset \Gamma_j$ so that $\texp(\bC\,N_{j+1}) \subset G_j(\bC)$.  Moreover, the hypotheses \eqref{E:NC} and (a) imply that $\texp(\s_j)$ commutes with $G_j$.  Whence, given any $N \in \s_j$ and $y > 0$, 
\begin{eqnarray*}
  \texp( \bi\,(N + y N_{j+1}) ) \cdot o_{j+1}
  & = & 
  \texp(\bi\,N) \,\texp(\bi\,y \,N_{j+1}) \cdot o_{j+1} \\
  & \subset &
  \texp(\bi\,N) \, D_j \\
  & = &
  \texp(\bi\,N) \, G_j(\bR) \cdot o_j \\
  & = &
  G_j(\bR) \, \texp(\bi\,N) \, \cdot o_j \\
  & \subset & 
  G_j(\bR) \, D \ = \ D \,.
\end{eqnarray*}
Setting 
$$
  \s_{j+1} \ = \ \tspan_{\bQ_{>0}}\{\s_j,N_{j+1}\}\,,
$$ 
we have shown that $\{\texp(\bC\,\s_{j+1}) \cdot o_{j+1} \} \in B(\s_{j+1})$ is a nilpotent orbit; this establishes the first two (of three) claims of Theorem \ref{T:matsuki} for $j+1$.  Also, we see that \eqref{E:NC} holds for $j+1$.

Given $i \le j+1$, complete the pair $\ttH^{\b_i} , N_i \in \fg_\bQ$ to a standard triple ($\S$\ref{S:sl2trp}) $\{ N_i^+ , \ttH^{\b_i} , N_i\}$ with $N_i^+ \in \fg^{\b_i}_\bQ$.  To see that the third claim $W_\sb(\s_{j+1}) = {}_{j+1}W_\sb$ holds it suffices to observe that the strong orthogonality of the roots implies 
\begin{eqnarray*}
  N^+ & = & N_1^+ \,+\cdots+\, N_{j+1}^+ \,,\\
  \ttY_{j+1} & = & \ttH^{\b_1}\,+\cdots+\,\ttH^{\b_{j+1}} \,,\\
  N & = & N_1 \,+\cdots+\, N_{j+1}
\end{eqnarray*}
is a standard triple in $\fg_\bQ$.

\subsubsection*{The induction: part 2}

It remains to show that the additional inductive hypotheses (a-c) hold for $j+1$.  Let $\Gamma_{j+1} \subset \Gamma_j$ be the Lie algebra $\Gamma_{\b_{j+1}}$ of $\S$\ref{S:a-prim}.  Then hypothesis (a) is immediate.  Likewise, let $G_{j+1} \subset G_j$ and $P_{j+1} \subset P_j$ be the corresponding $\sG$ and $\sP$, respectively; and let $D_{j+1} \subset \check D_{j+1}$ be the corresponding $\sD \subset \check \sD$, respectively.  Recall that $D_{j+1}$ is open in $\check D_{j+1}$; see $\S$\ref{S:a-prim}.  Thus the hypothesis (b) holds.  

Let $i \ge j+2$.  A priori the Cayley transformation $\bc_{\b_i}$ lies in $\tAd(G)$.  However, strong orthogonality implies that $\cB\backslash\{\b_1,\ldots,\b_{j+1}\} \subset \Delta(\Gamma_{j+1})$ and 
$$
  \fsl_2^{\b_i} \ = \ \bc_{j+1}(\fsl_2^{\b_i}) \ \subset \ \Gamma_{j+1}\,.  
$$
Therefore, the Cayley transformation $\bc_{\b_i}$ lies in $\tAd(G_{j+1})$.  Whence hypothesis (c) holds.  
\end{proof}

\begin{proof}[Proof of Lemma \ref{L:conj+Ci}]
By construction the Cartan subalgebra $\fh = \fh_0$ is preserved by conjugation and Cartan involution ($\S$\ref{S:Qform}).  Since the Cayley transform preserves these properties, the Cartan subalgebra $\fh_j$ is also preserved by conjugation and the Cartan involution.  For any such Cartan subalgebra the roots satisfy $\theta\a = -\bar\a$.
\end{proof}

\begin{proof}[Proof of Lemma \ref{L:EbarE}]
Define
$$
  \ft_j \ \dfn \ \fh_j \,\cap\, \fk_\bR \tand
  \fa_j \ \dfn \ \fh_j \,\cap\, \fk^\perp_\bR \,.
$$
Then the real form $\fh_j(\bR) \dfn \fh_j \cap \fg_\bR$ is $\ft_j \op \fa_j$.  In analogy with \eqref{E:t},
\begin{equation} \nonumber 
  \ft_j \ = \ \bigcap_{i=1}^j \tker\{ \b_i : \fh_\bR \to \bC \}
  \tand
  \fa_j \ = \ \tspan_\bR\{ \ttH^{\b_1} , \ldots , \ttH^{\b_j} \}\,.
\end{equation}
Moreover, 
\begin{equation} \label{E:conj}
\hbox{the roots of $\fh_j$ are $\bR$--valued on $\fa_j$, and $\bi\,\bR$--valued on $\ft_j$.}
\end{equation}
Since $\a(\ttE) \in \bR$, we see from \eqref{E:conj} that $\ttE \in \bi \ft_j \,\op\,\fa_j$.  Therefore, $\ttE + \overline{\ttE} \in \fa_j$.  Thus
$$
  \ttE \,+\, \overline{\ttE} \ \in \ 
  \tspan_\bR\{ \ttH^{\b_1} , \ldots , \ttH^{\b_j} \} \,.
$$
Since the roots $\b_i$ are real, \eqref{E:a=1} and \eqref{E:sorth} yield $\b_i(\ttE + \overline{\ttE}) = \b_i(\ttE) + \b_i(\ttE)= 2$.
\end{proof}

\begin{proof}[Proof of Lemma \ref{L:bar-b}]
By definition, the conjugate root $\bar\b$ is given by $\bar\b(\xi) = \b(\bar\xi)$ for any $\xi \in \fh_j$.  Whence \eqref{E:bar-b} holds when restricted to $\ft_j$.  It remains to show that the equation holds when restricted to $\fa_j$.  Since the latter is spanned by the $\{\ttH^{\b_i}\}_{i=1}^j$, it suffices to show that \eqref{E:bar-b} holds when evaluated on any of these $\ttH^{\b_i}$.  Strong orthogonality of the roots $\b_i \in \cB$ and \eqref{E:int} imply that 
\begin{equation} \label{E:sorth}
  2\,\frac{(\b_i , \b_k )}{(\b_k,\b_k)} \ = \ 
  \b_i(\ttH^{\b_k}) \ = \ 2 \d_{ik} \,,\quad \hbox{for any } \ 
  1 \le i,k \le s \,.
\end{equation}
It follows that \eqref{E:bar-b} holds when evaluated on the $\{\ttH^{\b_k}\}_{k=1}^j$.
\end{proof}

\begin{proof}[Proof of Lemma \ref{L:O1c}]
From the definition \eqref{E:gl} we see that $\fh_1$ is a Cartan subalgebra of $\fg_0$, and ${}_1\ttH^{\b_1}$ is contained in the center $\fz$ of ${}_1\fg_0$.  It then follows from \eqref{E:t} that $\fh_1 \,\cap\, {}_1\fg_0^\tss$ is a compact Cartan subalgebra of ${}_1\fg_0^\tss$.
\end{proof}

\begin{proof}[Proof of Lemma \ref{L:codimO1}]
For the duration of the proof, let $\a$ denote the simple root $\b_1 \in \cB$.  Set $\fg^{p,q} = {}_1\fg^{p,q}$, and $o = o_1$ and $\cO = \cO_1$.  By \cite[Proposition 4.1]{MR3331177}, the (real) codimension of $\cO$ is the number of roots in 
$$
  \Delta(+,+) \ = \ \{ \gamma \in \Delta \ | \ 
  \gamma(\ttE) \,,\ \bar\gamma(\ttE) > 0 \} \,.
$$
We will show that 
\begin{equation} \label{E:codim=1}
  \Delta(+,+) \ = \ \{ \a\}  \,;
\end{equation}
whence $\tcodim_\bR\cO_1 = 1$, establishing the lemma.  By Remark \ref{R:codimO1}, $\a\in\Delta(+,+)$.

Now suppose that $\gamma \in \Delta(+,+)$. 
The two inequalities $\gamma(\ttE),\bar\gamma(\ttE)>0$ imply that both $\gamma$ and $\bar\gamma$ are positive roots.  But from Lemma \ref{L:bar-b}, we see that $\bar\gamma = -\gamma + \ell \a$, where $\ell = \gamma(\ttH^\a)$.  In particular, if $\gamma$ is not a multiple of the simple root $\a$, then $\bar\gamma$ is a negative root.  Whence $\gamma=\a$.
\end{proof}

\begin{proof}[Proof of Proposition \ref{P:maxP}]

By \cite[Proposition 4.14]{MR3331177}, the codimension one $G_\bR$--orbits in $\tbd(G_\bR\cdot o_0)$ are of the form $G_\bR\cdot(g_\b\cdot o_0)$, with $\b \in\Delta(\fg^+)$.  In fact, we must have $\b\in\Delta(\fg^1)$ in order not to violate \cite[Corollary 4.4]{MR3331177}.

So let $\b$ be a (noncompact, positive) root with $\b(\ttE)=1$. Recall $\ttE=\ttS^{\tti}$.  We will show that $\b$ belongs to the orbit of $\a_{\tti}$ under the subgroup $W(\fg^0,\fh)\subset W(\fg_{\bR},\fh_{\bR})$ of the real Weyl group.  It will follow that $g_\b\cdot o_0$ and $g_{\a_{\tti}}\cdot o_0$ belong to the same $G_\bR$--orbit, giving the desired uniqueness.

Acting by $W(\fg^0,\fh)$ if necessary, we may assume $(\b,\a_j)\leq 0$ for all $j\neq \tti$; equivalently,
\begin{equation} \label{E:1}
  \a_j(\ttH_\b)\ \leq \ 0 \,.  
\end{equation}
(That is, use the Weyl group of $G^0$ to reflect the orthogonal projection of $\b$ to $\ker(\ttE)$ into the negative Weyl chamber.)  Set $\cB=\{\b\}$ and $\fg^{p,q}= {}_1 \fg^{p,q}$ for the bigrading associated to $o_1=g_\b\cdot o_0$.  For any root $\a \in\Delta(\fg,\fh_1)$, write $\fg^{p(\a),q(\a)}\supset \fg^{\a}$.  By Lemma \ref{L:EbarE}, $p(\a)+q(\a)=\a(\ttH^{\b})$, and of course $p(\b)=1=q(\b)$.

Now for any positive root $\a$, we have $$ \a=p(\a)\a_{\tti}+\sum_{j\neq \tti} m_j(\a)\a_j$$ where $m_j(\a)\geq 0$.  Since $p(\b)=1$ and $\b$ is positive, 
\[
  p(\a_{\tti})+q(\a_{\tti}) \ = \ \a_{\tti}(\ttH^{\b}) \ = \ 
  \b(\ttH^\b) \ - \ \sum_{j\neq \tti} m_j(\b)\a_j(\ttH^{\b})
  \ \stackrel{\eqref{E:1}}{\geq} \  \b(\ttH^{\b})=2\,.
\]
But $p(\a_{\tti})=\a_{\tti}(\ttE)=\a_{\tti}(\ttS^{\tti})=1$, and so $q(\a_{\tti})\geq 1$.  By Remark \ref{R:codimO1}, $\text{codim}(G_\bR\cdot o_1)$ is now $\geq 2$ unless $\b=\a_{\tti}$.
\end{proof}

\section{Enhanced $\tSL_2$-orbits and adjoint varieties}\label{S:sl2+adj}

Following some motivation ($\S$\ref{S:SvHs}) and a brief review of boundary components of  MT domains ($\S$\ref{S:B+D}), we use the construction of \S\ref{S:cayleyconstr} to produce a class of MT subdomains ($\S$\ref{S:sl2con}) which are usually not Schubert varieties, hence will be interesting to compare with them in cohomology.  (We only take this up in a limited way here, in $\S$\ref{S:[X(N)]eg} for the $G_2$ adjoint variety.)  The remainder of this section ($\S\S$\ref{S:adI}-\ref{S:adIII}) is devoted to working out these subdomains and boundary components for the fundamental adjoint varieties, where they have striking relationships to both the varieties of lines studied in $\S$\ref{S:sec5} and the Calabi-Yau VHS studied in \cite{FL}.

\subsection{Schubert varieties and horizontal subdomains}\label{S:SvHs}

Let $\check{D}=G/P$ be the compact dual of a MT domain $D$, $X\subset \check{D}$ a horizontal subvariety, and $Y$ a (nonempty) connected component of $X\cap D$. Recall the result of Friedman and Laza:

\begin{proposition} \label{Prop A}
If $X$ is smooth and $Y$ has strongly quasi-projective image in an arithmetic group quotient $\Gamma \setminus D$, then $Y$ is a (necessarily Hermitian) Mumford--Tate subdomain of $D$.
\end{proposition}

\noindent Motivated by the second author's study of ``Schubert VHS'' \cite{MR3217458}, the authors showed the apparently unrelated result \cite[Theorem 1.3]{KR2} that for $X$ a (horizontal) \emph{Schubert} variety, $X$ \emph{is smooth if and only if it is a} (homogeneously embedded, possibly reducible) \emph{Hermitian symmetric space}. However, if we assume that $D$ is a MTC domain, this result has the following corollary --- clearly analogous to Proposition \ref{Prop A}, but with a stronger assumption on $X$ and none on $Y$:

\begin{proposition} \label{Prop B}
If $X$ is a smooth Schubert variety, then up to translation by $G_{\bR}$, $Y$ is a (necessarily Hermitian) Mumford--Tate subdomain of $D$.
\end{proposition}

\begin{proof}
Translate $Y$ to pass through the distinguished base point $o_0$ of $D$ ($\S$\ref{S:HSgZ}), with corresponding grading element $\ttT$ and Hodge structure $\varphi_0: S^1\to G_\bR$.  By \cite[Theorem 1.3]{KR2}, the Schubert variety $X$ is a compact Hermitian symmetric space of the form $X = G'\cdot o$, where $G' \subset G$ is a semisimple subgroup corresponding to a subdiagram of the Dynkin diagram of $\fg$.  We have $Y=(G'_{\bR})^{\circ}\cdot o_0$.\footnote{Here the superscript $\circ$ denotes the connected identity component.} Since $\fg'=\fg'_1\oplus \fg'_0\oplus\fg'_{-1}$ decomposes into a direct sum of $\fh \cap \fg'$ and a subset of $\fh$--root spaces, it follows that $\fg'$ is stable under $\ttT \in \fh$.  Therefore, $\varphi_0(S^1)$ normalizes $G'_{\bR}$, and we set $\tilde{G}'_{\bR}\dfn \varphi_0(S^1)\cdot G'_{\bR}$.  By Remarks \ref{rem:MTC} and \ref{rem:MTCD}, this is underlain by a $\bQ$-algebraic group $\tilde{G}'_{\bQ}$.

Now since $\tad(\ttT)|_{\fg'_{-1}} = -\one_{\fg'_{-1}}$, the (complexified) $\bQ$--Lie-algebra-closure of a very general $\tad(\fg'_{\bR})$--translate of $\ttT$ contains $\fg'_{-1}=\tilde{\fg}'_{-1}$.  Therefore, so does the $\bC$--Lie algebra of the MT group $\tilde{G}''_{\bQ}$ of a very general $\tilde{G}'_{\bR}$-conjugate of $\varphi_0$.  It follows that $Y=(\tilde{G}''_{\bR})^{\circ}\cdot o_0$, where $\tilde{G}''_{\bQ}$ is the MT group of a very general point (i.e., Hodge structure) in $Y$; that is, $Y$ is a connected MT subdomain.
\end{proof}

\begin{remark}\label{rem:Corcon}
(i) The converse of Proposition \ref{Prop B} fails dramatically:  the compact dual $\check D'\subset \check D$ of a horizontal (Hermitian) MT subdomain $D'\subset D$ need not be a Schubert variety, even when $\check D'$ is a maximal integral of the IPR.

(ii) We expect that Proposition \ref{Prop B} itself fails for more general (non-MTC) MT domains, to an extent which should be describable in terms of the action of $\mathit{Gal}(\bar{\bQ}/\bQ)$ in the Dynkin diagram $\sD$.
\end{remark}

\noindent A large class of examples illustrating (i), called \emph{enhanced $\tSL_2$-orbits}, is constructed below in $\S$\ref{S:sl2con}. The simplest one is probably the $\tSL_2\times \tSL_2$ subdomain in the $G_2$-adjoint variety (see $\S\S$\ref{S:adII}-\ref{S:[X(N)]eg}). The value of having such explicit counterexamples is that, being horizontal, they can be expanded in nonnegative integer combinations of horizontal Schubert classes in $H^*(\check{D},\mathbb{Z})$, yielding a submonoid of smoothable classes in each degree.

\subsection{Boundary components and the naive limit map}\label{S:B+D}

First we review the definitions and recall some results from \cite{KP2012, MR3331177}.  Fix a MT domain $D \subset \check D$ and pairwise commuting nilpotents $N_1,\ldots,N_m \in \fg_\bQ$.  Let $\sigma\dfn \bQ_{\geq 0}\langle N_1,\ldots,N_m\rangle\subset \fg_{\bQ}$ be the rational nilpotent cone, with interior $\sigma^{\circ}\dfn \bQ_{>0}\langle N_1,\ldots ,N_m\rangle$.  By Theorem \ref{T:cks}(b), $W_{\sb}(\sigma)\dfn W_{\sb}(N)$ is independent of $N\in \sigma^{\circ}$.  Regarding\footnote{See $\S$\ref{S:no} for relevant definitions}
$$
  \tilde B(\sigma) \ \dfn \ 
  \{F^\sb \in \check D \ | \ (F^\sb ;N_1 ,\ldots,N_m ) \hbox{ is a nilpotent orbit} \}
$$
as a set of (limiting) $\bQ$-mixed Hodge structures $(V_{\bQ},F^\sb,W_{\sb}(\sigma))$, it makes sense to consider its generic  MT group (or that of some subset).  The \emph{boundary component} associated to  $\sigma$ is defined by
$$
  B(\sigma) \ \dfn \ \tilde B(\sigma) / \exp(\bC\,\sigma) \, ;
$$
the points of $B(\sigma)$ are $\sigma$-nilpotent orbits, or equivalently LMHS up to ``reparametrization'' (i.e. modulo $e^{\bC\,\sigma}$).  When $m=1$, we shall replace $\sigma$ everywhere in our notation by $N_1$; e.g. $\tilde{B}(N_1)$.  By \cite[Remark 5.6]{MR3331177}, we have $\tilde{B}(\sigma)=\cap_{N\in \sigma^{\circ}}\tilde{B}(N)$.

Let $D(\sigma)$ be the MT domain parameterizing the split polarized HS $\op_\ell (\tGr_\ell W_\sb(\sigma) , F^\sb)$. We have a natural map $\tilde{\rho}:\,\tilde{B}(\sigma)\twoheadrightarrow D(\sigma)$ given by taking the associated graded Hodge structure.  A choice of $\bQ$-split base point in $\tilde{B}(\sigma)$ gives rise to a section $D(\sigma)\hookrightarrow\tilde{B}(\sigma)$ of $\rho$, and we shall sometimes blur the distinction between $D(\sigma)$ and its image.

 Regarding $N_1,\ldots,N_m$ as elements of $\tEnd(\fg)$, let $W_\sb(\sigma)_\fg$ be the corresponding weight filtration of $\fg$.  Let $Z(\sigma) \subset G$ be the centralizer of $\sigma$, and let $\fz(\sigma)$ denote the Lie algebra of $Z(\sigma)$.  Note that 
\begin{equation} \label{E:zNinW0}
  \fz(\sigma) \ \subset \ W_0(\sigma)_\fg \,.
\end{equation}
Let $W_\sb(\sigma)_\fz = \fz(\sigma) \cap W_\sb(\sigma)_\fg$ be the induced filtration of $\fz(\sigma)$.  Then
$$
  \fz(\sigma) \ \simeq \ \bigoplus_{\ell\ge0} \fz_{-\ell}(\sigma)\,,\quad\hbox{where}\quad
  \fz_{-\ell}(\sigma) \ \dfn \ W_{-\ell}(\sigma)_\fz / W_{-\ell-1}(\sigma)_\fz \,.
$$
There is a natural tower of fibrations (factoring $\tilde{\rho}$)
\begin{equation} \label{E:tower}
  B(\sigma) \ \sur \ \cdots \ \sur \ B(\sigma)_{(k)} \ \stackrel{\rho_{\sigma}^{(k)}}{\sur} \ B(\sigma)_{(k-1)} \ \sur \ \cdots
  \ \sur \ B(\sigma)_{(1)} \ \stackrel{\rho_{\sigma}^{(1)}}{\sur} \ D(\sigma)
\end{equation}
with $\rho_{\sigma}^{(k)}$--fibre $\mathfrak{F}^{(k)}$ through $F^\bullet$ (mod $e^{\bC\,\sigma}$) equal to 
\begin{equation} \label{E:fibre}
  \mathfrak{F}^{(k)} \ = \ 
  \frac{\fz_{-k}(\sigma)}{F^0\,\fz_{-k}(\sigma)}\,,
  \quad \hbox{for } k \not=2 \,,\tand
  \mathfrak{F}^{(2)} \ = \ 
  \frac{\fz_{-2}(\sigma)}{\bC\,\sigma \, \op \, F^0\,\fz_{-2}(\sigma)} \,.
\end{equation}
(cf. \cite[$\S$7]{KP2012}).

Now let $\sigma\dfn \sigma_s = \tspan_{\bQ_{\ge0}}\{ N_1 , \ldots , N_s\}$ be the cone of Theorem \ref{T:matsuki}, with $N_j \in {}_s\fg^{-\b_j}$, and regard ${}_s F^\sb$ (\cf\eqref{E:FW}) as a $\bQ$-split base point in $\tilde{B}(\sigma)$.\footnote{Though $o_s\in\check D$ corresponds to this filtration, here we shall write ${}_s F^\sb$ for the point of $(D(\sigma)\subset )\,\tilde{B}(\sigma)$ and $o_s$ for the point of $\check D$.}  The Lie algebra $\fg_{B(\sigma)}$ of the generic MT group $G_{B(\sigma)}$ of $D(\sigma)$ may be described as follows (cf. \cite[$\S\S$4-5]{KP2012}).  Let $\ttE$ and $\ttY = \sum_{j=1}^s \ttH^{\b_j}$ be the grading elements of \eqref{E:gp}${}_{j=s}$ and \eqref{E:gl}${}_{j=s}$.  Set $\phi = \bi (\ttE - \bar \ttE)$ (which is equal to $\bi(2\ttE - \ttY)$ by Lemma \ref{L:EbarE}).  Then $\fg_{B(\sigma)}$ is the $\bQ$--Lie algebraic closure in $\fz_0(\sigma)=\cap_{j=1}^s \Gamma^0_{\b_j}$ of 
\[
  Z_0(\s)_\bR \cdot \bC\phi \,.
\]
A consequence of our construction (and Remark \ref{rem:MTCD}) is that $\phi\in \bi\fg_{\bR}$, and the roots of $\fh_s$ in $\fz_0 (\sigma)$ come in conjugate pairs defined over $\bQ (\mathbf{i})$.  Therefore, $\fg_{B(\sigma)}$ is the ($\bC$--)Lie algebraic closure of
\[
  \bC\phi \ + \
  \sum_{\tiny \begin{array}{c} \beta\in\Delta(\Gamma_s) \\ \beta(\phi)\neq 0 \end{array} } \fg^{\beta}\, ,
\]
and so $\fg_{B(\sigma)}\subseteq \Gamma_s + \bC \phi $ while $\fg^{ss}_{B(\sigma)}\subseteq \Gamma_s$.  In most cases of interest we will have equality.  In any case\footnote{The discrepancy between $\fg_{B(\sigma)}$ and $\Gamma_s + \bC\phi$ lies in $\ker(\tad\,\phi)$, and therefore stabilizes ${}_s F^\sb$.} we have
\[
D(\sigma) = G_{s,\bR} \cdot{}_s F^\sb \cong G_{s,\bR} / G_{s,\bR}\cap P
\]
(where $\mathit{Lie}(G_s)=\Gamma_s$), and this gives an open subset in the homogeneous variety
\[
\check D(\sigma)\dfn  G_s \cdot{}_s F^\sb \cong G_s / G_s \cap P .
\]
Now we could regard $\check D(\sigma)$ as a subvariety of $\check D$ by identifying it with $\check D_s \dfn  G_s \cdot o_s \subset \check D$ (cf. Remark \ref{R:cusp}).  However, for some purposes it is better to regard it as a separate variety and map it into $\check D$ in a different way:

\begin{definition} \label{D:naive}
Following \cite{MR3331177, MR3115136},\footnote{The term \emph{reduced limiting period mapping} is used in \cite[App. to $\S$10]{MR3115136}, but we prefer to reserve the term ``period mapping'' for Hodge--theoretic classifying maps associated to families of motives.  They are ``naive'' because they replace the limit MHS of a nilpotent orbit by its (much coarser) actual limit, which students in Hodge theory are trained to \emph{not} take.} the \emph{naive limit map} \[ \Phi^{\sigma}_{\infty}:\, B(\sigma)\to \tbd (D) \] sends $F^\sb\mapsto \lim_{\text{Im}(z)\to \infty} e^{zN} F^\sb$, which is independent of the choice of $N\in \sigma^{\circ}$, \cf\cite[Remark 5.6]{MR3331177}).
\end{definition}

Though $\Phi^{\sigma}_{\infty}$ is not isomorphic onto its image (denoted $\hat{B}(\sigma)$ in \cite{MR3331177}), its restriction to $D(\sigma)$ is a $G_{s,\bR}$--equivariant isomorphism, which extends to a $G_s$--equivariant embedding of $\check D(\sigma)$ into $\check D$.  However, the image of this embedding is not $\check D_s$.

To present  $\check D_s$ as the image of a naive limit map, recall the commuting standard triples $\{ N_j^+,\ttH^{\b_j},N_j \}$ from the proof of Theorem \ref{T:matsuki}, and put $\tilde{\sigma}\dfn \bQ_{\geq 0}\langle -N_1^+,\ldots ,-N_s^+\rangle $.  Setting ${}_s \tilde{\fg}^{p,q} \dfn  {}_s \fg^{-q,-p}$, ${}_s \tilde{F}^a \dfn  \oplus_{p\geq a;\, q\in \mathbb{Z}} \,{}_s \tilde{\fg}^{p,q}$, ${}_s \tilde{W}_b \dfn \oplus_{p+q\leq b}\, {}_s \tilde{\fg}^{p,q}$, and reasoning as in the proof of \cite[Theorem  5.15]{MR3331177}, one shows that:
\begin{bcirclist}
\item ${}_s \tilde{W}_{\sb} = W_{\sb}(\tilde{\sigma})_{\fg}$ is graded by $\tilde{\ttY}\dfn -\ttH^{\b_1}-\cdots -\ttH^{\b_s}$;

\item $( {}_s \tilde{F}^\sb ; -N_1^+,\ldots ,-N_s^+ )$ is a nilpotent orbit, and $( {}_s \tilde{F}^\sb, W_{\sb}(\tilde{\sigma})_{\fg} )$ is $\bQ$-split; and

\item $\Phi^{\tilde{\sigma}}_{\infty}$ sends ${}_s \tilde{F}^\sb$ to ${}_s F^\sb$.
\end{bcirclist}
We conclude that $D(\tilde{\sigma})$ and $\check D(\tilde{\sigma})$ are the $G_{s,\bR}$ and $G_s$--orbits of ${}_s \tilde{F}^\sb$, respectively, and $\Phi_{\infty}^{\tilde{\sigma}}$ restricts to a $G_s$-equivariant isomorphism from $\check D(\tilde{\sigma})$ to $\check D_s$.

\subsection{Construction of the enhanced $\tSL_2$-orbits} \label{S:sl2con}

The above discussion gave rise to identifications of $\check D(\sigma)$ with the two subsets $\check D_s$ and $\overline{\Phi^{\sigma}_{\infty}(D(\sigma))}^{\text{Zar}}$ of $\check D$; we shall now ``interpolate'' them.  Via the section $D(\sigma)\hookrightarrow \tilde{B}(\sigma)\twoheadrightarrow B(\sigma)$; each point $g\cdot{}_s F^\sb \in D(\sigma)$, with $g\in G_{s,\bR}$, produces a $\sigma$--nilpotent orbit $e^{\bC\, \sigma}g\cdot o_s\in\check D$ with $\bQ$-split LMHS.  Taking the Zariski closure of the union of these, we have the

\begin{definition}\label{D:ESL2}
The \emph{enhanced (multivariable) $\tSL_2$-orbit} associated to $(o_s,\sigma)$ is 
\[
  X(\sigma) \ \dfn \ \overline{e^{\bC \,\sigma}G_s\cdot o_s}^{\text{Zar}} 
  \ = \ \tilde{G}_s \cdot o_s \,,
\]
where 
\[
  \tilde{G}_s \ \dfn \  G_s \times (\prod_{j=1}^s \tSL_2^{\b_j} ) \,.
\]
\end{definition}

Note that $X(\sigma)$ does not identify with a subset of $\tilde{B}(\sigma)$, though $e^{\bC \,\sigma}G_s\cdot o_s$ does.

\begin{proposition}\label{P:X(sigma)}
The enhanced $\tSL_2$--orbit $X(\sigma)$ is the compact dual of a  Mumford--Tate subdomain $Y(\sigma)\subset D$; in particular, it is smooth.
\end{proposition}

\begin{proof}
The homogeneous description in Definition \ref{D:ESL2} already implies smoothness.  Since $o_s$ is $\bQ$-split (\emph{a fortiori} $\bR$-split) \cite[Lemma 5.14]{MR3331177} yields a diagram
\[
\begin{array}{ccccc} 
  \check D(\sigma)\times (\mathbb{P}^1)^{\times s} 
  & \buildrel{\simeq}\over{\to} & X(\sigma) & \subset & \check D \\ 
  \cup & & \cup & & \cup \\ 
  D(\sigma)\times \sH^{\times s} & \buildrel{\simeq}\over{\to} 
  & Y(\sigma) & \subset & D 
\end{array}
\]
in which the bottom row is given by 
\[
  (g\cdot{}_s F^\sb\,;\n z_1,\ldots ,z_s) \ \mapsto \ 
  \exp(\textstyle\sum s_j N_j) g\cdot{}_s F^\sb\,.
\]
(Note that the $\mathbb{P}^1$ factors need not be minimally embedded in $\check D$.)  Moreover, it is clear that the generic MT group of $Y(\sigma)$ is $G_{B(\sigma)}\times (\times_{j=1}^s SL_2^{\b_j})$, which acts transitively on $Y(\sigma)$.
\end{proof}

\begin{lemma} \label{L:X(sigma)}
The enhanced $\tSL_2$--orbit $X(\sigma)$ is horizontal (and Hermitian) if and only if $-1\leq \ttE (\a)\leq 1$ for every $\a\in\Delta$ strongly orthogonal to $\{\b_1,\ldots,\b_s\}$.
\end{lemma}

\begin{proof}
The following are equivalent:
\begin{bcirclist}
\item 
$X(\sigma)$ is horizontal.
\item 
$Y(\sigma)$ is horizontal.
\item 
$-1 \le \ttE(\a)\le 1$ for all $\a\in\Delta(\mathit{Lie}(\tilde{G}_s))$.
\item  
$-1 \le \ttE(\a)\le 1$ for all $\a\in\Delta(\mathit{Lie}(G_s))=\Delta(\Gamma_s)=\cap_j\,\Delta(\Gamma_{\b_j})$, \cf $\S$\ref{S:a-prim}.
\end{bcirclist}
\end{proof}

In \cite[Appendix A]{KR1long} we determine the bigradings $\{ {}_s \fg^{p,q}\}$ for $s=1$ and $P$ a maximal parabolic.  The condition of the lemma fails in many cases, but holds for a sizable subset which includes all the adjoint varieties (including type $A$, where the parabolic is not maximal).  We will elaborate on these cases in $\S\S$\ref{S:adI}-\ref{S:adII}.

One reason for studying the $X(\sigma)$ is to express classes of (effective linear combinations of) Schubert varieties in terms of ``simpler'' objects.  Of particular interest is the horizontal case, where $[X(\sigma)]$ can be written in terms of classes of horizontal Schubert varieties \cite{MR3217458}.  Here ``simpler'' might mean that $X(\sigma)$ is smooth while the Schubert varieties of a given dimension are singular; or it could simply be the case that $X(\sigma)$ (note its $\mathbb{P}^1$ factors) is decomposable while the $\{ X_w\}$ are not.  Along these lines we have the easy

\begin{proposition}
If $\check D=G/P$ with $P$ a maximal parabolic corresponding to a non-short simple root, then the $\{ X(\sigma)\}$ of dimension at least $2$ are never Schubert varieties.
\end{proposition}

\begin{proof}
By \cite[Proposition 3.7]{HongMok2013}, the smooth Schubert varieties are of the form $X(\mathcal{D}')$ for $\mathcal{D}'\subset \mathcal{D}$ a \emph{connected} subdiagram of the Dynkin diagram (\cf\cite[Remark 3.5]{KR2}).  So they do not have $\mathbb{P}^1$ factors.
\end{proof}

\subsection{Fundamental adjoint varieties I: codimension one orbits}\label{S:adI}

In the case that $\check D = G/P$ is a fundamental adjoint variety ($\S$\ref{S:AV}), a unique open $G_{\bR}$-orbit $D$ admits the structure of a MT domain (Remark \ref{rem:openGR}(ii)). Further, by Proposition \ref{P:maxP}, there is a unique codimension-one $G_\bR$--orbit $\cO_1\in \tbd(D)$. If $\a_{\tti}$ is the simple root associated to the maximal parabolic $P$, then taking $\cB = \{ \b_1 = \a_\tti\}$ yields a distinguished base point $o_1\in\cO_1$ via the construction of $\S$\ref{S:details}.

We claim that the bigradings $\fg = \op{}_1\fg^{p,q}$ associated to $o_1$ exhibit a uniform appearance for all the adjoint varieties, as indicated in Figure \ref{f:adj_bigr}.
\begin{figure}[!ht]
\caption{The bigrading $\fg = \op\,{}_1\fg^{p,q}$ for fundamental adjoint varieties.}
\setlength{\unitlength}{2.0pt}
\begin{picture}(50,50)(-22,-22)
\multiput(0,-10)(0,10){3}{\circle*{1.3}}
\multiput(10,-20)(0,10){4}{\circle*{1.3}}
\put(20,-10){\circle*{1.3}}
\multiput(-10,20)(0,-10){4}{\circle*{1.3}}
\put(-20,10){\circle*{1.3}}
\thicklines
\put(0,0){\circle{2.8}} \put(0,0){\circle{5.0}}
\multiput(-10,10)(10,0){2}{\circle{2.8}}
\multiput(0,-10)(10,0){2}{\circle{2.8}}
\multiput(-10,0)(20,0){2}{\circle{2.8}}
\thinlines
\put(0,0){\vector(1,0){23}} \put(24,0){\small{$p$}}
\put(0,0){\vector(-1,0){23}} \put(1,23){\small{$q$}}
\put(0,0){\vector(0,1){23}}
\put(0,0){\vector(0,-1){23}}
\end{picture}
\label{f:adj_bigr}
\end{figure}
There, the nodes $\bullet$ indicate a $\fg^{p,q}$ of dimension one, the once-circled nodes indicate dimension $a$, and the twice-circled node indicates dimension $b$, where the values of $a,b$ are as given in Table \ref{t:dims}.
\begin{footnotesize}
\begin{table}[!ht]
\caption{Dimensions of the $\fg^{p,q}$ in Figure \ref{f:adj_bigr}.}
\renewcommand{\arraystretch}{1.2}
\begin{tabular}{|c||c|c|c|c|c|c|c|}
  \hline
  $G$ & $B_r$ ($r\ge3$) & $D_r$ ($r\ge4$) & $E_6$ & $E_7$ & $E_8$ & $F_4$ & $G_2$ 
  \\ \hline
  {\setlength{\unitlength}{2.5pt}
   \begin{picture}(1,1)(-0.5,-0.7)
     \put(0,0){\circle*{1.3}} 
   \end{picture}
  } 
  & 1 & 1 & 1 & 1 & 1 & 1 & 1 \\ \hline
  {\setlength{\unitlength}{2.5pt}
   \begin{picture}(1,1)(-0.5,-0.7)
     \put(0,0){\circle*{1.3}} 
     \thicklines\put(0,0){\circle{2.5}} \thinlines
   \end{picture}
  } 
  & $2r-4$ & $2r-5$
  & 9 & 15 & 27 & 6 & 1 \\
  \hline
  {\setlength{\unitlength}{2.5pt}
   \begin{picture}(1,1)(-0.5,-0.8)
     \put(0,0){\circle*{1.3}} 
     \thicklines\put(0,0){\circle{2.5}} \put(0,0){\circle{3.8}} \thinlines
   \end{picture}
  } 
  & $2r^2 - 11 r + 18$ & $2r^2 - 13 r + 24$
  & 18 & 37 & 80 & 10 & 2 \\
  \hline
\end{tabular}
\label{t:dims}
\end{table}
\end{footnotesize}

Writing $\fg^{p,q}$ for ${}_1 \fg^{p,q}$ and $h^{p,q}\dfn \dim_\bC\fg^{p,q}$, we recall that $\fg^{p,q}=\fg^p\cap\fg_{p+q}$, where the grading $\fg=\oplus\fg^p$ [resp. $\fg =\oplus\fg_{\ell}$] is induced by $\ttE=\ttS^{\tti}$ [resp. $\ttH^{\b_1}=\ttH^{\a_{\tti}}$].  By Lemma \ref{L:conj+Ci} we have $h^{p,q}=h^{q,p}=h^{-q,-p}=h^{-p,-q}$. To verify Figure \ref{f:adj_bigr} we shall also need Lemma \ref{L:adjFvW} below, which gives $\dim(\fg_{\ell})=\dim(\fg^{-\ell})$. It also implies 
\[
  w(\fg^{p,q}) \ = \ w(\fg^p\cap\fg_{p+q}) \ = \ \fg^{-(p+q)}\cap\fg_{-p}
  \ = \ \fg^{-p-q,q}\,,
\]
where $w \in \sW$ is as in Lemma \ref{L:adjFvW}, and hence the additional symmetry $h^{p,q}=h^{-p-q,q}$ of the Hodge--Deligne numbers.

From $\S$\ref{S:AV}, we know that $\dim(\fg^{-2})=\dim(\fg^2)=1$ and $\dim(\fg^j)=0$ for $|j|>2$.  Hence $\dim(\fg_{-2})=\dim(\fg_2)=1$ (and $\dim(\fg_{\ell})=0$ for $|\ell |>2$), so that $h^{p,q}=h^{q,p}$ implies $\fg_{-2}=\fg^{-1,-1}$, $\fg_2=\fg^{1,1}$, $h^{1,1}=h^{-1,-1}=1$.  The ``extra symmetry'' now gives $h^{-2,1}=h^{2,-1}=h^{-1,2}=h^{1,-2}=1$, as well as $h^{1,0}=h^{-1,0}=h^{0,1}=h^{0,-1}=h^{-1,1}=h^{1,-1}=:a$.

Finally, to obtain the dimensions in Table \ref{t:dims} we solve
\begin{eqnarray*} 2a + 3 & = & \tdim_\bC\fg^- \ = \ \tdim_\bC\check D \ = \ n \,,\\  6a + b + 6 & = & \tdim_\bC \fg \ = \ N+1 \,. \end{eqnarray*}
using Table \ref{t:numerics}.

\subsection{A curious symmetry} \label{S:sym}

The main result of this section is Lemma \ref{L:adjFvW}, which asserts that (in the case of $\tti$ associated to a fundamental adjoint representation) the decompositions \eqref{E:gp}${}_{j=1}$ and \eqref{E:gl}${}_{j=1}$ are congruent under the action of the Weyl group.

The grading element corresponding to $P$ is $\ttS^\tti$.  By Table \ref{t:highest_rt}, the largest $\ttS^\tti$--eigenvalue on $\fg$ is $\tilde\a(\ttS^\tti) = 2$.  Therefore, the graded decomposition \eqref{E:gp}${}_{j=1}$ of $\fg$ into $\ttS^\tti$--eigenspaces is 
$$
  \fg \ = \ \fg^2 \,\op\, \fg^1 \,\op\, \fg^0 \,\op\, \fg^{-1} 
  \,\op\, \fg^{-2} \,.
$$
As discussed in $\S$\ref{S:Edecomp}, $\fg^2$ is one-dimensional; indeed,  
\begin{equation} \label{E:g^ta}
  \fg^{\pm2} \ = \ \fg^{\pm\tilde\a} \,.
\end{equation}

Let $\sW \subset \tAut(\rtL) \subset \tAut(\fh^*)$ denote the Weyl group of $\fg$.  Recall the $\ttH^{\a_\tti}$--eigenspace decomposition \eqref{E:gl}${}_{j=1}$, and the filtrations \eqref{E:FW}${}_{j=1}$.

\begin{lemma} \label{L:adjFvW}
There exists $w \in \sW$ such that $w(\fg_\ell) = \fg^{-\ell}$.  In particular, $w(W_\ell) = F^{-\ell}$.
\end{lemma}

\begin{corollary}
The filtration $F^\sb$ is the weight filtration associated to a nilpotent $0 \not= \tilde N \in \fg^{\tilde\a}$.
\end{corollary}

Before proving Lemma \ref{L:adjFvW} we first establish Lemma \ref{L:tS}.

\begin{lemma} \label{L:tS}
Define  $\tilde \ttS \in [\fg^{\tilde\a} \,,\, \fg^{-\tilde\a}] \subset \fh$ by $\tilde\a(\tilde\ttS) = 2$.  Then $\tilde \ttS = \ttS^\tti$.
\end{lemma}


\begin{proof}[Proof of Lemma \ref{L:tS}]
Let $0\not=\xi \in \fg^{\a_j}$.  Fix $\tilde N \in\fg^{\tilde\a}$ and $\tilde N^- \in \fg^{-\tilde\a}$ such that $[\tilde N , \tilde N^-] = \tilde\ttS$.  Then
\begin{eqnarray*}
  \a_j(\tilde\ttS) \, \x & = & \left[\tilde\ttS , \x \right] 
  \ = \ \big[ \big[ \tilde N , \tilde N^- \big] \,,\, \x \big] \\
  & = & \big[ \big[ \tilde N , \x \big] \,,\, \tilde N^- \big]
  \ + \ \big[ \big[ \x , \tilde N^- \big] \,,\, \tilde N \big] \,.
\end{eqnarray*}
Since $\tilde \a$ is a highest root, $\tilde\a+\a_j$ is not a root.  Therefore, $[\tilde N , \x] = 0$.  Similarly, the bracket $[\x , \tilde N^-]$ is nonzero if and only if $-\tilde\a+\a_j$ is a root; equivalently, $\tilde\a - \a_j$ is a root.  The latter is equivalent to $\fg^{-\a_j} \not\subset \fp$; that is, $j = \tti$.  Therefore, $\a_j(\tilde\ttS) = 0$ if $j \not= \tti$.  It follows that $\tilde\ttS$ is a multiple of $\ttS_\tti$.  Since both $\tilde\a(\ttS_\tti)$ and $\tilde\a(\tilde\ttS)$ equal 2, it must be the case that $\tilde\ttS = \ttS^\tti$.
\end{proof}

\begin{proof}[Proof of Lemma \ref{L:adjFvW}]
Note that $\a_\tti$ is not a short root.  Since the highest root $\tilde\a$ of $\fg$ is also not a short root, it follows $\a_\tti$ and $\tilde\a$ have the same length.  Therefore, there exists a Weyl group element $w \in \sW$ mapping $w(-\a_\tti) = \tilde\a$, \cf\cite[Lemma 10.4.C]{MR499562}.  Therefore,
$$
  w(\ttH^\tti) \ \in \ w \left[\fg^{-\a_\tti}\,,\,\fg^{\a_\tti} \right]
   \ = \ \left[w\fg^{-\a_\tti}\,,\,w\fg^{\a_\tti} \right]
   \ = \ \left[\fg^{\tilde\a}\,,\,\fg^{-\tilde \a} \right] \,.
$$
By Lemma \ref{L:tS}, $w(\ttH^\tti)$ is necessarily a multiple of $\tilde\ttS = \ttS^\tti$.  Moreover, 
$$
  2 \ = \ \a_\tti(\ttH^\tti) \ = \ -(w^{-1}\tilde\a)(\ttH^\tti) \ = \ -\tilde\a(w\ttH^\tti)
$$
forces $w(\ttH^\tti) = -\ttS^\tti$.  Given $\x \in \fg_\ell$, we have 
$$
  \ell\, w(\x) \ = \ w( \ell \x ) \ = \ w[ \ttH^\tti , \x] \ = \ -[\ttS^\tti , w(\x) ] \,.
$$
Thus $w(\x) \in \fg^{-\ell}$, and we conclude $w(\fg_\ell) = \fg^{-\ell}$.
\end{proof}

Table \ref{t:H} expresses the $\ttH^{\a_\tti}$ in terms of the grading elements $\{ \ttS^j\}$.
%
%
\begin{small}
\begin{table}[!ht]
\caption{The grading element $\ttH$.}
\renewcommand{\arraystretch}{1.2}
\begin{tabular}{|c|c|}
\hline $G$ & $\ttH$ \\ \hline\hline
 $D_4$ & $-\ttS^1 \,+\, 2 \ttS^2 \,-\, \ttS^3 \,-\,\ttS^4$ \\
 $B_r$, $D_r$ & $-\ttS^1 \,+\, 2 \ttS^2 \,-\, \ttS^3$ \\ 
 $E_6$ & $2 \ttS^2 \,-\, \ttS^4$ \\
 $E_7$ & $2 \ttS^1 \,-\, \ttS^3$ \\
 $E_8$ & $-\ttS^7 \,+\, 2 \ttS^8$ \\
 $F_4$ &  $2 \ttS^1 \,-\, \ttS^2$ \\
 $G_2$ & $- \ttS^1 \,+\, 2 \ttS^2$ \\ \hline
\end{tabular}
\label{t:H}
\end{table}
\end{small}

\begin{remark}\label{R:irrmod}
For the grading $\fg=\oplus\fg^p$ associated to a fundamental adjoint variety, we know that $\fg^{-1}$ is an irreducible $\fg^0$--module.  It will be important in the sequel that the symmetry in Lemma \ref{L:adjFvW} identifies $\fg^{-1}$ as a $\fg^0$-module with $\fg_1$ as a $\fg_0$--module.
\end{remark}

\subsection{Fundamental adjoint varieties II: boundary components and enhanced $\tSL_2$-orbits} \label{S:adII}

We continue with the notation of $\S$\ref{S:adI}, and write $N\in \fg^{-1,-1}\cap \fg_{\bQ}$ for the nil-negative element of the standard triple associated to $\b_1 =\a_{\tti}$.  In this section we briefly describe $X(N)$ and $B(N)$.

In each case it is clear from the description in $\S$\ref{S:B+D} that $\fg_{B(N)}^{ss}$ is the semisimple Lie subalgebra of $\fg$ whose roots are the ones strongly orthogonal to the $\a_{\tti}$. Since $\fg_1\cong\fg_{-1}=\fz_{-1}(N)$ is a faithful representation of $G^{ss}_{B(N)}$,\footnote{$G^{ss}_{B(N)}$ is the group $G_s$ ($s=1$) of $\S$\ref{S:sl2con}, but we will not write $G_1$ due to notational conflict with $\fg_1$ (which is obviously not its Lie algebra).} and $\phi\in\fg^{ss}_{B(N)}$, $G_{B(N)}^{ss}$ is the generic MT group of the Hodge structures on $\fg_1$ parameterized by $D(N)$.  Moreover, the action of $G_{B(N)}^{ss}$ on the line $\fg^{\tilde{\a}}=\fg^{2,-1}\subset \fg_1$ presents $\check D(N)$ as a subvariety of $\mathbb{P}\fg_1$.  Since the image of $\tilde{\a}$ under $w$ (cf. $\S$\ref{S:sym}) is the highest $\fg^0$-weight of $\fg^{-1}$, we obtain

\begin{proposition}
The compact dual $\check D(N)\subset\mathbb{P}\fg_1$ is isomorphic to the variety of lines $\cC_o\subset \mathbb{P}\fg^{-1}$.  We therefore have (in $\check D$) $X(N)\cong\mathbb{P}^1\times\check D(N)\cong\mathbb{P}^1\times\cC_o$.
\end{proposition}

\begin{proof}
We only need to check that the $\mathbb{P}^1$ factor of $X(N)$ is minimally embedded in $\check D$.  This follows from $\tilde{\a}(\ttH^{\a_{\tti}})=1$, since $\tilde{\a}$ is also the highest weight of $\fg$.
\end{proof}

Turning to the boundary component $B(N)\buildrel{\rho}\over{\twoheadrightarrow} D(N)$, we may think of points in $D(N)$ as Hodge flags $F^\sb$ on the Tate twist $U_N\dfn \fg_1(-1)=\fg_{-1}(-2)$ (weight 3 HS).  From \eqref{E:fibre} we see at once that the fibres of $\rho$ take the form $U_N/F^2 U_N$.

Now the object that provides a partial compactification of the quotient of $D$ by a neat arithmetic subgroup $\mathfrak{G}$ of $G_{\bQ}$, is the quotient $\overline{B(N)}$ of $B(N)$ by $\mathfrak{G}\cap Z(N)$, \cf\cite[$\S$7]{KP2012}.

\begin{proposition} \label{P:adjB(N)}
For each of the fundamental adjoint varieties, $\overline{B(N)}\buildrel{\bar{\rho}}\over{\twoheadrightarrow}\overline{D(N)}$ is a family of intermediate Jacobians associated to a variation of Hodge structure (with Hodge numbers $(1,a,a,1)$) over a Shimura variety.  The corresponding variations $\mathcal{U}_N\to D(N)$ recover (with one exception) the list of weight 3 maximal Hermitian VHS of Calabi-Yau type from \rm{\cite{FL}}.
\end{proposition}

\begin{proof}
Immediate from comparison with \cite{FL} and the fact that $G_{B(N)}^{ss}$ is the MT group of $\mathcal{U}_N$.
\end{proof}

\begin{remark}\label{R:adjB(N)}
We refer to Corollary 2.29 and Theorem 6.7 of \cite{FL} for further description of these maximal variations.  The missing case, labeled $\mathrm{I}_{1,n}$ in \cite[Corollary 2.29]{FL}, arises in the same way from a boundary component of the (non-fundamental) adjoint variety for $A_{n+2}$.

With this addition, the compact duals $\check D(N)$ yield the homogeneous Legendrian varieties as presented in \cite[Theorem  6.7]{FL}; whereas the varieties of lines $\cC_o$ realize them (as in \cite{MR2372722}) as the subadjoint varieties.  The Weyl flip of $\S$\ref{S:sym} toggles between these realizations.  To the authors, this suggests comparing the homology classes of $X(N)$ and the Schubert variety $\mathit{Cone}(\cC_o)$, \cf $\S$\ref{S:[X(N)]eg}.
\end{remark}

The data for all the fundamental adjoint varieties is summarized in Table \ref{t:FAV}, where we note that $\dim(X(N))=a+1$.
\small
\begin{table}
\caption{$D(N)$ and $X(N)$ for the fundamental adjoint varieties}
\renewcommand{\arraystretch}{1.2}
\begin{tabular}{|c||c|c|c|c|c|c|}
\hline $\fg$ & $\mathfrak{so}(n+4)$ ($n\geq 3$) & $\mathfrak{e}_6$ & $\mathfrak{e}_7$ & $\mathfrak{e}_8$ & $\mathfrak{f}_4$ & $\mathfrak{g}_2$ \\ \hline \hline
$\fg_{\bR}$ & $\mathfrak{so}(4,n)$ & $\mathrm{E\, II}$ & $\mathrm{E\, VI}$ & $\mathrm{E\, IX}$ & $\mathrm{F\, I}$ & $\mathrm{G}$\\
$\fg^{ss}_{B(N),\bR}$ & \tiny $\mathfrak{so}(2,n-2)\oplus \mathfrak{su}(1,1)$ & $\mathfrak{su}(3,3)$ & $\mathfrak{so}^*(12)$ & $\mathrm{E\, VII}$ & $\mathfrak{sp}(3,\bR)$ & $\mathfrak{su}(1,1)$ \\
$D(N)$ & $\sH\times\mathrm{IV}_{n-2}$ & $\mathrm{I}_{3,3}$ & $\mathrm{II}_6$ & $\mathrm{E\, VII}$ & $\mathrm{III}_3$ & $\sH$ \\
$X(N)$ & \tiny $\mathbb{P}^1\times\mathbb{P}^1\times\mathcal{Q}^{n-2}$ & \tiny $\mathbb{P}^1\times Gr(3,\bC^6)$ & \tiny $\mathbb{P}^1\times \mathcal{S}_6$ & \tiny $\mathbb{P}^1\times E_7/P_7$ & \tiny $\mathbb{P}^1\times LG(3,\bC^6)$ & \tiny $\mathbb{P}^1\times v_3(\mathbb{P}^1)$ \\
$\dim(X(N))$ & $n$ & $10$ & $16$ & $28$ & $7$ & $2$ \\ \hline
\end{tabular}
\label{t:FAV}
\end{table}
\normalsize
The type of $D(N)$ as a Hermitian symmetric domain comes from \cite{FL}; otherwise, see \cite[Appendix A]{KR1long} for more details.

\begin{proposition} \label{P:X(N)max}
The $X(N)$ of Table \ref{t:FAV} are all maximal horizontal subvarieties of $\check D$.
\end{proposition}

\begin{proof}
By \cite[Corollary 3.13]{MR3217458}, the dimension of a horizontal manifold is bounded by the maximal dimension of the horizontal Schubert varieties (HSV) in $\check D$.  So \cite[Corollary 2.18]{KR2} establishes the proposition for $\fg = \fso(n+4)$.  Likewise \cite[Example 5.9, and Corollaries 5.13 and 5.29]{MR3217458} yield Proposition \ref{P:X(N)max} for the exceptional $\fg = \fe_6 , \ff_4, \fg_2$.

Computing as in \cite[$\S$5]{MR3217458}, one may confirm that the maximal HSV in the $E_7$--adjoint variety $\check D = E_7(\bC)/P_1$ are all of dimension 16.  Indeed, they are the $X_w$ with $\Delta(w) = \{ \a \in \Delta \ | \ \a(\ttS^1) = 1 \,,\ \a(\ttT_w) \le 0 \}$ given by (the first column of) Table \ref{t:maxE}.  Similarly, the maximal HSV in the $E_8$--adjoint variety $\check D = E_8(\bC)/P_8$ are all of dimension 28; they are the $X_w$ with $\Delta(w) = \{ \a \in \Delta \ | \ \a(\ttS^8) = 1 \,,\ \a(\ttT_w) \le 0 \}$ given by (the second column of) Table \ref{t:maxE}.
\end{proof}

\small
\begin{table}
\caption{The maximal horizontal Schubert varieties in the $E_7$ and $E_8$--adjoint varieties}
\renewcommand{\arraystretch}{1.2}
\begin{tabular}{|c|c|}
\hline 
\multicolumn{2}{|c|}{The grading element $\ttT_w$} \\ \hline
  $\fe_7$ & $\fe_8$ \\ \hline
  $-\ttS^1+\ttS^3$ & $\ttS^2-\ttS^8$ \\
  $-\ttS^1+\ttS^5$ & $\ttS^5 - 2\ttS^8$ \\
  $-2\ttS^1+\ttS^3+\ttS^6$ & $\ttS^2+\ttS^6-3\ttS^8$ \\
  $-3\ttS^1+\ttS^3+\ttS^5+\ttS^7$ & $\ttS^2+\ttS^5+\ttS^7-5\ttS^8$ \\
  $-2\ttS^1+\ttS^4+\ttS^7$ & $\ttS^4+\ttS^7-4\ttS^8$ \\
  $-\ttS^1+\ttS^2+\ttS^7$ & $\ttS^3+\ttS^7-3\ttS^8$ \\
  $\ttS^7$ & $\ttS^1+\ttS^7-2\ttS^8$ \\
  & $\ttS^7-\ttS^8$.
\\ \hline
\end{tabular}
\label{t:maxE}
\end{table}
\normalsize
%

\subsection{Fundamental adjoint varieties III: orbits of higher codimension} \label{S:adIII}

We now show that the construction of $\S$\ref{S:details} yields all the $G_{\bR}$-orbits in $\tbd(D)$ for each of the fundamental adjoint varieties, and indicate the LMHS types.

\begin{proposition}\label{P:adIII}
Let $\check{D}$ be a fundamental adjoint variety, $D$ the unique open $G_{\bR}$-orbit with the structure of a Mumford--Tate domain, and $\mathcal{O}\subset \tbd(D)$ a $G_{\bR}$-orbit. Then $\mathcal{O}$ is of the form $\mathcal{O}_{\mathcal{B}}$, hence is polarizable (and can be reached by a sequence of Cayley transforms in strongly orthogonal ``horizontal'' roots). Moreover, with the exception of type $D_4$, there is no more than one $G_{\bR}$-orbit in each codimension.
\end{proposition}

\begin{proof}
By Remark \ref{R:HSDcase}, all orbits in $\tbd(D(N))/G^{ss}_{B(N),\bR}$ are obtained (from $D(N)$) by a sequence of Cayley transforms in strongly orthogonal roots. (Outside the reducible cases --- $G$ of type $B_r,D_r$, and $D(N)\cong \sH\times \mathrm{IV}_{n-2}$ --- we need just one linear sequence of CTs.) As $Gr^W_{\pm 1}\fg$ is a faithful representation of $G^{ss}_{B(N)}$, we have $\check{D}(N)\cong G^{ss}_{B(N)}(\bC)\cdot F^{\bullet} \subseteq G(\bC) \cdot F^{\bullet} \cong \check{D}$, hence we may view $D(N)\subset \cO_1 \subset \tbd(D)$.  From Figure \ref{f:adj_bigr} it is evident that $\alpha_{\tti}$ is strongly orthogonal to all roots of $\fg^{ss}_{B(N)}$. Thus the orbits in the image of
\begin{equation*}
\Theta:\; \tbd(D(N))/G^{ss}_{B(N),\bR} \twoheadrightarrow \{G_{\bR}\cdot\tbd(D(N))\}/G_{\bR} \subseteq \tbd(D)^{pol}/G_{\bR} \subseteq \tbd(D)/G_{\bR}
\end{equation*}
are all of the form $\cO_{\cB}$. We claim that (in all but one case) $\Theta$ is surjective.

To show this, we shall make use of Matsuki duality \cite[$\S$6.6]{MR1247501}), which yields a bijection between the sets of $K_{\bC}$- and $G_{\bR}$-orbits in $\check{D}$, with the posets defined by orbit closure relations exactly opposite.  Let $\cO_{\cB}=G_{\bR}\cdot o_{\cB}$ be one of the orbits in $\text{image}(\Theta)$, with Hodge-Deligne numbers $\{ h_{\cB}^{p,q}\}_{-2\leq p,q\leq 2}$. By \cite[$\S\S$5.3-4]{MR3331177} (and using $h^{0,0}_D = h^{0,0}_{\cB} + 2h^{0,1}_{\cB} + 2h^{0,2}_{\cB}$) one has the dimension formulas
\begin{equation*}
c_{\cB} := \text{codim}_{\bR}(G_{\bR}\cdot o_{\cB}) = h^{1,1}_{\cB}+2h^{1,2}_{\cB}+h^{2,2}_{\cB}
\end{equation*}
and
\begin{eqnarray*}
k_{\cB} & := & \dim_{\bR}(K_{\bR}\cdot o_{\cB}) = \dim_{\cR}(\fk_{\bR}\cap \fg^{0,0}_{\cB}) \\
&=& (h^{0,0}_D + h^{2,-2}_D + h^{-2,2}_D) - (h^{0,0}_{\cB} - h^{1,1}_{\cB} + h^{2,2}_{\cB}) \\
&=& 2h^{0,1}_{\cB} + 2h^{0,2}_{\cB} + h^{1,1}_{\cB} - h^{2,2}_{\cB} + 2.
\end{eqnarray*}
Since $\text{codim}_{\bR}(K_{\bR}\cdot o_{\cB}) = \text{codim}_{\bR} (K_{\bC}\cdot o_{\cB}) + \text{codim}_{\bR} (G_{\bR}\cdot o_{\cB})$, we find in addition that
\begin{equation*}
\mu_{\cB}:= \dim_{\bC} (K_{\bC}\cdot o_{\cB}) -1 = \tfrac{1}{2} (c_{\cB}+k_{\cB}) - 1 = h^{0,1}_{\cB}+ h^{0,2}_{\cB} + h^{1,1}_{\cB} + h^{1,2}_{\cB} ,
\end{equation*}
which we have normalized so that $\mu_D = 0$.

Using the description of ${}_1\fg_{\pm 1}$ as ${}_1\fg_0$-modules, we may compute the effects of the Cayley transforms on the Hodge-Deligne numbers, obtaining the list in Table \ref{t:nd},\footnote{here $SO(n+4)$ corresponds to $B_{r}$ with $n=2r-3$ [resp. $D_{r}$ with $n=2r-4$]} except for the middle orbit in the last row. (The parentheses mean that it is {\it not} in $\text{im}(\Theta)$.) 
\small
\begin{table}
\caption{Normalized dimensions of $K_{\bC}$-orbits (dual to $G_{\bR}$-orbits in $\text{image}(\Theta)$).}
\renewcommand{\arraystretch}{1.2}
\begin{tabular}{|c||c|c|}
\hline $G$ & $\mu_{\cB}$ on $\text{im}(\Theta)$ &  $c_{\cB}$ on $\text{im}(\Theta)$ \\ \hline \hline
$SO(7)$ & $3, 4, 5, 6$ & $1, 3, 4, 7$ \\
$SO(8)$ & $4, 5, 5, 5, 7, 8$ & $1, 4, 4, 4, 5, 9$ \\
{\tiny $\begin{matrix} SO(n+4) \\ (n\geq 5) \end{matrix}$}  & $ n, 2n-3, n+1, 2n-1, 2n$ & $ 1, 4, n, n+1, 2n+1$ \\
$E_6$ & $10, 15, 19, 20 $ & $1, 6, 11, 21 $ \\
$E_7$ & $16, 25, 31, 32$ & $1, 8, 17, 33$ \\
$E_8$ & $28, 45, 55, 56$ & $1, 12, 29, 57$ \\
$F_4$ & $7, 10, 13, 14$ & $1, 5, 8, 15$ \\
$G_2$ & $2, (3,) 4$ & $1, (3,) 5$ \\ 
\hline
\end{tabular}
\label{t:nd}
\end{table}
\normalsize
Finally, using {\tt kgporder} in ATLAS \cite{ATLAS} to compute the poset of $K_{\bC}$-orbits in $\check{D}$, and interpreting the results by Matsuki duality, we confirm that there are 3 ($G_2$), resp. 4 ($B_3, E_6, E_7, E_8, F_4$), 5 ($B_{r\geq 4}, D_{r\geq 5}$), 6 ($D_4$) $G_{\bR}$-orbits in $\tbd(D)$.  (In fact, the $\mu_{\cB}$ numbers can be shown to match using {\tt kgp}.) Therefore $\Theta$ is surjective except in the $G_2$ case, where one knows the remaining orbit is of the form $\cO_{\beta_1}$ by Example \ref{eg:codimO1}.
\end{proof}

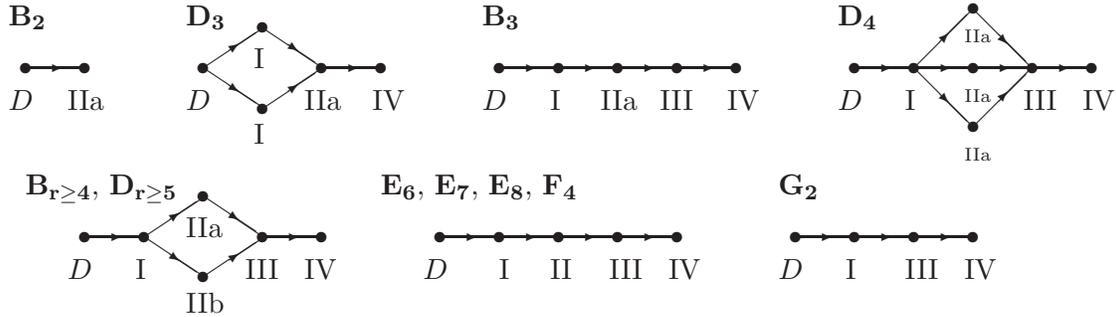
\begin{figure}
\caption{Hasse diagrams for $\text{cl}(D)$ in (mostly) fundamental adjoint cases.}
\begin{center}
\setlength{\unitlength}{3.2pt}
\begin{picture}(125,50)(0,-5)
\put(-2,35){$\mathbf{B_2}$}
\multiput(0,30)(7,0){2}{\circle*{1.2}}
\put(0,30){\Vector[m](1,0){7}}
\put(-2,25){$D$}\put(5,25){IIa}
\put(19,35){$\mathbf{D_3}$}
\put(21,30){\circle*{1.2}}\put(28,34.8){\circle*{1.2}}\put(28,25.2){\circle*{1.2}}\multiput(35,30)(7,0){2}{\circle*{1.2}}
\put(35,30){\Vector[m](1,0){7}}\put(21,30){\Vector[m](3,2){7}}\put(21,30){\Vector[m](3,-2){7}}\put(28,25.2){\Vector[m](3,2){7}}\put(28,34.8){\Vector[m](3,-2){7}}
\put(19,25){$D$}\put(27,21){I}\put(27,30){I}\put(33,25){IIa}\put(41,25){IV}
\put(54,35){$\mathbf{B_3}$}
\multiput(56,30)(7,0){5}{\circle*{1.2}}
\multiput(56,30)(7,0){4}{\Vector[m](1,0){7}}
\put(54,25){$D$}\put(62,25){I}\put(68,25){IIa}\put(75,25){III}\put(83,25){IV}
\put(96,35){$\mathbf{D_4}$}
\multiput(98,30)(7,0){5}{\circle*{1.2}}\put(112,37){\circle*{1.2}}\put(112,23){\circle*{1.2}}
\multiput(98,30)(7,0){4}{\Vector[m](1,0){7}}\put(105,30){\Vector[m](1,1){7}}\put(105,30){\Vector[m](1,-1){7}}\put(112,37){\Vector[m](1,-1){7}}\put(112,23){\Vector[m](1,1){7}}
\put(96,25){$D$}\put(104,25){I}\put(111,19){\tiny IIa}\put(111,26){\tiny IIa}\put(111,33){\tiny IIa}\put(118,25){III}\put(125,25){IV}
\put(0,15){$\mathbf{B_{r\geq 4} ,\, D_{r\geq 5}}$}
\multiput(7,10)(7,0){2}{\circle*{1,2}}\multiput(28,10)(7,0){2}{\circle*{1,2}}\put(21,14.8){\circle*{1.2}}\put(21,5.2){\circle*{1.2}}
\put(7,10){\Vector[m](1,0){7}}\put(28,10){\Vector[m](1,0){7}}\put(14,10){\Vector[m](3,2){7}}\put(14,10){\Vector[m](3,-2){7}}\put(21,14.8){\Vector[m](3,-2){7}}\put(21,5.2){\Vector[m](3,2){7}}
\put(5,5){$D$}\put(13,5){I}\put(19,1){IIb}\put(19,10){IIa}\put(26,5){III}\put(33,5){IV}
\put(42,15){$\mathbf{E_6,\, E_7,\, E_8,\, F_4}$}
\multiput(49,10)(7,0){5}{\circle*{1.2}}
\multiput(49,10)(7,0){4}{\Vector[m](1,0){7}}
\put(47,5){$D$}\put(56,5){I}\put(62,5){II}\put(69,5){III}\put(76,5){IV}
\put(89,15){$\mathbf{G_2}$}
\multiput(91,10)(7,0){4}{\circle*{1.2}}
\multiput(91,10)(7,0){3}{\Vector[m](1,0){7}}
\put(89,5){$D$}\put(97,5){I}\put(104,5){III}\put(111,5){IV}
\end{picture}
\label{f:hasse}
\end{center}
\end{figure}

The poset structure on $\text{cl}(D)=D\cup \tbd(D)$ in each case is displayed in Figure \ref{f:hasse}, where we include the (non-fundamental) adjoint varieties for $B_2$ and $D_3$ for completeness, and arrows indicate the orbit closure relation \[ \cO \geq \cO ' \iff \overline{\cO} \supseteq \cO ' . \] The Roman numerals designate the LMHS types (on $\fg$) corresponding to each $G_{\bR}$-orbit:\footnote{For the period domain case $G=SO(n+4)$ ($D=D_{(2,n,2)}$), types I, II, III, IV, and V of \cite{KPR} correspond to our types I, IIa, IIb, III, and IV, respectively.}
\begin{description}
\item[I] LMHS as in Figure \ref{f:adj_bigr}.
\item[II] LMHS as in Figure \ref{f:adj_bigr2}, with a bifurcation into {\bf IIa} and {\bf IIb} for $B_r,D_r$.
\item[III] $90^{\circ}$ rotation of Figure \ref{f:adj_bigr}.
\item[IV] Hodge-Tate LMHS.
\end{description}

\begin{figure}
\caption{The bigrading $\fg = \op \fg^{p,q}$ for type {\bf II}.}
\begin{center}
\setlength{\unitlength}{2.0pt}
\begin{picture}(50,50)(-22,-22)
\multiput(-20,0)(10,0){5}{\circle*{1.5}}
\multiput(-10,-10)(10,0){3}{\circle*{1.5}}
\multiput(-10,10)(10,0){3}{\circle*{1.5}}
\put(0,20){\circle*{1.5}}
\put(0,-20){\circle*{1.5}}
\thicklines
\put(0,0){\circle{2.8}} \put(0,0){\circle{5.0}}
\multiput(-10,0)(20,0){2}{\circle{3.0}}
\multiput(-11,-10.6)(20,0){2}{\framebox(2,2)}
\multiput(0,-10)(0,20){2}{\circle{3.0}}
\multiput(-11,9.4)(20,0){2}{\framebox(2,2)}
\thinlines
\put(0,0){\vector(1,0){25}} \put(26,0){\small{$p$}}
\put(0,0){\vector(-1,0){25}} \put(1,25){\small{$q$}}
\put(0,0){\vector(0,1){25}}
\put(0,0){\vector(0,-1){25}}
\end{picture}
\label{f:adj_bigr2}
\end{center}
\end{figure}

\noindent The Hodge-Deligne numbers for type {\bf II} are displayed in Table \ref{t:typeII}.

\begin{footnotesize}
\begin{table}
\caption{Dimensions of the $\fg^{p,q}$ in Figure \ref{f:adj_bigr2}.}
\renewcommand{\arraystretch}{1.2}
\begin{tabular}{|c||c|c|c|c|c|c|}
  \hline
  $G$ & {\tiny $\begin{matrix} SO(n+4) \\ (n\geq 1) \end{matrix}$} ({\bf IIa}) & {\tiny $\begin{matrix} SO(n+4) \\ (n\geq 5) \end{matrix}$}({\bf IIb}) & $E_6$ & $E_7$ & $E_8$ & $F_4$ 
  \\ \hline
  {\setlength{\unitlength}{2.5pt}
   \begin{picture}(1,1)(-0.5,-0.7)
     \put(0,0){\circle*{1.3}} 
   \end{picture}
  } 
  & 1 & 1 & 1 & 1 & 1 & 1  \\ \hline
  {\setlength{\unitlength}{2.5pt}
   \begin{picture}(1,1)(-0.5,-0.7)
     \put(0,0){\circle*{1.3}} 
     \thicklines\put(0,0){\circle{2.5}} \thinlines
   \end{picture}
  } 
  & $0$ & $2n-8$
  & 8 & 16 & 32 & 4  \\
  \hline
 {\setlength{\unitlength}{2.5pt}
   \begin{picture}(1,1)(-0.5,-0.7)
     \put(-1,-0.7){\framebox(2,2)} 
     \thicklines\put(0,0){\circle*{1.3}} \thinlines
   \end{picture}
  } 
  & $n$ & $4$
  & 6 & 8 & 12 & 5  \\
  \hline
  {\setlength{\unitlength}{2.5pt}
   \begin{picture}(1,1)(-0.5,-0.8)
     \put(0,0){\circle*{1.3}} 
     \thicklines\put(0,0){\circle{2.5}} \put(0,0){\circle{3.8}} \thinlines
   \end{picture}
  } 
  & $\tfrac{1}{2}n^2 - \tfrac{1}{2}n+ 2$ & $\tfrac{1}{2}n^2 - \frac{9}{2}n+18$  & 18 & 33 & 68 & 12  \\
  \hline
\end{tabular}
\label{t:typeII}
\end{table}
\end{footnotesize}

\subsection{Computing $[X(N)]$: a simple example} \label{S:[X(N)]eg}

One purpose of introducing the $X(\sigma)$ was to produce smooth algebraic representatives of classes in $H_*(\check D,\mathbb{Z})$.  In particular, for the $G_2$ and $F_4$ adjoint varieties, none of the maximal HSV are smooth.  So it seems natural to conclude this paper by computing $[X(N)]$ for the $G_2$-adjoint variety.\footnote{The analogous computation of $[X(N)]$ in the case of the $\tSO(2r+1,\bC)$--adjoint variety $\check D$ is worked out in \cite[Appendix B]{KR1long}. The systematic computation of $[X(\sigma)]$ will be taken up in a subsequent paper, using different methods.}

Recall the Schubert variety $X$ of \eqref{E:Xo}.  Because $\tdim_\bC X(N) = 2$ and $H_4(\check D , \bZ) = \bZ[X]$, the class $[X(N)]$ will necessarily be a multiple of $[X]$.  We will determine this multiple.

We begin with the description of $X$ as a Tits transform; see \cite[\S4]{KR2} for more detail.  Let $Q \subset G_2$ be the maximal parabolic associated to the first simple root, so that $G_2/Q$ is a five dimensional quadric, and the minimal homogeneous embedding lies in $\bP V_{\w_1}$.  Then $G_2/Q$ parameterizes a uniruling of $\check D$ by $\bP^1s$, \cf Table \ref{t:adjvar_lines}.  Let $G'$ be the connected simple subgroup of $G_2$ associated to the simple root $\a_1$, and let $\Sigma \subset G_2/Q$ be the $G'$--orbit of the identity coset $Q/Q$.  By \cite[(4.7a) and Lemma 4.10]{KR2}, the variety $X$ is the Tits transform $\cT(\Sigma)$.  If $\eta$ is the fundamental weight of $G'$, then $\w_1$ restricts to $\eta$ on $\fg'$.  It follows that $\Sigma$ is a minimal homogeneous embedding of $\bP^1$.

Let $\Sigma(N)$ denote the $G^\mathrm{ss}_{B(N)}(\bC)$--orbit of $Q/Q$ in $G_2/Q$.  Then $X(N)$ is the Tits transform $\cT(\Sigma(N))$.  Note that the simple root $2 \a_1 + \a_2$ of $G^\mathrm{ss}_{B(N)}$ is image of $\a_1$ under the Weyl group element $w = (12)\in\sW$.  It follows that $G^\mathrm{ss}_{B(N)} = \tAd_w(G')$.  As a consequence we may identify $\eta$ with the fundamental weight of $G^\mathrm{ss}_{B(N)}$.  

\begin{claim*}
The restriction of $\w_1$ to $G^\mathrm{ss}_{B(N)}$ is $2\eta$.  
\end{claim*}

\noindent It follows that $\Sigma(N)$ is the second Veronese embedding of $\bP^1$.  Since $\tdim_\bC \Sigma(N) = 1$ and $H_2(G_2/Q,\bZ) = \bZ[\Sigma]$, we may conclude that
$$
  [\Sigma(N)] \ = \ 2 \,[\Sigma] \,.
$$
It follows from \cite[Lemmas 3.11 and 3.13]{MR3130568} that 
$$
  [X(N)] \ = \ 2\,[X] \,.
$$

\begin{proof}[Proof of claim]
The Lie algebra $\fg^\tss_{B(N)}(\bC)$ is $\fg^\a \,\op\, \bC\,\ttH^\a \,\op\,\fg^{-\a}$ where $\a = 2\a_1 + \a_2$.  The fundamental weight $\eta$ is defined by $\eta(\ttH^\a) = 1$.  On the other hand $\ttH^\a = \ttS^1$, as an element of the Cartan subalgebra of $\fg_2$, and $\w_1 = 2\a_1 + \a_2$.  Therefore $\w_1(\ttH^\a) = 2$, and the claim follows.
\end{proof}

\appendix

\section{Dynkin diagrams} \label{S:dynkin}
For the reader's convenience we include in Figure \ref{f:dynkin} the Dynkin diagrams of the complex simple Lie algebras.  Recall that: each node corresponds to a simple root $\a_i \in \sS$; two nodes are connected if and only if $\langle \a_i , x\a_j \rangle \not=0$ and in this case the number if edges is $|\a_i|^2/|\a_j|^2 \ge 1$ (that is, $i,j$ are ordered so that the inequality holds).  Below, if $G = B_r$, then $r \ge 3$; and if $G = D_r$, then $r \ge 4$.
%
%
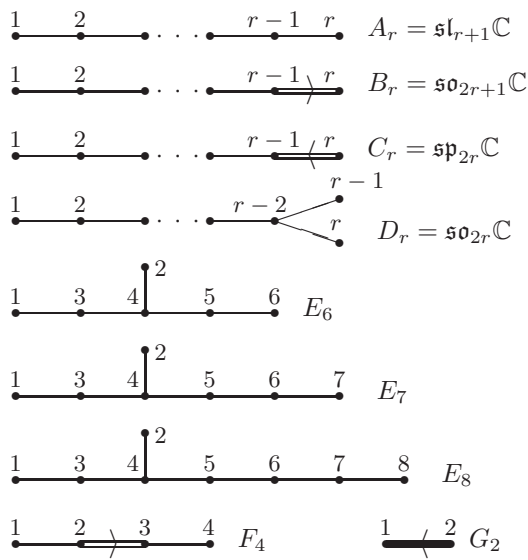
\begin{figure}[!ht] 
\caption{Dynkin diagrams for complex simple Lie algebras.}
\begin{center}
\setlength{\unitlength}{3.5pt}
\begin{picture}(50,65)(0,-5)
\multiput(0,55)(7,0){3}{\circle*{0.8}}
\put(0,55){\line(1,0){14}}
\multiput(15.5,55)(2,0){3}{\circle*{0.3}}
\multiput(21,55)(7,0){3}{\circle*{0.8}}
\put(21,55){\line(1,0){14}}
\put(-0.7,56){\footnotesize{1}}
\put(6.3,56){\footnotesize{2}}
\put(25,56){\footnotesize{$r-1$}}
\put(33.3,56){\footnotesize{$r$}}
\put(38,55){\small{$A_r = \fsl_{r+1}\bC$}}
\multiput(0,49)(7,0){3}{\circle*{0.8}}
\put(0,49){\line(1,0){14}}
\multiput(15.5,49)(2,0){3}{\circle*{0.3}}
\multiput(21,49)(7,0){3}{\circle*{0.8}}
\put(21,49){\line(1,0){7}}
\multiput(28,48.75)(0,0.5){2}{\line(1,0){7}}
\put(31.3,48.4){$\rangle$}
\put(-0.7,50){\footnotesize{1}}
\put(6.3,50){\footnotesize{2}}
\put(25,50){\footnotesize{$r-1$}}
\put(33.3,50){\footnotesize{$r$}}
\put(38,49){\small{$B_r = \fso_{2r+1}\bC$}}
\multiput(0,42)(7,0){3}{\circle*{0.8}}
\put(0,42){\line(1,0){14}}
\multiput(15.5,42)(2,0){3}{\circle*{0.3}}
\multiput(21,42)(7,0){3}{\circle*{0.8}}
\put(21,42){\line(1,0){7}}
\multiput(28,41.75)(0,0.5){2}{\line(1,0){7}}
\put(31.3,41.4){$\langle$}
\put(-0.7,43){\footnotesize{1}}
\put(6.3,43){\footnotesize{2}}
\put(25,43){\footnotesize{$r-1$}}
\put(33.3,43){\footnotesize{$r$}}
\put(38,42){\small{$C_r = \fsp_{2r}\bC$}}
\multiput(0,35)(7,0){3}{\circle*{0.8}}
\put(0,35){\line(1,0){14}}
\multiput(15.5,35)(2,0){3}{\circle*{0.3}}
\multiput(21,35)(7,0){2}{\circle*{0.8}}
\put(21,35){\line(1,0){7}}
\multiput(35,32.67)(0,4.66){2}{\circle*{0.8}}
\put(28,35){\line(3,1){7}}\put(28,35){\line(3,-1){7}}
\put(-0.7,36){\footnotesize{1}}
\put(6.3,36){\footnotesize{2}}
\put(23.5,36){\footnotesize{$r-2$}}
\put(34,38.5){\footnotesize{$r-1$}}
\put(34,34){\footnotesize{$r$}}
\put(39,33){\small{$D_r = \fso_{2r}\bC$}}
\multiput(0,25)(7,0){5}{\circle*{0.8}}
\put(0,25){\line(1,0){28}}
\put(14,30){\circle*{0.8}} 
\put(14,30){\line(0,-1){5}} 
\put(-0.7,26){\footnotesize{1}}
\put(6.3,26){\footnotesize{3}}
\put(15,29){\footnotesize{2}}
\put(12,26){\footnotesize{4}}
\put(20.3,26){\footnotesize{5}}
\put(27.3,26){\footnotesize{6}}
\put(31,24.7){\small{$E_6$}}
\multiput(0,16)(7,0){6}{\circle*{0.8}}
\put(0,16){\line(1,0){35}}
\put(14,21){\circle*{0.8}} \put(14,21){\line(0,-1){5}}
\put(-0.7,17){\footnotesize{1}}
\put(6.3,17){\footnotesize{3}}
\put(15,20){\footnotesize{2}}
\put(12,17){\footnotesize{4}}
\put(20.3,17){\footnotesize{5}}
\put(27.3,17){\footnotesize{6}}
\put(34.3,17){\footnotesize{7}}
\put(39,15.7){\small{$E_7$}}
\multiput(0,7)(7,0){7}{\circle*{0.8}}
\put(0,7){\line(1,0){42}}
\put(14,12){\circle*{0.8}} \put(14,12){\line(0,-1){5}}
\put(-0.7,8){\footnotesize{1}}
\put(6.3,8){\footnotesize{3}}
\put(15,11){\footnotesize{2}}
\put(12,8){\footnotesize{4}}
\put(20.3,8){\footnotesize{5}}
\put(27.3,8){\footnotesize{6}}
\put(34.3,8){\footnotesize{7}}
\put(41.3,8){\footnotesize{8}}
\put(46,6.7){\small{$E_8$}} 
\multiput(0,0)(7,0){4}{\circle*{0.8}} 
\put(0,0){\line(1,0){7}}
\put(7,0.3){\line(1,0){7}} \put(7,-0.3){\line(1,0){7}}
\put(10,-0.7){$\rangle$}
\put(14,0){\line(1,0){7}} \put(24,-0.3){\small{$F_4$}}
\put(-0.7,1){\footnotesize{1}}
\put(6.3,1){\footnotesize{2}}
\put(13.3,1){\footnotesize{3}}
\put(20.2,1){\footnotesize{4}}
\multiput(40,0)(7,0){2}{\circle*{0.8}}
\put(40,0){\line(1,0){7}}\put(40,0.3){\line(1,0){7}}
\put(40,-0.3){\line(1,0){7}}
\put(43.5,-0.7){$\langle$} \put(49,-0.3){\small{$G_2$}}
\put(39.3,1){\footnotesize{1}}
\put(46.3,1){\footnotesize{2}}
\end{picture}
\label{f:dynkin}
\end{center}
\end{figure}
\input{KR1.bbl}


\end{document}

%% file: macros.tex
\def\naive{{na\"ive}}

\def\cf{cf.~}

\def\sqrtminusone{\hbox{\footnotesize{$\sqrt{-1}$}}}

\newcommand{\hsp}[1]{{\hbox{\hspace{#1}}}}

\newcounter{letcnt} 

\def\a{\alpha}  
\def\b{\beta}  
\def\d{\delta}

\def\z{\zeta}
\def\m{\mu}
\def\n{\nu}
\def\s{\sigma}

\def\w{\omega} 

\def\x{\xi}

\def\fa{\mathfrak{a}} 
  
\def\tAd{\mathrm{Ad}} \def\tad{\mathrm{ad}}
 
\def\tAut{\mathrm{Aut}}
\def\cB{\mathcal B}

\def\fb{\mathfrak{b}} 
 \def\tbd{\mathrm{bd}}
\def\bC{\mathbb C} \def\cC{\mathcal C}
\def\sC{\mathscr{C}} 
\def\fc{\mathfrak{c}}  
\def\bc{\mathbf{c}}
 
\def\tcodim{\mathrm{codim}} 
 \def\sD{\mathscr{D}}

\def\tdeg{\mathrm{deg}} 

 \def\tdim{\mathrm{dim}}
 
 \def\ttE{\mathtt{E}}
\def\tEnd{\mathrm{End}} 

 \def\fe{\mathfrak{e}} 
 \def\teven{\mathrm{even}}
\def\texp{\mathrm{exp}}
\def\cF{\mathcal F} 

\def\ff{\mathfrak{f}}

\def\tFlag{\mathrm{Flag}}

\def\sG{\mathscr{G}}

\def\tGr{\mathrm{Gr}}
\def\fg{{\mathfrak{g}}}

 \def\ttH{\mathtt{H}}
 \def\sH{\mathscr{H}}
 
\def\tHom{\mathrm{Hom}}
\def\fh{\mathfrak{h}} 

\def\bi{\mathbf{i}} 
\def\tti{\mathtt{i}}

 \def\tIm{\mathrm{Im}}
 \def\tim{\mathrm{im}}

\def\fk{\mathfrak{k}} 
\def\bbk{\mathbbm{k}}
 
 \def\tker{\mathrm{ker}}

\def\tlim{\mathrm{lim}}
\def\tLG{\mathrm{LG}}

\def\sfN{\mathsf{N}} \def\sN{\mathscr{N}}

 \def\cO{\mathcal O}
 \def\tOG{\mathrm{OG}}

 \def\todd{\mathrm{odd}}

\def\bP{\mathbb P} 
\def\sP{\mathscr{P}} 
 
\def\fp{\mathfrak{p}} 
 
\def\tprim{\mathrm{prim}}

\def\bQ{\mathbb Q} \def\cQ{\mathcal Q} 
 
\def\fq{\mathfrak{q}} 
 
\def\bR{\mathbb R} \def\cR{\mathcal R}

\def\fr{\mathfrak{r}}

\def\trank{\mathrm{rank}}
 \def\cS{\mathcal S}
\def\sS{\mathscr{S}}
  
\def\ttS{\mathtt{S}} 
\def\fs{\mathfrak{s}}
 
\def\tss{\mathrm{ss}}
\def\tSL{\mathrm{SL}} \def\tSO{\mathrm{SO}}

\def\tStab{\mathrm{Stab}} \def\tSym{\mathrm{Sym}}

\def\tSU{\mathrm{SU}}
 \def\tspan{\mathrm{span}}
\def\fsl{\mathfrak{sl}} \def\fso{\mathfrak{so}} 
\def\fsp{\mathfrak{sp}} \def\fsu{\mathfrak{su}}
\def\cT{\mathcal T}  
 \def\ttT{\mathtt{T}}
\def\ft{\mathfrak{t}}

\def\sU{\mathscr{U}}

 \def\cV{\mathcal V}

 \def\sW{\mathscr{W}}

\def\bx{\mathbf{x}}
\def\cY{\mathcal{Y}} 
\def\ttY{\mathtt{Y}} 

   \def\bZ{\mathbb Z}
 
\def\fz{\mathfrak{z}} 
 
\def\half{\tfrac{1}{2}}

\def\one{\mathbbm{1}}
\def\tand{\quad\hbox{and}\quad}
\def\del{\partial}
\def\dfn{\stackrel{\hbox{\tiny{dfn}}}{=}}
\def\sb{{\hbox{\tiny{$\bullet$}}}}

\def\inj{\hookrightarrow}
\def\sur{\twoheadrightarrow}

\def\op{\oplus}
\def\ot{\otimes}

\def\tw{\hbox{\small $\bigwedge$}}

\def\rtL{{\Lambda_\mathrm{rt}}}

\newcounter{numcnt}

\newcounter{cnt}

\newcounter{acnt}
\newenvironment{a_list}{ 
  \begin{list}{{(\alph{acnt})}}
   {\usecounter{acnt} \setlength{\itemsep}{3pt}
    \setlength{\leftmargin}{25pt} 
    \setlength{\labelwidth}{20pt}
    \setlength{\listparindent}{20pt} }
   }
   {\end{list}}
\newenvironment{a_list_emph}{ 
  \begin{list}{{\emph{(\alph{acnt})}}}
   {\usecounter{acnt} \setlength{\itemsep}{3pt}
    \setlength{\leftmargin}{25pt} 
    \setlength{\labelwidth}{20pt}
    \setlength{\listparindent}{20pt} }
   }
   {\end{list}}
\newenvironment{a_list_bold}{ 
  \begin{list}{{\bf{(\alph{acnt})}}}
   {\usecounter{acnt} \setlength{\itemsep}{3pt}
    \setlength{\leftmargin}{25pt} 
    \setlength{\labelwidth}{20pt}
    \setlength{\listparindent}{20pt} }
   }
   {\end{list}}
\newcounter{Acnt}

\newcounter{icnt}

\newcounter{Icnt}

\newcounter{exam_cnt}

\newcounter{mccnt}

\newenvironment{bcirclist}{ 
  \begin{list}{\boldmath$\circ$\unboldmath}
   {\usecounter{cnt} \setlength{\itemsep}{2pt}
    \setlength{\leftmargin}{15pt} \setlength{\labelwidth}{20pt}
    \setlength{\listparindent}{20pt} }
   }
   {\end{list}}


%% file: thms.tex
\newtheorem{corollary}[equation]{Corollary}
\newtheorem{lemma}[equation]{Lemma}
\newtheorem{proposition}[equation]{Proposition}
\newtheorem{theorem}[equation]{Theorem}

\theoremstyle{definition}

\newtheorem*{boldQ*}{Question}
\newtheorem*{boldP*}{Problem}

\theoremstyle{definition}

\theoremstyle{remark}
\newtheorem*{assume*}{Assume}
\newtheorem*{answer*}{Answer}

\newtheorem*{claim*}{Claim}

\newtheorem{definition}[equation]{Definition}
\newtheorem*{definition*}{Definition}
\newtheorem{example}[equation]{Example}
\newtheorem*{example*}{Example}
\newtheorem*{hint*}{Hint}
\newtheorem*{notation*}{Notation}
\newtheorem{remark}[equation]{Remark}
\newtheorem*{remark*}{Remark}
\newtheorem*{remarks*}{Remarks}
\newtheorem*{fact*}{Fact}
\newtheorem*{emphL*}{Lemma}

\newtheorem*{emphQ*}{Question}
\newtheorem*{emphA*}{Answer}


%% file: KR1.bbl
\def\cprime{$'$} \def\Dbar{\leavevmode\lower.6ex\hbox to 0pt{\hskip-.23ex
  \accent"16\hss}D}